\documentclass[a4paper,reqno,12pt]{amsart}
\usepackage[T1]{fontenc}
\usepackage[utf8]{inputenc}
\usepackage[english]{babel}

\usepackage{amsmath}
\usepackage{amsthm}
\usepackage{amssymb}
\usepackage{mathtools}
\usepackage{amsfonts}
\usepackage{esint}
\usepackage{quoting}
\usepackage{graphics}
\usepackage{booktabs}
\usepackage{bm}
\usepackage{bbm}
\usepackage{lmodern}
\usepackage{stmaryrd}
\usepackage{caption}
\usepackage{enumitem}			
\usepackage{tikz}
\usepackage{mathrsfs}
\usepackage{wasysym}
\usepackage[pdftex,colorlinks=true,allcolors=blue,breaklinks]{hyperref}
\usepackage{orcidlink}
\usepackage[mathscr]{euscript}

\usepackage[margin=1.1in]{geometry}

\quotingsetup{font=small}

\theoremstyle{plain}
\newtheorem{proposition}{Proposition}[section]

\theoremstyle{plain}
\newtheorem{theorem}{Theorem}[section]
\numberwithin{equation}{section}	

\theoremstyle{plain}
\newtheorem{lemma}[theorem]{Lemma}

\theoremstyle{plain}
\newtheorem{corollary}{Corollary}[theorem]

\theoremstyle{definition}

\theoremstyle{remark}
\newtheorem{remark}{Remark}[section]

\theoremstyle{definition}

\DeclarePairedDelimiter{\abs}{\lvert}{\rvert}
\DeclarePairedDelimiter{\norma}{\lVert}{\rVert}

\renewcommand{\phi}{\varphi}
\renewcommand{\epsilon}{\varepsilon}

\DeclareMathOperator{\sign}{sign}
\DeclareMathOperator{\defi}{def}

\newcommand{\numberset}{\mathbb}
\newcommand{\N}{\numberset{N}}

\newcommand{\R}{\numberset{R}}

\newcommand{\sfera}{\numberset{S}}		
\newcommand{\Haus}{\mathcal{H}}
\newcommand{\past}{p^\ast}
\newcommand{\loc}{{\rm loc}}

\def\Xint#1{\mathchoice
	{\XXint\displaystyle\textstyle{#1}}%
	{\XXint\textstyle\scriptstyle{#1}}%
	{\XXint\scriptstyle\scriptscriptstyle{#1}}%
	{\XXint\scriptscriptstyle\scriptscriptstyle{#1}}%
	\!\int}
\def\XXint#1#2#3{{\setbox0=\hbox{$#1{#2#3}{\int}$ }
		\vcenter{\hbox{$#2#3$ }}\kern-.6\wd0}}

\def\dashint{\Xint-}

\allowdisplaybreaks

\begin{document}
	
	\title[Approximate radial symmetry for~$p$-Laplace equations]{Approximate radial symmetry for~$p$-Laplace equations via the moving planes method}
	\author[Michele Gatti]{Michele Gatti \orcidlink{0009-0002-6686-9684}}
	\address[]{Michele Gatti. Dipartimento di Matematica ‘Federigo Enriques’, Università degli Studi di Milano, Via Cesare Saldini 50, 20133, Milan, Italy}
	\email{michele.gatti1@unimi.it}
	
	\subjclass[2020]{Primary 35B33, 35B35, 35J92; Secondary 35B51}
	\date{\today}
	\keywords{Moving planes method, quantitative estimates, quasilinear elliptic equations, stability}
	
	\begin{abstract}
		We investigate quasi-symmetry for small perturbations of the Gidas-Ni-Nirenberg problem involving the~$p$-Laplacian and for small perturbations the critical~$p$-Laplace equation for~$p>2$.
		
		To achieve these results, we provide a quantitative review of the work by Damascelli \& Sciunzi~\cite{ds-har} concerning the weak Harnack comparison inequality and the local boundedness comparison inequality. Moreover, we prove a comparison principle for small domains.
	\end{abstract}
	
	\maketitle
	
	
	\section{Introduction}
	
	The classification of positive solutions to semilinear and quasilinear equations
	is a classical topic in PDEs and geometric analysis, with applications in several contexts~\cite{berg,tc-schrod,sc-stellar,fink-llev,mss,nehari}. Among the key contributions in this field are the seminal papers of Gidas, Ni \& Nirenberg~\cite{gnn-srp,gnn}, which had a significant impact through both the techniques they introduced and the results they established.
	
	In recent years, there has been increasing interest in their quantitative counterparts and, more generally, in the quantitative analysis of such problems. Despite the case of semilinear equations seems to be extensively explored~\cite{ccg,cicopepo,cfm,dsw,fg,ross}, only a few contributions are available in the quasilinear case~\cite{cirli,dpsv}.
	
	The main goal of this paper is to investigate quasi-symmetry results for a Gidas-Ni-Nirenberg-type problem on the unit ball involving the~$p$-Laplacian and a small perturbation of the critical~$p$-Laplace equation in~$\R^n$ via the moving plane method.
	
	In order to accurately present the literature, clarify our motivations, and outline our contributions, the introduction is divided into two different subsections, each corresponding to one of the problems addressed in this work.
	
	
	\subsection{Quantitative result in the unit ball}
	Let~$B_1 \subseteq \R^n$ be the unit ball,~$n\geq 2$, and~$p \in \left(1,+\infty\right)$. Let us consider~$u \in C^1(\overline{B_1})$ a weak solution of the problem
	\begin{equation}
		\label{eq:mainprob}
		\begin{cases}
			\begin{aligned}
				& -\Delta_p u = \kappa f(u) && \mbox{in } B_1, \\
				& u > 0 					&& \mbox{in } B_1, \\
				& u=0						&& \mbox{on } \partial B_1,
			\end{aligned}
		\end{cases}
	\end{equation}
	where
	\begin{equation}
		\label{eq:mainhyp-f}
		\begin{split}
			&f: [0,+\infty) \to [0,+\infty) \mbox{ is such that } f\in
			C^{0,1}_{\loc}([0,+\infty)) \mbox{ and }\\
			&\quad f(u)>0 \mbox{ for } u \in (0,+\infty),
		\end{split}
	\end{equation}
	and
	\begin{equation}
		\label{eq:mainhyp-k}
		\kappa: B_1 \to (0,+\infty) \mbox{ with } \kappa \in L^\infty(B_1).
	\end{equation}
	
	The symmetry and monotonicity of solutions to~\eqref{eq:mainprob} when~$\kappa \equiv 1$ has been thoroughly investigated. The case of~$p=2$ goes back to the aforementioned paper of Gidas, Ni \& Nirenberg~\cite{gnn-srp}. More recently, Damascelli \& Pacella~\cite{dam-pacella} studied the problem for~$1<p<2$, exploiting the moving plane method alongside topological arguments. Finally, for~$p>2$, Damascelli \& Sciunzi~\cite{ds-poinc} extended this result by combining the moving plane method with certain regularity properties established in the same paper.
	
	Additionally, analogous results, allowing even discontinuous nonlinearities, have been proven by Lions~\cite{lions}, Kesavan \& Pacella~\cite{kpac}, and Serra~\cite{serra} through integral methods boiling down on the isoperimetric inequality and Pohozaev’s identity. \newline
	
	From the quantitative point of view, Rosset~\cite{ross} and Ciraolo, Cozzi, Perugini \& Pollastro~\cite{cicopepo} investigated the stability of~\eqref{eq:mainprob} when~$p=2$. More precisely, in Theorem~1.2 of~\cite{cicopepo}, they proved the following result.
	
	\begin{theorem}
		\label{th:cicopepo}
		Let~$f: [0,+\infty) \to \R$ be a non-negative locally Lipschitz continuous function and~$\kappa \in C^1(\overline{B_1})$ be a non-negative function. Let~$u \in C^2(B_1) \cap C(\overline{B_1})$ be a classical solution to~\eqref{eq:mainprob} for~$p=2$ satisfying
		\begin{equation*}
			\frac{1}{C_0} \leq \norma{u}_{L^\infty(B_1)} \leq C_0,
		\end{equation*}
		for some~$C_0 \geq 1$. Then
		\begin{equation*}
			\abs{u(x)-u(y)} \leq C \defi(\kappa)^\alpha \quad \mbox{for all } x,y \in B_1 \; \text{such that} \; \abs{x}=\abs{y}
		\end{equation*}
		and
		\begin{equation*}
			\partial_r u (x) < C \defi(\kappa)^\alpha \quad \mbox{for all } x \in B_1 \setminus \{0\},
		\end{equation*}
		for some constants~$\alpha \in (0,1]$ and~$C>0$, depending on~$n$,~$\norma{f}_{C^{0,1}([0,C_0])}$,~$\norma{\kappa}_{L^\infty(B_1)}$, and~$C_0$, and
		\begin{equation*}
			\defi(\kappa) \coloneqq \norma{\nabla^T \kappa}_{L^\infty(B_1)} + \norma{\partial_r^+ \kappa}_{L^\infty(B_1)}.
		\end{equation*}
		Here,~$\partial_r^+$ denotes the positive part of the radial derivative and~$\nabla^T$ is the angular gradient.
	\end{theorem}
	
	In order to prove Theorem~\ref{th:cicopepo}, the authors performed a quantitative version of the moving plane method, strongly exploiting the linearity of Laplacian and the equivalence between maximum and comparison principles for~$p=2$.
	
	The key ingredients of their proof are the Alexandroff-Bakelman-Pucci inequality, which is used to derive estimates for the difference between two solutions in small domains, and the weak Harnack inequality, used to propagate positivity of a suitable function. See Steps~2 and~4 of the proof of Theorem~1.2 in~\cite{cicopepo}, respectively, for more precise details.
	
	From a different perspective, Dipierro, Gon\c{c}alves da Silva, Poggesi \& Valdinoci~\cite{dpsv}, as well as Ciraolo \& Li~\cite{cirli}, proved a quantitative Gidas-Ni-Nirenberg-type result for the~$p$-Laplace operator via integral identities and using the quantitative isoperimetric inequality. However, their works focus on the stability of solutions under domain perturbations and quantify the proximity of the domain~$\Omega$ to a ball. \newline
	
	In view of the aforementioned symmetry results, it is natural to ask whether some version of Theorem~\ref{th:cicopepo} still holds also for~$p \neq 2$. Trying to answer this question, in the first part of the present paper, we investigate the stability of problem~\eqref{eq:mainprob} in the sense of Theorem~\ref{th:cicopepo} for~$p>2$.
	
	Our general strategy is based upon the one developed in~\cite{cicopepo}. However, the primary challenges when dealing with the case~$p \neq 2$ arise from the nonlinear nature of the~$p$-Laplace operator and the lack of equivalence between the maximum and comparison principles. Furthermore, deriving a useful equation for the difference of two solutions to~\eqref{eq:mainprob} is not straightforward. As a result, we must first develop the tools needed to handle a general~$p \in (2,+\infty)$, which may also be of independent interest.
	
	At this stage, we mention that a significant portion of the present work will be devoted to a quantitative revision of the weak Harnack comparison inequality and the local boundedness comparison inequality established by Damascelli \& Sciunzi~\cite{ds-har}. This revision is motivated by the observation that the constants in the original results depend implicitly on the radius of the ball on which the inequalities are applied. Consequently, when propagating, for instance, the weak Harnack inequality along a chain, the constants are not under control.
	
	Our first main result is the following.
	
	\begin{theorem}
		\label{th:mainth}
		Assume that~$f: [0,+\infty) \to [0,+\infty)$ satisfies~\eqref{eq:mainhyp-f} and that~$\kappa: B_1 \to (0,+\infty)$ satisfies~\eqref{eq:mainhyp-k}. Suppose that~$p \in \left(2,+\infty\right)$ and let~$u \in C^1(\overline{B_1})$ be a non-negative weak solution to~\eqref{eq:mainprob} such that
		\begin{equation}
			\label{eq:boundu}
			\frac{1}{C_0} \leq \norma{u}_{L^\infty(B_1)} \leq C_0,
		\end{equation}
		for some~$C_0 \geq 1$. Assume furthermore that, for some constant~$F>0$, we have
		\begin{equation}
			\label{eq:ipotesif}
			f(u) \leq F u^{p-1}.
		\end{equation}
		Then, there exist a large constant~$C \geq 1$ and a constant~$\alpha \in (0,1)$, such that, if~$\defi(\kappa) \in \left(0,\frac{1}{C}\right)$,
		\begin{equation}
			\label{eq:qsym-toprove}
			\abs*{u(x)-u(y)} \leq C \, \abs*{\log\left(C \defi(\kappa)\right)}^{-\alpha} \quad \mbox{for all } x,y \in B_1 \mbox{ such that } \abs{x}=\abs{y},
		\end{equation}
		where
		\begin{equation}
			\label{eq:deficit}
			\defi(\kappa) \coloneqq \underset{x \in B_1}{\mathrm{osc}} \kappa(x).
		\end{equation}
		In particular, the constants~$C$ and~$\alpha$ depend only on~$n$,~$p$, $C_0$,~$\norma{f}_{C^{0,1}([0,C_0])}$,~$f$,~$\norma{\kappa}_{L^\infty(B_1)}$, and~$F$.
	\end{theorem}
	
	Theorem~\ref{th:mainth} provides a quantitative quasi-symmetry result, with a logarithmic-type dependence, in terms of the deficit~\eqref{eq:deficit}. In particular, when~$\defi(\kappa)=0$ and~$\kappa$ is sufficiently regular, it is constant, and the radial symmetry of non-negative solutions to~\eqref{eq:mainprob} in well-established. Moreover, the right-hand side of~\eqref{eq:qsym-toprove} vanishes when~$\defi(\kappa) \to 0$, consistently with the symmetry of~$u$.
	
	We also point out that the logarithmic-type dependence in~\eqref{eq:qsym-toprove}, in contrast with a power-like dependence of Theorem~\ref{th:cicopepo}, arises due to technical challenges related to the weak Harnack inequality. See Remark~\ref{rem:wh-logtype} below for a more precise explanation.
	
	Additionally, we emphasize that the dependence on~$f$ in Theorem~\ref{th:mainth} is clear. Notably, whenever~$f \geq f_0$ in~$[0,C_0]$ for some constant~$f_0 >0$, then the dependence is through this constant. For further details, see Remark~\ref{rem:dep-f} in Appendix~\ref{ap:proofsum}.
	
	Finally, it is worth mentioning that Damascelli \& Pardo~\cite{dam-par} derived conditions on~$f$ that ensure an a priori estimate of the type given in the right-most inequality of~\eqref{eq:boundu} for solutions to~\eqref{eq:mainprob}.
	
	\begin{remark}
	\label{rem:no-quasimon}
		Upon comparing Theorem~\ref{th:cicopepo} with Theorem~\ref{th:mainth}, it appears that the latter lacks quasi-monotonicity estimates for solutions. In our case, obtaining such an estimate seems difficult using arguments analogous to those exploited in~\cite{cicopepo} due to technical challenges in applying barrier methods -- see Step~6 of the proof of Theorem~1.2 in~\cite{cicopepo}. A more suitable approach might involve the linearized operator, as it is known that derivatives satisfy the linearized equation -- see Lemma~2.3 in~\cite{ds-poinc}. However, these would require a comparison principle for the linearized operator, possibly involving two distinct solutions, and, to the best of our knowledge, such a result is not currently available in the literature.
	\end{remark}
	
	
	\subsection{Quantitative result in the whole space}
	
	For~$2<p<n$, let~$u \in \mathcal{D}^{1,p}(\R^n)$ be a non-negative, non-trivial weak solution to
	\begin{equation}
		\label{eq:mainprob-crit}
		\Delta_p u + \kappa\!\left(x\right) u^{\past-1} = 0 \quad\mbox{in } \R^n,
	\end{equation}
	where~$\kappa: \R^n \to (0,+\infty)$ satisfies~$\kappa \in L^\infty(\R^n)$ and is bounded below, that is
	\begin{equation}
		\label{eq:k-below-sp}
		\kappa \geq \underline{\kappa} >0 \quad \mbox{a.e.\ in } \R^n.
	\end{equation}
	Recall that here and in the following
	\begin{equation*}
		\mathcal{D}^{1,p}(\R^n) \coloneqq \left\{u \in L^{\past}\!(\R^n) \,\big\lvert\, \nabla u \in L^p(\R^n) \right\} \!,
	\end{equation*}
	with~$\past$ being the critical exponent for the Sobolev embedding.
	
	The classification of non-negative solutions to~\eqref{eq:mainprob-crit} in the case~$\kappa \equiv 1$, namely
	\begin{equation}
		\label{eq:crit}
		\Delta_p u + u^{\past-1} = 0
	\end{equation}
	has a long history. It is worth recalling that the family of functions
	\begin{equation*}
		U_p[z,\lambda] (x) \coloneqq \left( \frac{\lambda^\frac{1}{p-1} \, n^\frac{1}{p} \left(\frac{n-p}{p-1}\right)^{\!\!\frac{p-1}{p}}}{\lambda^\frac{p}{p-1}+\abs*{x-z}^\frac{p}{p-1}} \right)^{\!\!\frac{n-p}{p}} \!,
	\end{equation*}
	where~$z \in \R^n$ is the center of the function and~$\lambda>0$ its scaling factor, is well-known to be a family of explicit solutions. These functions are known as \emph{Talenti bubbles}. Furthermore, thanks to Talenti's work~\cite{tal}, it is well-established that these functions minimize the Sobolev quotient
	\begin{equation*}
		S_p = S \coloneqq \min_{u \in \mathcal{D}^{1,p}(\R^n) \setminus \{ 0 \}} \frac{ \norma*{\nabla u}_{L^p(\R^n)}}{\norma*{u}_{L^{\past}\!(\R^n)}}.
	\end{equation*} 
	
	When~$p=2$ the classification of the non-negative solutions to~\eqref{eq:crit} can be traced back to the groundbreaking paper of Gidas, Ni \& Nirenberg~\cite{gnn}, where they require a suitable decay at infinity. This condition was later removed by Caffarelli, Gidas \& Spruck~\cite{cgs} and Chen \& Li~\cite{cl}, leading to the conclusion that Talenti bubbles are the only non-negative solutions to~\eqref{eq:crit}.
	
	Concerning the~$p$-Laplace equation, the classification result for energy solutions, i.e.\ functions in the energy space~$\mathcal{D}^{1,p}(\R^n)$, for~$1<p<2$ was initially obtained by Damascelli, Merch\'an, Montoro \& Sciunzi~\cite{dm} under the additional assumption of local Lipschitz continuity of the nonlinearity, namely for~$p \geq \frac{2n}{n+2}$. Later, V\'etois~\cite{vet} removed this restriction extending the result to the whole range~$1<p<2$. Finally, Sciunzi~\cite{sciu} proved that, for~$2<p<n$, the only non-negative solutions to~\eqref{eq:crit} are the bubbles, thereby providing a complete classification of energy solutions. \newline
	
	To the best of our knowledge, from the quantitative perspective, only the case~$p=2$ has been extensively studied and has received significant attention in recent years.
	
	Looking for a quantitative counterpart of Struwe's stability result~\cite{struwe}, Ciraolo, Figalli \& Maggi~\cite{cfm} showed that if~$u \in \mathcal{D}^{1,2}(\R^n)$ has nearly the same energy as a Talenti bubble, then it is quantitatively close to a single bubble in the space~$\mathcal{D}^{1,2}(\R^n)$. This result was later extended by Figalli \& Glaudo~\cite{fg} and Deng, Sun \& Wei~\cite{dsw} when the energy of the function~$u$ is approximately that of~$m \in \N$ Talenti bubbles. In particular, in~\cite{cfm} they define the reference constant
	\begin{equation}
		\label{eq:kappa0-def}
		\kappa_0=\kappa_0(u) \coloneqq \frac{\int_{\R^n} \kappa(x) u^{\past} dx}{\int_{\R^n} u^{\past} dx} = \frac{\int_{\R^n} \,\abs{\nabla u}^p \, dx}{\int_{\R^n} u^{\past} dx} ,
	\end{equation}
	and measure the proximity of~$\kappa$ to~$\kappa_0$ in terms of the deficit
	\begin{equation}
		\label{eq:def_cfm}
		\defi(u,\kappa) \coloneqq \norma*{\left(\kappa-\kappa_0\right)u^{\past-1}}_{L^{(\past)'}\!(\R^n)}.
	\end{equation}
	As a consequence of  Ciraolo, Figalli \& Maggi's theorem, if a function~$u \in \mathcal{D}^{1,2}(\R^n)$ has nearly the energy of a bubble and almost solves~\eqref{eq:crit}, it is both quasi-symmetric and quasi-monotone decreasing in terms of~\eqref{eq:def_cfm}. However, this result and the methods therein seem to apply only to equations of the type~\eqref{eq:mainprob-crit} with~$p=2$. In fact, the technique developed in~\cite{cfm,dsw,fg} strongly relies on the linearity of the Laplacian and the Hilbert structure of~$\mathcal{D}^{1,2}(\R^n)$, as well as the knowledge of the explicit family of solutions to~\eqref{eq:crit} and their spectral properties.
	
	For these reasons, when examining the proximity of non-energy classical solutions to~\eqref{eq:mainprob-crit} and of solutions to more general semilinear equations involving the Laplacian to the radial configuration, Ciraolo, Cozzi \& Gatti~\cite{ccg} adopted different approaches. \newline
	
	In the second part of the present work, we aim to show that the technique developed in~\cite{ccg} for handling energy solutions of a semilinear equation can be extended to~\eqref{eq:mainprob-crit}, provided some necessary tools are available.
	
	The technique in~\cite{ccg} boils down to the quantitative analysis of the moving planes method and heavily relies on the weak Harnack inequality together with the local boundedness inequality -- see~\cite[Theorem~8.17, Theorem~8.18, and Theorem~8.25]{gt} for the case of linear operators.
	
	As mentioned above, the equivalence between maximum and comparison principles notably simplifies the problem in the case of the Laplacian. However, in the general scenario, where such equivalence is absent, we must first develop the appropriate tools before proving the main result. Some of these tools are also required to address problem~\eqref{eq:mainprob} in the unit ball. \newline
	
	We now introduce our main assumptions and the almost-symmetry result related to the perturbed critical equation~\eqref{eq:mainprob-crit}.
	
	To our purposes, we impose a global decay assumption, namely that there exists a constant~$C_0\geq 1$ such that
	\begin{equation}
		\label{eq:decadimento}
		u(x) \leq \frac{C_0}{1+\abs{x}^\frac{n-p}{p-1}} \quad \mbox{for a.e.~} x \in \R^n
	\end{equation}
	and a lower bound for the gradient, namely that there exist constants~$c_1>0$, $R_0 \geq 1$ such that
	\begin{equation}
		\label{eq:bb-grad}
		\abs*{\nabla u (x)} \geq c_1 \abs{x}^\frac{1-n}{p-1} \quad \mbox{for a.e.~} \abs{x} \geq R_0. 
	\end{equation}
	
	We emphasize that both these hypotheses are quite natural. In fact, Theorem~1.1 of~\cite{vet} establishes that any solution to~\eqref{eq:mainprob-crit}, with~$\kappa \in L^\infty(\R^n)$, exhibits a decay of the form~\eqref{eq:decadimento}, with constants depending on~$u$. For additional discussion on this assumption and its role in the proof, see Section~1.2 in~\cite{ccg}. On the other hand, Sciunzi~\cite{sciu} proved the validity of~\eqref{eq:bb-grad} for non-negative solutions to~\eqref{eq:crit} by employing a blow-up argument. In this case, as well, the constants appearing in the estimate of Theorem~2.2 in~\cite{sciu} depend on~$u$. Since the same proof applies also for~\eqref{eq:mainprob-crit}, here, we fix the constants for quantitative purposes.
	
	Our second main result is the following.
	
	\begin{theorem}
		\label{th:plap-spazio}
		Let~$n \geq 3$ be an integer,~$2 < p < n$, and~$\kappa: \R^n \to (0,+\infty)$ be such that~$\kappa \in L^\infty(\R^n)$ and~\eqref{eq:k-below-sp} holds true. Moreover, let~$u \in \mathcal{D}^{1,p}(\R^n)$ be a non-negative, non-trivial weak solution to~\eqref{eq:mainprob-crit} satisfying~\eqref{eq:decadimento}-\eqref{eq:bb-grad}.
		
		Then, there exist a point~$\mathcal{O} \in \R^n$, a large constant~$C>0$, and a constant~$\vartheta>0$, such that, if~$\defi(u,\kappa) \in (0,1)$, then
		\begin{equation*}
			\abs*{u(x)-u(y)} \leq C \,\abs{\log\defi(u,\kappa)}^{-\vartheta}
		\end{equation*}
		for every~$x,y \in \R^n$ satisfying~$\abs*{x-\mathcal{O}}=\abs*{y-\mathcal{O}}$. Moreover, if~$u_\Theta$ denotes any rotation of~$u$ centered at~$\mathcal{O}$, we also have
		\begin{equation*}
			\norma*{u-u_\Theta}_{\mathcal{D}^{1,p}(\R^n)} 
			\leq C \,\abs{\log\defi(u,\kappa)}^{-\vartheta}.
		\end{equation*}
		The constant~$C$ depends only on~$n$,~$p$,~$C_0$, $c_1$,~$\norma{\kappa}_{L^\infty(\R^n)}$,~$\underline{\kappa}$, and~$R_0$, whereas~$\vartheta$ depends only on~$n$ and~$p$.
	\end{theorem}

	A remark analogous to Remark~\ref{rem:no-quasimon} also applies when comparing Theorem~\ref{th:plap-spazio} with Theorem~1.1 in~\cite{ccg}.
	
	\begin{remark}
		We note that most of the results in the present paper hold for the case~$p>2$. We believe that, by further refining the proofs in~\cite{ds-har}, as done below for~$p>2$, one can extend these results also to the range
		\begin{equation}
		\label{eq:range-p-ext}
			p \in \left(\frac{2n+2}{n+2},2\right)\!.
		\end{equation}
		This extension requires a more detailed technical analysis of the proofs in~\cite{ds-har}. However, to keep the discussion concise, we refrain from pursuing it in the present work.
		
		We also note that the range~\eqref{eq:range-p-ext} is optimal for the technique developed in~\cite{ds-har} and further analyzed here, due to integrability constraints on~$\abs*{\nabla u}^{p-2}$.
		
		Finally, we emphasize that the restriction~\eqref{eq:range-p-ext} is necessary only for the weak Harnack inequality and the local boundedness comparison inequality to hold, suggesting that it may be merely a technical limitation in obtaining quasi-symmetry estimates for solutions to~\eqref{eq:mainprob} and~\eqref{eq:mainprob-crit}.
	\end{remark}
	
	
	\section{Technical tools \& Comparison inequalities}
	\label{sec:tectool}
	
	Suppose that~$\Omega \subseteq \R^n$ is an open set. In this section we are mainly concerned with functions~$u_1,u_2\in C^1(\Omega)$ which are weak subsolution and supersolution to
	\begin{equation}
		\label{eq:eqforu}
		-\Delta_p u_i + c \, u_i = g_i(x,u_i) \quad \mbox{in } \Omega
	\end{equation}
	for~$i=1,2$, where~$g_i: \Omega \times \R \to \R$ and~$c: \Omega \to \R$ with~$c \in L^\infty(\Omega)$.
	
	Recall that a function~$u_1\in C^1(\Omega)$ is a weak subsolution to~\eqref{eq:eqforu} if
	\begin{equation*}
		\int_{\Omega} \,\abs*{\nabla u_1}^{p-2} \left\langle \nabla u_1 , \nabla \phi \right\rangle + c \, u_1 \phi \, dx \leq \int_{\Omega} g_1(x,u_1) \phi \, dx \quad \mbox{for all } \phi \in C^\infty_c(\Omega) \mbox{ with } \phi \geq 0
	\end{equation*}
	and a function~$u_2\in C^1(\Omega)$ is a weak supersolution to~\eqref{eq:eqforu} if
	\begin{equation*}
		\int_{\Omega} \,\abs*{\nabla u_2}^{p-2} \left\langle \nabla u_2 , \nabla \phi \right\rangle + c \, u_2 \phi \, dx \geq \int_{\Omega} g_2(x,u_2) \phi \, dx \quad \mbox{for all } \phi \in C^\infty_c(\Omega) \mbox{ with } \phi \geq 0.
	\end{equation*}
	
	\subsection{Fundamental inequalities.} In the following, we will repeatedly use some well-known inequalities, which we recall here for convenience. We refer to Lemma~2.1 in~\cite{dam-comp} for their proof.
	For any $\eta,\eta' \in \R^n$ with $\abs{\eta}+\abs{\eta'}>0$ we have
	\begin{align}
		\label{eq:disp-basso}
		\left\langle \abs{\eta}^{p-2}\,\eta - \abs{\eta'}^{p-2}\,\eta', \eta-\eta' \right\rangle &\geq c(p) \left(\abs{\eta}+\abs{\eta'}\right)^{p-2} \abs*{\eta-\eta'}^2, \\
		\label{eq:disp-alto}
		\abs*{\abs{\eta}^{p-2}\,\eta - \abs{\eta'}^{p-2}\,\eta'} &\leq C(p) \left(\abs{\eta}+\abs{\eta'}\right)^{p-2} \abs{\eta-\eta'}.
	\end{align}
	Moreover, for any $\eta,\eta' \in \R^n$ we have
	\begin{align}
		\label{eq:disp-basso-2}
		\left\langle \abs{\eta}^{p-2}\,\eta - \abs{\eta'}^{p-2}\,\eta', \eta-\eta' \right\rangle &\geq \hat{c}(p) \abs*{\eta-\eta'}^p \quad\text{if} \;p\geq 2, \\
		\notag
		\abs*{\abs{\eta}^{p-2}\,\eta - \abs{\eta'}^{p-2}\,\eta'} &\leq \widehat{C}(p) \abs{\eta-\eta'}^{p-1} \quad\text{if} \;1<p\leq 2.
	\end{align}
	
	\subsection{Functional setting and known results.} We introduce here the functional setting. Following~\cite{stamp-mu}, for~$\rho \in L^1(\Omega)$ and~$1\leq q <+\infty$, we define the Banach space~$H^{1,q}_\rho(\Omega)$ as the completion of~$C^1(\overline{\Omega})$ with respect to the norm
	\begin{equation*}
		\norma{\cdot}_{H^{1,q}_\rho(\Omega)} \coloneqq \norma{\cdot}_{L^q(\Omega)}+\norma{\nabla\cdot}_{L^q_\rho(\Omega)}, \quad \norma{\nabla\cdot}^q_{L^q_\rho(\Omega)} = \int_\Omega \rho \,\abs{\nabla\cdot}^q \, dx.
	\end{equation*}
	Moreover, we define~$H^{1,q}_{0,\rho}(\Omega)$ as the closure of~$C^1_c(\Omega)$ in~$H^{1,q}_\rho(\Omega)$.
	
	Observe that if~$\rho \in L^\infty(\Omega)$, then~$W^{1,q}(\Omega)$ is continuously embedded in~$H^{1,q}_\rho(\Omega)$. Additionally, if~$u \in W^{1,q}(\Omega)$ for~$q\geq2$ and~$\rho = \abs{\nabla u}^{q-2}$, H\"older inequality implies that~$W^{1,q}(\Omega)$ is continuously embedded in~$H^{1,2}_\rho(\Omega)$. \newline
	
	For reference, we summarize here some known results from Theorem~8, Remark~4, and Corollary~2 of~\cite{fms}, originally proven in Theorem~3.1 and Theorem~3.2 of~\cite{ds-poinc}.
	
	\begin{theorem}[Weighted Sobolev inequality]
		\label{th:w-sobolev}
		Let~$\Omega \subseteq \R^n$ be a bounded domain and~$\rho$ a non-negative weight function such that
		\begin{equation}
			\label{eq:condizione-peso}
			\int_\Omega \frac{1}{\rho^t \abs{x-y}^\gamma} \, dy \leq C^\ast,
		\end{equation}
		with~$C^\ast$ independent of~$x\in\Omega$,~$t=\frac{p-1}{p-2}r$,~$r\in\left(\frac{p-2}{p-1},1\right)$, and~$\gamma<n-2$ or~$\gamma=0$ if~$n=2$.
		
		Then, there exists a constant~$C_\mathcal{S}>0$ such that
		\begin{equation}
			\label{eq:sobolev-pesata}
			\norma{u}_{L^q(\Omega)} \leq C_\mathcal{S} \norma{\nabla u}_{L^2_\rho(\Omega)}
		\end{equation}
		for any~$u \in H^{1,2}_{0,\rho}(\Omega)$ and any~$1\leq q < 2^\ast(t)$, where
		\begin{equation*}
			\frac{1}{2^\ast(t)} = \frac{1}{2} - \frac{1}{n} + \frac{1}{t} \left(  \frac{1}{2} - \frac{\gamma}{2n} \right)
		\end{equation*}
		and
		\begin{gather*}
			C_\mathcal{S} = C_\mathcal{S}(\Omega) \coloneqq \widehat{C} \left(C^\ast\right)^\frac{1}{2t}\left(C_M\right)^\frac{1}{(2t)'}, \\
			\widehat{C}\coloneqq\frac{1}{n\abs{B_1}}, \quad C_M\coloneqq \left(\frac{1-\delta}{\mu-\delta}\right)^{\!\!1-\delta} \abs{B_1}^{\mu-\delta} \,\abs{\Omega}^{\mu-\delta},
		\end{gather*}
		where~$\mu$ and~$\delta$ are positive constants depending only on~$n$,~$q$,~$t$, and~$\gamma$.
		
		In addition, if~$\Omega$ is convex,~$S \subseteq \Omega$ is any measurable subset with positive measure, and~$u \in H^{1,2}_{\rho}(\Omega)$ with $\int_{S} u \, dx = 0$, then~\eqref{eq:sobolev-pesata} holds with
		\begin{equation*}
			\widehat{C} \coloneqq \frac{\mathrm{diam}(\Omega)^n}{n\abs*{S}}.
		\end{equation*}
	\end{theorem}
	
	\begin{remark}
		Note that the largest value of~$2^\ast(t)$ is attained for the limiting values of~$t$ and~$\gamma$. Therefore,~\eqref{eq:sobolev-pesata} holds true for any~$1\leq q <2_M$ where
		\begin{equation*}
			\label{eq:def-2M}
			\frac{1}{2_M} = \frac{1}{2} - \frac{1}{n} + \frac{p-2}{p-1} \frac{1}{n}.
		\end{equation*}
		Moreover, it follows that~$2_M>2$.
	\end{remark}
	
	\begin{corollary}[Weighted Poincaré inequality]
		\label{cor:w-poinc}
		Let~$\Omega \subseteq \R^n$ be a bounded domain,~$\rho$ a non-negative weight function fulfilling~\eqref{eq:condizione-peso}, and~$u \in H^{1,2}_{0,\rho}(\Omega)$. Then, we have
		\begin{equation}
			\label{eq:poincare-pesata}
			\int_\Omega u^2 \, dx \leq \mathcal{C}_P(\Omega)\, C_\mathcal{S}^2 \int_\Omega \rho \,\abs{\nabla u}^2 \, dx
		\end{equation}
		with~$\mathcal{C}_P(\Omega) \to 0$ as~$\abs{\Omega} \to 0$. Moreover, given any~$\theta \in (0,1)$, we can assume that
		\begin{equation*}
			\mathcal{C}_P(\Omega) \leq \abs{\Omega}^\frac{2\theta}{(p-1)n}.
		\end{equation*}
		In addition, if~$\Omega$ is convex and~$u \in H^{1,2}_{\rho}(\Omega)$ vanishes almost everywhere on a measurable subset with positive measure~$S \subseteq \Omega$, then~\eqref{eq:poincare-pesata} holds with~$\Omega \setminus S$ in place of~$\Omega$.
	\end{corollary}
	
	In order to exploit~\eqref{eq:sobolev-pesata} and~\eqref{eq:poincare-pesata}, when~$p>2$, we need to suppose that
	\begin{equation}
		\label{eq:assunzione-pesi}
		\mbox{at least one of } \abs{\nabla u_1}^{p-2} \mbox{ and }\abs{\nabla u_2}^{p-2} \mbox{ fulfills~\eqref{eq:condizione-peso},}
	\end{equation}
	for some admissible values of~$t$ and~$\gamma$. Moreover, we assume a universal bound on the gradient, that is
	\begin{equation}
		\label{eq:bound-gradiente}
		\norma{\nabla u_1}_{L^\infty(\Omega)} \leq C_1, \quad \norma{\nabla u_2}_{L^\infty(\Omega)} \leq C_1
	\end{equation}
	for some constant~$C_1>0$.
	
	We will later verify that, for our problems~\eqref{eq:mainprob} and~\eqref{eq:mainprob-crit}, the conditions~\eqref{eq:assunzione-pesi} and~\eqref{eq:bound-gradiente} are satisfied by the appropriate functions.
	
	\subsection{Comparison principle for small domains.}
	
	First, we address the issue of replacing Alexandroff-Bakelman-Pucci estimate in~\cite{cicopepo} with an appropriate tool for our purposes. To this end, we establish a comparison principle for small domains for weak subsolutions and supersolutions to~\eqref{eq:eqforu}, involving the~$L^\infty$-norm of the right-hand side.
	
	Our proof is partially inspired by Lemma~6.2.1 in~\cite{ps} and is based on an iterative technique. The result is proved for~$p>1$. Furthermore, when~$p>2$, we will extensively use the weighted Sobolev inequality and the weighted Poincaré inequality of Theorem~\ref{th:w-sobolev} and Corollary~\ref{cor:w-poinc}.
	
	\begin{proposition}
		\label{prop:wcp}
		Let~$\Omega \subseteq \R^n$ be a bounded domain, and let~$u_1,u_2\in C^1(\overline{\Omega})$ be a weak subsolution and a weak supersolution to~\eqref{eq:eqforu}, respectively. Moreover, assume that~$g_i(\cdot,u_i) \in L^\infty(\Omega)$ for~$i=1,2$, that~$c \in L^\infty(\Omega)$, and that~$u_1 \leq u_2$ on~$\partial\Omega$.
		
		For~$1<p<2$, assume that~\eqref{eq:bound-gradiente} holds. Then, there exists
		a small~$\delta \in (0,1)$ such that
		\begin{equation*}
			\norma{(u_1-u_2)_{+}}_{L^\infty(\Omega)} \leq \mathscr{K} \,\norma{g_1(\cdot,u_1)-g_2(\cdot,u_1)}_{L^\infty(\Omega)}
		\end{equation*}
		for some constant~$\mathscr{K} \geq 1$, provided that~$\abs{\Omega} \leq \delta$. Here,~$\delta$ depends only on~$n$,~$p$,~$\norma*{c}_{L^\infty(\Omega)}$, while~$\mathscr{K}$ depends only on~$n$,~$p$,~$L$, and~$C_1$.
		
		For~$p>2$, assume that~\eqref{eq:assunzione-pesi} holds. Then, the same conclusion follows with~$\delta \in (0,1)$ and~$\mathscr{K} \geq 1$, now depending only on~$n$,~$p$,~$\norma*{c}_{L^\infty(\Omega)}$, and an upper bound for~$C_\mathcal{S}$.
	\end{proposition}
	\begin{proof}
		Set~$\mathsf{k}\coloneqq\norma{g_1(\cdot,u_1)-g_2(\cdot,u_2)}_{L^\infty(\Omega)}$ and define
		\begin{equation*}
			w\coloneqq(u_1-u_2)_{+}+\mathsf{k}.
		\end{equation*}
		Clearly,~$w \geq \mathsf{k}$ in~$\Omega$ and~$w = \mathsf{k}$ on $\partial\Omega$ with~$w \in W^{1,2}(\Omega) \cap L^\infty(\Omega)$.
		
		Fix~$\mathsf{k}<l<m$ and define, as in Lemma~6.2.1 of~\cite{ps}, the functions
		\begin{equation*}
			\label{eq:defpsi}
			\psi(t)\coloneqq\frac{r^q}{q}
			\begin{cases}
				\begin{aligned}
					& 0 						&& \mbox{if } t\leq l, \\
					& t^q-l^q 					&& \mbox{if } l<t<m, \\
					& qm^{q-1}t-(q-1)m^q-l^q	&& \mbox{if } t\geq m;
				\end{aligned}
			\end{cases}
		\end{equation*}
		\begin{equation}
			\label{eq:defv}
			v(t)\coloneqq
			\begin{cases}
				\begin{aligned}
					& l^r					&& \mbox{if } t\leq l, \\
					& t^r					&& \mbox{if } l<t<m, \\
					& rm^{r-1}t-(r-1)m^r	&& \mbox{if } t\geq m,
				\end{aligned}
			\end{cases}
		\end{equation}
		where~$q,r \in \R$ with~$q\geq 1$, and~$r$ is determined by the relation
		\begin{equation}
			\label{eq:rel_qr}
			q-1=2(r-1),
		\end{equation}
		so that~$r\geq1$. Moreover, observe that~$\psi$ and~$v$ are non-negative, convex, monotone increasing, piecewise smooth except for corners at~$t=l$, and linear when~$t \geq m$.
		
		Now, we test~\eqref{eq:eqforu} with~$\psi(w)$. Note that~$\psi(w)$ vanishes in a neighborhood of~$\partial\Omega$. Additionally, when~$w \leq l$, we have~$\psi(w)=0$, so the integrals need to be evaluated only on the set~$\Omega_+ \coloneqq \{x \in \Omega \,\lvert\, w(x)>l\}$, where~$(u_1-u_2)_{+}=u_1-u_2>0$. Subtracting, we get
		\begin{equation*}
			\begin{split}
				\int_{\Omega}& \left\langle \abs{\nabla u_1}^{p-2}\,\nabla u_1 - \abs{\nabla u_2}^{p-2}\,\nabla u_2, \nabla (u_1-u_2) \right\rangle \chi_{\{u_1>u_2\}} \psi'(w) + c \left(u_1-u_2\right) \psi(w) \, dx \\
				&\leq \int_{\Omega} f \psi(w) \, dx,
			\end{split}
		\end{equation*}
		where~$f \coloneqq g_1(\cdot,u_1)-g_2(\cdot,u_2)$. On~$\Omega_+$ we have
		\begin{equation*}
			\abs*{f - c \left(u_1-u_2\right)} \leq \bigl(1+\norma*{c}_{L^\infty(\Omega)}\bigr) w,
		\end{equation*}
		therefore, exploiting~\eqref{eq:disp-basso}, we deduce
		\begin{equation}
			\label{eq:test1}
			c(p) \int_{\Omega} \left(\abs*{\nabla u_1}+\abs*{\nabla u_2} \right)^{p-2} \abs*{\nabla (u_1-u_2)_+}^2 \,\psi'(w) \, dx \leq \left(1+\norma*{c}_{L^\infty(\Omega)}\right) \int_{\Omega} w \psi(w) \, dx.
		\end{equation}
		To proceed further, observe that, by~\eqref{eq:rel_qr}, it follows that~$\psi'(t)= \left(v'(t)\right)^2$. Therefore,
		\begin{equation*}
			\psi(t) = \int_{0}^t \left(v'(s)\right)^2 ds \leq v'(t)v(t).
		\end{equation*}
		Moreover, by~\eqref{eq:defv}, we have~$tv'(t) \leq rv(t)$. Now, writing~$v=v(w)$, we obtain
		\begin{gather*}
			\abs*{\nabla (u_1-u_2)_+}^2 \,\psi'(w) = \abs*{\nabla w}^2 \,\psi'(w) = \abs*{\nabla v}^2, \\
			w \psi(w) \leq wv'(w)v(w) \leq r v(w)^2 \leq r^2 v^2.
		\end{gather*}
		As a result,~\eqref{eq:test1} entails
		\begin{equation}
			\label{eq:test2}
			c(p) \int_{\Omega} \left(\abs*{\nabla u_1}+\abs*{\nabla u_2} \right)^{p-2} \abs*{\nabla v}^2 \, dx \leq \left(1+\norma*{c}_{L^\infty(\Omega)}\right) r^2 \int_{\Omega} v^2 \, dx.
		\end{equation}
		
		We now estimate an appropriate Lebesgue norm of~$v$ in terms of the left-hand side of~\eqref{eq:test2}. Ultimately, we will return to~$w$ and apply the iterative technique by choosing a sequence of values for~$r$. To proceed, we distinguish between two cases based on the value of~$p$.\newline
		
		\textbf{Case}~$\mathbf{1<p<2}$. By defining~$\beta \coloneqq c(p) (2C_1)^{p-2}$ and exploiting~\eqref{eq:bound-gradiente}, form~\eqref{eq:test2}, we deduce
		\begin{equation}
			\label{eq:test3}
			\int_{\Omega}\, \abs*{\nabla v}^2 \, dx \leq \frac{1+\norma*{c}_{L^\infty(\Omega)}}{\beta}\,r^2 \int_{\Omega}\, v^2 \, dx.
		\end{equation}
		By~\eqref{eq:defv}, it follows that~$v = l^r$ in a neighborhood of~$\partial\Omega$, hence Sobolev inequality implies
		\begin{equation*}
			\begin{split}
				\norma{v}_{L^{2^\ast}\!(\Omega)} &\leq \norma{v-l^r}_{L^{2^\ast}\!(\Omega)} + \norma{l^r}_{L^{2^\ast}\!(\Omega)} \\
				&\leq
				\begin{cases}
					\begin{aligned}
						& S^{-1} \norma{\nabla v}_{L^{2}(\Omega)} + \abs{\Omega}^{-\frac{1}{n}}\norma{v}_{L^{2}(\Omega)}					&& \text{if} \; n\geq 3, \\
						& S^{-1} \norma{\nabla v}_{L^{2}(\Omega)} + \abs{\Omega}^{-\frac{1}{4}}\norma{v}_{L^{2}(\Omega)}					&& \text{if} \; n=2,
					\end{aligned}
				\end{cases}
			\end{split}
		\end{equation*}
		where we have chosen~$2^\ast=4$ when~$n=2$. Here we used that
		\begin{equation*}
			\norma{l^r}_{L^{2^\ast}\!(\Omega)} = \abs{\Omega}^{\frac{1}{2^\ast}-\frac{1}{2}} \norma{l^r}_{L^{2}(\Omega)}
		\end{equation*}
		and that~$l^r \leq v$ in~$\Omega$. Since we may assume that~$\abs{\Omega} \leq 1$,~\eqref{eq:test3} yields
		\begin{equation}
			\label{eq:iter-0}
			\norma{v}_{L^{2^\ast}\!(\Omega)} \leq Kr \norma{v}_{L^{2}(\Omega)},
		\end{equation}
		where
		\begin{equation*}
			K \coloneqq \left(S^{-1} \sqrt{\frac{1+\norma*{c}_{L^\infty(\Omega)}}{\beta}}+1 \right) \abs{\Omega}^{-\frac{1}{\alpha(n)}} = \mathsf{K} \,\abs{\Omega}^{-\frac{1}{\alpha(n)}} \quad \mbox{with } \alpha(n) \coloneqq
			\begin{cases}
				\begin{aligned}
					& n		&& \text{if} \; n\geq 3, \\
					& 4		&& \text{if} \; n=2.
				\end{aligned}
			\end{cases}
		\end{equation*}
		Define now the set~$\Omega' \coloneqq \{x\in\Omega \,\lvert\, \mathsf{k}\leq w(x)<m\}$ and observe that~$v \leq w^r +(l^r-\mathsf{k}^r)$. Therefore, from~\eqref{eq:iter-0}, we infer
		\begin{equation*}
			\norma{w}^r_{L^{2^\ast r}(\Omega')} \leq Kr \left( \norma{w}^r_{L^{2r}(\Omega)}+\abs{\Omega}^\frac{1}{2} \left( l^r-\mathsf{k}^r \right) \right)
		\end{equation*}
		and letting~$l\to \mathsf{k}$,~$m\to + \infty$, by Fatou's lemma, we conclude that
		\begin{equation}
			\label{eq:iter-1}
			\norma{w}_{L^{2^\ast r}(\Omega)} \leq \left( Kr \right)^\frac{1}{r}  \norma{w}_{L^{2r}(\Omega)}.
		\end{equation}
		Let us set~$\chi \coloneqq 2^\ast/2>1$ and rewrite~\eqref{eq:iter-1} in the form
		\begin{equation*}
			\norma{w}_{L^{2\chi r}(\Omega)} \leq \left( Kr \right)^\frac{1}{r}  \norma{w}_{L^{2r}(\Omega)},
		\end{equation*}
		where every norm is finite since~$w \in L^\infty(\Omega)$.
		
		We are ready to carry out the iterative technique. First, choose~$r=1$, so that
		\begin{equation*}
			\norma{w}_{L^{2\chi}(\Omega)} \leq K \norma{w}_{L^{2}(\Omega)}.
		\end{equation*}
		Then, take~$r=\chi$, so that
		\begin{equation*}
			\norma{w}_{L^{2\chi^2}(\Omega)} \leq \left( K\chi \right)^\frac{1}{\chi}  \norma{w}_{L^{2\chi}(\Omega)} \leq K^{1+\frac{1}{\chi}} \chi^\frac{1}{\chi} \norma{w}_{L^{2}(\Omega)}.
		\end{equation*}
		Proceeding in this way, we obtain
		\begin{equation}
			\label{eq:iter-2}
			\norma{w}_{L^{2\chi^j}(\Omega)} \leq K^{\Sigma_{j-1}} \chi^{\Sigma'_{j-1}} \norma{w}_{L^{2}(\Omega)} \quad \mbox{for every } j \geq 1,
		\end{equation}
		where
		\begin{equation}
			\label{eq:sigmas}
			\Sigma_{j-1} \coloneqq \sum_{i=0}^{j-1} \frac{1}{\chi^i} \to \frac{\chi}{\chi-1}, \qquad \Sigma'_{j-1} \coloneqq \sum_{i=0}^{j-1} \frac{i}{\chi^i} \to \frac{\chi}{(\chi-1)^2},
		\end{equation}
		as~$j \to +\infty$. By letting~$j \to +\infty$ in~\eqref{eq:iter-2}, we deduce
		\begin{equation*}
			\norma{w}_{L^{\infty}(\Omega)} \leq
			\begin{cases}
				\begin{aligned}
					(Ke)^\frac{n}{2}&\norma{w}_{L^{2}(\Omega)}		&& \text{if} \; n\geq 3, \\
					4K^2 &\norma{w}_{L^{2}(\Omega)}					&& \text{if} \; n=2.
				\end{aligned}
			\end{cases}
		\end{equation*}
		Hence, the triangle inequality yields
		\begin{equation}
			\label{eq:iter-final}
			\norma{w}_{L^{\infty}(\Omega)} \leq
			\begin{cases}
				\begin{aligned}
					(Ke)^\frac{n}{2} &\norma{(u_1-u_2)_{+}}_{L^{2}(\Omega)} + (\mathsf{K}e)^\frac{n}{2} \,\mathsf{k}	&& \text{if} \; n\geq 3, \\
					4K^2 &\norma{(u_1-u_2)_{+}}_{L^{2}(\Omega)} + 4\mathsf{K}^2 \,\mathsf{k}		&& \text{if} \; n=2.
				\end{aligned}
			\end{cases}
		\end{equation}
		
		We aim to estimate~$\norma{(u_1-u_2)_{+}}_{L^{\infty}(\Omega)}$. Note that the simplest approach -- the~$L^{\infty}$-estimate of the~$L^{2}$-norm -- would yield a constant in terms of~$\abs{\Omega}$ that cancels out with the one hidden in~$K$. As a consequence,~$\norma{(u_1-u_2)_{+}}_{L^{\infty}(\Omega)}$ cannot be reabsorbed at the left-hand side.
		
		To overcome this issue, we start by testing~\eqref{eq:eqforu} with~$(u_1-u_2)_{+}$. Subtracting, we get
		\begin{equation*}
			\begin{split}
				\int_{\Omega} \left\langle \abs{\nabla u_1}^{p-2} \,\nabla u_1 - \abs{\nabla u_2}^{p-2} \,\nabla u_2, \nabla (u_1-u_2) \right\rangle \chi_{\{u_1>u_2\}} \, dx \leq \int_{\Omega} f \, (u_1-u_2)_+ \, dx,
			\end{split}
		\end{equation*}
		furthermore, observe that~$	\abs*{f} \, (u_1-u_2)_+ \leq w^2$ and, proceeding as above, we deduce
		\begin{equation}
			\label{eq:grad-w^2-1}
			\int_{\Omega}\, \abs*{\nabla (u_1-u_2)_+}^2 \, dx \leq \frac{1}{\beta} 	\int_{\Omega} w^2 \, dx.
		\end{equation}
		An application of the Poincaré inequality, together with~\eqref{eq:grad-w^2-1}, gives us
		\begin{equation*}
			\int_{\Omega} (u_1-u_2)_+^2 \, dx \leq \frac{1}{\beta} \left( \frac{\abs{\Omega}}{\abs{B_1}} \right)^{\!\! \frac{2}{n}} \int_{\Omega} w^2 \, dx \leq \frac{1}{\beta} \left( \frac{\abs{\Omega}}{\abs{B_1}} \right)^{\!\! \frac{2}{n}} \abs{\Omega} \norma{w}^2_{L^\infty(\Omega)}.
		\end{equation*}
		Combining this estimate with~\eqref{eq:iter-final}, we finally deduce that
		\begin{equation*}
			\norma{(u_1-u_2)_+}_{L^{\infty}(\Omega)} \leq \norma{w}_{L^{\infty}(\Omega)} \leq
			\begin{cases}
				\begin{aligned}
					2(\mathsf{K}e)^\frac{n}{2} &\, \mathsf{k}	&& \mbox{if } n\geq 3, \\
					8\mathsf{K}^2 &\, \mathsf{k}				&& \mbox{if } n=2,
				\end{aligned}
			\end{cases}
		\end{equation*}
		provided that
		\begin{equation*}
			\abs{\Omega} \leq
			\begin{cases}
				\begin{aligned}
					\min\left\{1, \left( \frac{\beta}{4(\mathsf{K}e)^n} \right)^{\!\! \frac{n}{2}} \abs{B_1} \right\}	&& \mbox{if } n\geq 3, \\
					\min\left\{1, \left( \frac{\beta}{(8\mathsf{K}^2)^2} \right)^{\!\! \frac{n}{2}} \abs{B_1}\right\}	&& \mbox{if } n=2.
				\end{aligned}
			\end{cases}
		\end{equation*}
		\medskip
		
		\textbf{Case}~$\mathbf{p>2}$. By setting~$\rho \coloneqq \left(\abs*{\nabla u_1}+\abs*{\nabla u_2} \right)^{p-2}$, we can rewrite~\eqref{eq:test2} as
		\begin{equation*}
			\label{eq:test3'}
			\int_{\Omega} \rho \,\abs*{\nabla v}^2 \, dx \leq \frac{1+\norma*{c}_{L^\infty(\Omega)}}{c(p)}\,r^2 \int_{\Omega} v^2 \, dx.
		\end{equation*}
		Now, choosing~$2<q<2_M$ and arguing as before, with~\eqref{eq:sobolev-pesata} in place of Sobolev inequality, we obtain
		\begin{equation*}
			\label{eq:iter-0'}
			\norma{v}_{L^{q}(\Omega)} \leq Kr \norma{v}_{L^{2}(\Omega)},
		\end{equation*}
		where, assuming that~$\abs{\Omega}\leq 1$, we define
		\begin{equation*}
			K \coloneqq \left(C_\mathcal{S} \sqrt{\frac{1+\norma*{c}_{L^\infty(\Omega)}}{c(p)}}+1 \right) \abs{\Omega}^{\frac{1}{q}-\frac{1}{2}} = \mathsf{K} \,\abs{\Omega}^{\frac{1}{q}-\frac{1}{2}}.
		\end{equation*}
		Note that the weight~$\rho$ satisfies~\eqref{eq:condizione-peso} since
		\begin{equation*}
			\frac{1}{\left(\abs*{\nabla u_1}+\abs*{\nabla u_2} \right)^{p-2}} \leq \frac{1}{\abs*{\nabla u_i}^{p-2}} \quad \mbox{for } i=1,2
		\end{equation*}
		and~\eqref{eq:assunzione-pesi} holds true. Setting~$\chi\coloneqq q/2>1$ and repeating the argument above, we replace~\eqref{eq:iter-final} with
		\begin{equation}
			\label{eq:iter-final'}
			\norma{w}_{L^{\infty}(\Omega)} \leq (Ke)^\frac{q}{q-2} \norma{(u_1-u_2)_{+}}_{L^{2}(\Omega)} + (\mathsf{K}e)^\frac{q}{q-2} \,\mathsf{k}.
		\end{equation}
		We now substitute~\eqref{eq:grad-w^2-1} with
		\begin{equation*}
			\label{eq:grad-w^2-1'}
			\int_{\Omega} \rho \,\abs*{\nabla (u_1-u_2)_+}^2 \, dx \leq \frac{1}{c(p)} \int_{\Omega} w^2 \, dx.
		\end{equation*}
		An application of the weighted Poincaré inequality~\eqref{eq:poincare-pesata} gives
		\begin{equation*}
			\int_{\Omega} (u_1-u_2)_+^2 \, dx \leq \frac{1}{c(p)} \,\mathcal{C}_P(\Omega) \,C_\mathcal{S}^2 \,\abs{\Omega} \norma{w}^2_{L^\infty(\Omega)}.
		\end{equation*}
		Finally, combining this estimate with~\eqref{eq:iter-final'}, we conclude that
		\begin{equation*}
			\norma{(u_1-u_2)_+}_{L^{\infty}(\Omega)} \leq \norma{w}_{L^{\infty}(\Omega)} \leq 2(\mathsf{K}e)^\frac{q}{q-2} \,\mathsf{k},
		\end{equation*}
		provided that
		\begin{equation*}
			\begin{split}
				\mathcal{C}_P(\Omega) &\leq \frac{1}{4} \, c(p) \, C_\mathcal{S}^{-2}(\mathsf{K}e)^{-\frac{2q}{q-2}} \quad \mbox{and} \quad \abs{\Omega} \leq 1,\\
				\mbox{or} \quad \abs{\Omega} &\leq	\min\left\{1, \left( 4c(p)^{-1}C_\mathcal{S}^2(\mathsf{K}e)^\frac{2q}{q-2} \right)^{\!\! -(p-1)n} \right\}.
			\end{split}
		\end{equation*}
	\end{proof}
	
	\subsection{Weak Harnack comparison inequality.}
	
	Next, we turn to the weak Harnack comparison inequality. When considering the difference of a supersolution and a subsolution to~\eqref{eq:mainprob}, it has already been established in Theorem~3.3 of~\cite{ds-har}, by exploiting Moser's iterative technique~\cite{mos-har} and Trudinger's improvement~\cite{trud-reg,trud-lin}. However, the constant appearing in the estimate depends on the radius of the ball in an unclear way.
	
	Following this proof closely, we show that it can be adapted to the case of~\eqref{eq:eqforu}, with an explicit dependence on the radius.
	
	\begin{theorem}
		\label{th:wh}
		Let~$\Omega \subseteq \R^n$ be an open set, and let~$u_1,u_2\in C^1(\Omega)$ be a weak subsolution and a weak supersolution to~\eqref{eq:eqforu} with~$p > 2$, respectively. Moreover, suppose that~\eqref{eq:assunzione-pesi}-\eqref{eq:bound-gradiente} hold true.
		
		Assume that~$\overline{B_{5R}(x_0)} \subseteq \Omega$ for some~$R \leq 1$. Let~$q > 2$ be any real number such that
		\begin{equation}
			\label{eq:defi-q>}
			q > \frac{2 \, 2_M}{2_M-2},
		\end{equation}
		and fix any~$2<\mathfrak{q}<2_M$, depending on~$n$,~$p$, and~$q$, such that
		\begin{equation}
			\label{eq:cond-q-0}
			q > \frac{2\mathfrak{q}}{\mathfrak{q}-2}.
		\end{equation}
		Finally, suppose that~$c \in L^\infty\!\left(B_{5R}(x_0)\right)$, that~$g_1(\cdot,u_1)-g_2(\cdot,u_2) \in L^q\!\left(B_{5R}(x_0)\right)$, define $\mathsf{k}\coloneqq\norma{g_1(\cdot,u_1)-g_2(\cdot,u_2)}_{L^q\left(B_{5R}(x_0)\right)}$, and assume that, for some~$\mathfrak{c}\geq 1$, it holds
		\begin{equation}
			\label{eq:positivita}
			u_1-u_2 \leq \mathfrak{c} \mathsf{k} \quad\mbox{in } B_{5R}(x_0).
		\end{equation}
		
		Then, for every~$0<s<\chi\coloneqq\mathfrak{q}/2$, there exists a constant~$\mathfrak{C} \geq 1$ such that
		\begin{equation}
			\label{eq:weak-harnak-fond}
			\mathcal{M}(R) \norma*{u_2-u_1}_{L^s(B_{2R}(x_0))} \leq \mathfrak{C} \left(\inf_{B_{R}(x_0)} (u_2-u_1) +2\mathfrak{c} \mathsf{k}\right) \!,
		\end{equation}
		where~$\mathcal{M}(R)$ is the function of the radius given by
		\begin{equation*}
			\mathcal{M}(R) \coloneqq  R^{-n/s} \left(c_\flat R\right)^{C_\natural/R^\frac{2}{\mathfrak{q}-2}}
		\end{equation*} 
		for some small~$c_\flat \in (0,1)$ and some large~$C_\natural \geq 1$.
		
		The constants~$\mathfrak{C}$,~$c_\flat$, and~$C_\natural$ depend only on~$n$,~$p$,~$s$,~$q$,~$C_1$,~$\norma*{c}_{L^\infty(B_{5R}(x_0))}$, as well as an upper bound for~$C_\mathcal{S}$ on~$B_{5R}(x_0)$.
	\end{theorem}

	\begin{remark}
		It is easy to see that the conditions in~\eqref{eq:defi-q>}-\eqref{eq:cond-q-0} are equivalent to requiring that~$q>2$ and
		\begin{equation*}
			\frac{2q}{q-2} < \mathfrak{q} < 2_M.
		\end{equation*}
	\end{remark}

	\begin{proof}[Proof of Theorem~\ref{th:wh}]
		Set~$\widehat{\mathsf{k}} \coloneqq \mathfrak{c} \mathsf{k} \geq \mathsf{k}$ and define
		\begin{equation*}
			v \coloneqq u_2-u_1+2\widehat{\mathsf{k}} \geq 0.
		\end{equation*}
		Note that~$v \in C^1(\Omega)$ and we may assume that~$v \geq \tau>0$ in~$B_{5R}(x_0)$. If this is not true, we can consider~$v+\tau$ and then let~$\tau\to 0$.
		
		Take~$\eta \in C^1_c(B_{5R}(x_0))$, with~$\eta\geq 0$ in~$\Omega$, and define
		\begin{equation*}
			\phi \coloneqq \eta^2 v^\beta \geq 0
		\end{equation*}
		for~$\beta<0$, so that
		\begin{equation*}
			\nabla\phi=2\eta v^\beta \,\nabla\eta + \beta\eta^2 v^{\beta-1} \,\nabla v.
		\end{equation*}
		By using~$\phi$ as test function in~\eqref{eq:eqforu}, we get
		\begin{equation*}
			\begin{split}
				\int_\Omega &-\!\left\langle \abs*{\nabla u_1}^{p-2}\,\nabla u_1, \nabla\phi \right\rangle - c \, u_1 \phi +g_1\left(x,u_1\right)\phi \, dx \\
				&\geq \int_\Omega -\left\langle \abs*{\nabla u_2}^{p-2}\,\nabla u_2, \nabla\phi \right\rangle - c \, u_2 \phi +g_2\left(x,u_2\right)\phi \, dx,
			\end{split}
		\end{equation*}
		that is
		\begin{equation}
			\label{eq:eq-diff-1}
			\begin{split}
				\int_\Omega &-\beta \eta^2 v^{\beta-1} \left\langle \abs*{\nabla u_2}^{p-2}\,\nabla u_2 - \abs*{\nabla u_1}^{p-2}\,\nabla u_1, \nabla v\right\rangle \\
				&- 2\eta v^\beta \left\langle \abs*{\nabla u_2}^{p-2}\,\nabla u_2 - \abs*{\nabla u_1}^{p-2}\,\nabla u_1, \nabla \eta\right\rangle dx \leq \int_\Omega c \left(u_2-u_1\right) \phi + f \phi \, dx,
			\end{split}
		\end{equation}
		where~$f \coloneqq g_1(\cdot,u_1)-g_2(\cdot,u_2)$. It is crucial to observe that~\eqref{eq:positivita} implies
		\begin{equation}
			\label{eq:stime-v}
			\abs*{u_2-u_1} \leq v \quad \mbox{and} \quad v\geq \widehat{\mathsf{k}} \quad\mbox{in } B_{5R}(x_0).
		\end{equation}
		With~\eqref{eq:stime-v} at hand and exploiting~\eqref{eq:disp-basso}-\eqref{eq:disp-alto}, we can estimate the terms of the left-hand side and the first term on the right-hand side of~\eqref{eq:eq-diff-1} as follows
		\begin{gather*}
			-\beta \left\langle \abs*{\nabla u_2}^{p-2}\,\nabla u_2 - \abs*{\nabla u_1}^{p-2}\,\nabla u_1, \nabla v\right\rangle \geq c(p) \,\abs{\beta} \,\rho \,\abs{\nabla v}^2,\\
			\abs*{2\eta v^\beta \left\langle \abs*{\nabla u_2}^{p-2}\,\nabla u_2 - \abs*{\nabla u_1}^{p-2}\,\nabla u_1, \nabla \eta\right\rangle} \leq 2C(p) \eta v^\beta \rho \,\abs{\nabla v}\abs{\nabla \eta}, \\
			\abs*{c  \left(u_2-u_1\right) \phi} \leq \norma*{c}_{L^\infty(B_{5R}(x_0))} \eta^2 v^{\beta+1},
		\end{gather*}
		where~$\rho \coloneqq \left(\abs*{\nabla u_1}+\abs*{\nabla u_2} \right)^{p-2}$. Then, by H\"older inequality, we obtain
		\begin{equation}
			\label{eq:rhs-wh}
			\int_\Omega f \phi \, dx = \int_\Omega \,\frac{f}{v} \, \eta^2 v^{\beta+1} \, dx \leq \left(\int_{\mathrm{supp}\,\eta} \,\abs*{\frac{f}{v}}^q \, dx \right)^{\!\!\frac{1}{q}} \left(\int_{\Omega} \left(\eta^2 v^{\beta+1}\right)^{q'} dx \right)^{\!\!\frac{1}{q'}}\!,
		\end{equation}
		where~$q'$ is the H\"older conjugate of~$q$. In light of~\eqref{eq:cond-q-0}, note that~$2<\mathfrak{q}<2_M$ satisfies
		\begin{equation}
			\label{eq:cond-q}
			q' < \frac{2\mathfrak{q}}{\mathfrak{q}+2} \quad \mbox{and} \quad q' < \frac{\mathfrak{q}}{2},
		\end{equation}
		and let~$\theta \in (0,1)$ be such that~$q'=\theta+\left(1-\theta\right) \mathfrak{q}/2$. Furthermore observe that, by~\eqref{eq:assunzione-pesi}, $\rho \in L^\infty(\Omega)$ fulfills~\eqref{eq:condizione-peso}. Then, H\"older inequality and the weighted Sobolev inequality~\eqref{eq:sobolev-pesata} entail
		\begin{align*}
			\int_{\Omega} &\left(\eta^2 v^{\beta+1}\right)^{q'} dx \leq \left(\int_{\Omega} \eta^2 v^{\beta+1} \, dx \right)^{\!\!\theta} \left(\int_{\Omega} \left(\eta v^\frac{{\beta+1}}{2}\right)^{\!\mathfrak{q}} dx \right)^{\!\!1-\theta} \\
			&\leq \left(\int_{\Omega} \eta^2 v^{\beta+1} \, dx \right)^{\!\!\theta} \left(C_\mathcal{S}^2 \int_{\Omega} \rho\,\abs*{\nabla\eta}^2 v^{\beta+1} + \left(\frac{\abs*{\beta+1}}{2}\right)^{\!2} \rho\eta^2 v^{\beta-1} \,\abs{\nabla v}^2 \, dx \right)^{\!\!(1-\theta)\frac{\mathfrak{q}}{2}} \\
			&\leq C_q \, C_\mathcal{S}^{(1-\theta)\mathfrak{q}} \left(\int_{\Omega} \eta^2 v^{\beta+1} \, dx \right)^{\!\!\theta} \cdot \\
			&\quad \cdot \left\{ \left(\int_{\Omega} \rho\,\abs*{\nabla\eta}^2 v^{\beta+1} dx \right)^{\!\!(1-\theta)\frac{\mathfrak{q}}{2}} + \abs*{\beta+1}^{(1-\theta)\mathfrak{q}} \left(\int_{\Omega} \rho\eta^2 v^{\beta-1} \,\abs{\nabla v}^2 \, dx \right)^{\!\!(1-\theta)\frac{\mathfrak{q}}{2}} \right\} \\
			&\leq C_q \, C_\mathcal{S}^{(1-\theta)\mathfrak{q}} \left(\int_\Omega \left[\eta^2+\rho\,\abs*{\nabla\eta}^2\right] v^{\beta+1} \, dx\right)^{\!\! q'} \\
			&\quad + C_q \, C_\mathcal{S}^{(1-\theta)\mathfrak{q}} \,\abs*{\beta+1}^{(1-\theta)\mathfrak{q}} \left(\int_{\Omega} \eta^2 v^{\beta+1} \, dx \right)^{\!\!\theta} \left( \int_{\Omega} \rho\eta^2 v^{\beta-1} \,\abs{\nabla v}^2 \, dx \right)^{\!\!(1-\theta)\frac{\mathfrak{q}}{2}}\!,
		\end{align*}
		where~$C_q \coloneqq  \max\left\{1, 2^{(1-\theta)\frac{\mathfrak{q}}{2}-1}\right\}$. Further, by exploiting the second estimate in~\eqref{eq:stime-v} and recalling the definition of~$\mathsf{k}$, it follows that
		\begin{equation}
			\label{eq:est-f<1}
			\left(\int_{\mathrm{supp}\,\eta} \,\abs*{\frac{f}{v}}^q \, dx \right)^{\!\!\frac{1}{q}} \leq \frac{1}{\mathsf{k}} \,\norma*{f}_{L^q(\mathrm{supp}\,\eta)} \leq 1.
		\end{equation}
		Note that~\eqref{eq:est-f<1} remains valid even when~$\mathsf{k}=0$, since in this case~$f=0$ a.e.\ in~$B_{5R}(x_0)$. Thus, going back to~\eqref{eq:rhs-wh}, we deduce
		\begin{equation}
			\label{eq:intf-silim}
			\begin{split}
				\int_\Omega &f \phi \, dx \leq C_q^\frac{1}{q'} C_\mathcal{S}^{(1-\theta)\frac{\mathfrak{q}}{q'}} \int_\Omega \left[\eta^2+\rho\,\abs*{\nabla\eta}^2\right] v^{\beta+1} \, dx \\
				&+ C_q^\frac{1}{q'} C_\mathcal{S}^{(1-\theta)\frac{\mathfrak{q}}{q'}} \,\abs*{\beta+1}^{(1-\theta)\frac{\mathfrak{q}}{q'}}  \left(\int_{\Omega} \eta^2 v^{\beta+1} \, dx \right)^{\!\!\frac{\theta}{q'}} \left( \int_{\Omega} \rho\eta^2 v^{\beta-1} \,\abs{\nabla v}^2 \, dx \right)^{\!\!(1-\theta)\frac{\mathfrak{q}}{2q'}}\!.
			\end{split}
		\end{equation}
		Finally,~\eqref{eq:eq-diff-1} yields
		\begin{align}
			\label{eq:eq-diff-2}
			c(p) &\,\abs{\beta} \int_\Omega \rho\eta^2 v^{\beta-1} \,\abs{\nabla v}^2 \, dx \leq 2C(p) \int_\Omega \rho \eta v^\beta \,\abs{\nabla v}\abs{\nabla \eta} \notag \\
			&+ C_q^\frac{1}{q'} C_\mathcal{S}^{(1-\theta)\frac{\mathfrak{q}}{q'}} \int_\Omega \left[\eta^2+\rho\,\abs*{\nabla\eta}^2\right] v^{\beta+1} \, dx + \norma*{c}_{L^\infty(B_{5R}(x_0))} \int_\Omega \eta^2 v^{\beta+1} \, dx \\
			&+ C_q^\frac{1}{q'} C_\mathcal{S}^{(1-\theta)\frac{\mathfrak{q}}{q'}} \,\abs*{\beta+1}^{(1-\theta)\frac{\mathfrak{q}}{q'}}  \left(\int_{\Omega} \eta^2 v^{\beta+1} \, dx \right)^{\!\!\frac{\theta}{q'}} \left( \int_{\Omega} \rho\eta^2 v^{\beta-1} \,\abs{\nabla v}^2 \, dx \right)^{\!\!(1-\theta)\frac{\mathfrak{q}}{2q'}}\!. \notag
		\end{align}
		We now need to exclude the possibility that the last summand on the right-hand side of~\eqref{eq:eq-diff-2} blows up in the iteration when~$\abs{\beta} >1$. To this end, thanks to the definition of~$\theta$ and the first condition in~\eqref{eq:cond-q}, it follows that
		\begin{equation}
			\label{eq:tq'<1}
			(1-\theta)\frac{\mathfrak{q}}{q'} = 2 \left(1-\frac{\theta}{q'}\right) \in (0,1),
		\end{equation}
		which, in turn, implies~$q'<2\theta$ and
		\begin{equation*}
			\frac{\abs*{\beta+1}^{(1-\theta)\frac{\mathfrak{q}}{q'}} }{\abs{\beta}} \leq \max\left\{1,\frac{1}{\abs{\beta}}\right\} \quad \mbox{for every } \beta<0.
		\end{equation*}
		Hence, from~\eqref{eq:eq-diff-2}, we get
		\begin{equation*}
			\begin{split}
				c(p) &\int_\Omega \rho\eta^2 v^{\beta-1} \,\abs{\nabla v}^2 \, dx \leq 2C(p) \max\left\{1,\frac{1}{\abs{\beta}}\right\} \int_\Omega \rho \eta v^\beta \,\abs{\nabla v}\abs{\nabla \eta} \, dx \\
				&+ \norma*{c}_{L^\infty(B_{5R}(x_0))} \max\left\{1,\frac{1}{\abs{\beta}}\right\} \int_\Omega \eta^2 v^{\beta+1} \, dx \\
				&+ C_q^\frac{1}{q'} C_\mathcal{S}^{(1-\theta)\frac{\mathfrak{q}}{q'}} \max\left\{1,\frac{1}{\abs{\beta}}\right\} \int_\Omega \left[\eta^2+\rho\,\abs*{\nabla\eta}^2\right] v^{\beta+1} \, dx \\
				&+ C_q^\frac{1}{q'} C_\mathcal{S}^{(1-\theta)\frac{\mathfrak{q}}{q'}} \max\left\{1,\frac{1}{\abs{\beta}}\right\} \left(\int_{\Omega} \eta^2 v^{\beta+1} \, dx \right)^{\!\!\frac{\theta}{q'}} \left( \int_{\Omega} \rho\eta^2 v^{\beta-1} \,\abs{\nabla v}^2 \, dx \right)^{\!\!(1-\theta)\frac{\mathfrak{q}}{2q'}}\!.
			\end{split}
		\end{equation*}
		Finally, by applying Young's inequality a couple of times, we deduce
		\begin{equation*}
			\begin{split}
				\int_\Omega \rho\eta^2 v^{\beta-1} \,\abs{\nabla v}^2 \, dx &\leq K \max\left\{1,\frac{1}{\abs{\beta}}\right\}^{\!2} \max\left\{1, C_\mathcal{S}^{\frac{1-\theta}{\theta}\mathfrak{q}} \right\} \int_\Omega \left[\eta^2+\rho\,\abs*{\nabla\eta}^2\right] v^{\beta+1} \, dx \\
				&\leq K \max\left\{1,\frac{1}{\abs{\beta}}\right\}^{\!2} \max\left\{1, C_\mathcal{S}^4 \right\} \int_\Omega \left[\eta^2+\rho\,\abs*{\nabla\eta}^2\right] v^{\beta+1} \, dx \\
			\end{split}
		\end{equation*}
		for some~$K \geq 1$ depending only on~$n$,~$p$,~$q$, and~$\norma*{c}_{L^\infty(B_{5R}(x_0))}$. The last estimate follows by noting that, in view of~\eqref{eq:tq'<1} and the definition of~$\theta$, we have
		\begin{equation*}
			\frac{1-\theta}{\theta}\,\mathfrak{q} = 2\,\frac{q'-\theta}{\theta} \leq 4.
		\end{equation*}
		
		Define now~$r \coloneqq \beta+1$ and
		\begin{equation*}
			w \coloneqq
			\begin{cases}
				\begin{aligned}
					& v^\frac{\beta+1}{2}	&& \mbox{if } \beta\neq -1, \\
					& \log v				&& \mbox{if } \beta = -1,
				\end{aligned}
			\end{cases}
		\end{equation*}
		so that we can rewrite~\eqref{eq:eq-diff-2} as
		\begin{equation}
			\label{eq:eq-diff-final}
			\int_\Omega \rho\eta^2 \,\abs{\nabla w}^2 \, dx 
			\leq
			\begin{cases}
				\begin{aligned}
					& K' \max\left\{1,\frac{1}{\abs{\beta}}\right\}^{\! 2} r^2 \int_\Omega \left[\eta^2+\rho\,\abs*{\nabla\eta}^2\right] w^2 \, dx	&& \mbox{if } \beta\neq -1, \\
					& K' \int_\Omega \left[\eta^2+\rho\,\abs*{\nabla\eta}^2\right] dx	&& \mbox{if } \beta = -1,
				\end{aligned}
			\end{cases}
		\end{equation}
		where~$K' \geq 1$ depends on an upper bound for~$C_\mathcal{S}$, as well.
		
		Since~$2<\mathfrak{q}<2_M$, by the weighted Sobolev inequality~\eqref{eq:sobolev-pesata}, we get
		\begin{equation*}
			\norma{\eta w}^2_{L^{\mathfrak{q}}(\Omega)} \leq C_\mathcal{S}^2 \int_\Omega \rho \left[\eta^2\,\abs*{\nabla w}^2+w^2\,\abs*{\nabla\eta}^2\right] dx.
		\end{equation*}
		For~$\beta\neq -1$, this, together with~\eqref{eq:bound-gradiente} and~\eqref{eq:eq-diff-final}, implies that
		\begin{align}
			\label{eq:daiterare}
			\begin{split}
				\norma{\eta w}^2_{L^{\mathfrak{q}}(\Omega)}
				&\leq C \, C_\mathcal{S}^2 \max\left\{1,\frac{1}{\abs{\beta}}\right\}^{\! 2} \left(1+r^2\right) \int_\Omega \left[\eta^2+\abs*{\nabla\eta}^2\right] w^2 \, dx \\
				&\quad\leq C_2 \max\left\{1,\frac{1}{\abs{\beta}}\right\}^{\! 2} \left(1+\abs*{r}\right)^2  \int_\Omega \left[\eta^2+\abs*{\nabla\eta}^2\right] w^2 \, dx,
			\end{split}
		\end{align}
		where~$C_2\geq1$ depends only on~$n$,~$p$,~$q$,~$C_1$,~$\norma*{c}_{L^\infty(B_{5R}(x_0))}$, as well as an upper bound for~$C_\mathcal{S}$. This equation is precisely (A.11) in~\cite{ds-har} or (5.12) in~\cite{trud-lin}.
		
		Recall now that, for~$m\neq 0$ and~$\delta>0$, we write
		\begin{equation*}
			\Phi(m,\delta,v) \coloneqq \left( \int_{B_\delta(x_0)} \,\abs{v}^m \, dx \right)^{\!\! \frac{1}{m}} \!,
		\end{equation*}
		and, in addition,~$\Phi$ enjoys the following well-known properties
		\begin{gather*}
			\lim_{m \to +\infty} \Phi(m,\delta,v) = \Phi(+\infty,\delta,v) = \sup_{B_\delta(x_0)} \abs{v}, \\
			\lim_{m \to -\infty} \Phi(m,\delta,v) = \Phi(-\infty,\delta,v) = \inf_{B_\delta(x_0)} \abs{v}.
		\end{gather*}
		Assume that~$R \leq h' < h'' \leq 5R$ and~$R \leq 1$. From~\eqref{eq:daiterare}, arguing as in the deduction of (A.14)-(A.15) in~\cite{ds-har}, we get
		\begin{gather}
			\label{eq:dait->0}
			\Phi(\chi r,h',v) \leq \left(\frac{C_3 \left(1+r\right)}{\abs*{\beta}\left(h''-h'\right)}\right)^{\!\! \frac{2}{r}} \Phi(r,h'',v) \quad \mbox{if } 0<r<1, \\
			\label{eq:dait-<0}
			\Phi(\chi r,h',v) \geq \left(\frac{C_3 \left(1+\abs*{r}\right)}{h''-h'}\right)^{\!\! \frac{2}{r}} \Phi(r,h'',v) \quad \mbox{if } r<0,
		\end{gather}
		for some~$C_3\geq 1$ depending on an upper bound for~$C_\mathcal{S}$. These two inequalities are the elementary building blocks for Moser's iteration.
		
		Provided that we have (A.18) in~\cite{ds-har} or (5.17) in~\cite{trud-lin}, that is
		\begin{equation}
			\label{eq:A18}
			\Phi\left(r_0,\frac{5}{2}R,v\right) \leq C_\sharp^\frac{1}{r_0} \,\Phi\left(-r_0,\frac{5}{2}R,v\right)
		\end{equation}
		for some~$r_0>0$ and~$C_\sharp \geq 1$, we can apply Moser's iteration starting form~\eqref{eq:dait-<0} and reach the conclusion for~$0<s\leq r_0$.
		
		More precisely, we take
		\begin{equation*}
			r_k = -r_0 \chi^k \quad \mbox{ and } \quad h_k = R\left(1+\frac{3}{2^{1+k}}\right).
		\end{equation*}
		Iterating and adopting the notation in~\eqref{eq:sigmas}, we get
		\begin{equation*}
			\begin{split}
				\Phi(r_{k+1},h_{k+1},v) &\geq \left(\frac{C_3}{3}\right)^{\!\! -\frac{2}{r_0 \chi^k}} 2^{-\frac{2(k+1)}{r_0 \chi^k}} \chi^{-\frac{2k}{r_0 \chi^k}} \left(1+r_0\right)^{-\frac{2}{r_0 \chi^k}} R^{\frac{2}{r_0 \chi^k}} \,\Phi(r_k,h_k,v) \\
				&\geq \left(\frac{C_3}{3}\right)^{\!\! -\frac{2}{r_0}\Sigma_k} 2^{-\frac{2}{r_0}\left(\Sigma_k+\Sigma'_k\right)} \chi^{-\frac{2}{r_0}\Sigma'_k} \left(1+r_0\right)^{-\frac{2}{r_0}\Sigma_k} R^{\frac{2}{r_0}\Sigma_k} \,\Phi\left(-r_0,\frac{5}{2}R,v\right) \!.
			\end{split}
		\end{equation*}
		Since~$r_k \to -\infty$,~$h_k \to R$,~$\Sigma_k \to \Sigma$, and~$\Sigma_k' \to \Sigma'$ as~$k\to +\infty$, taking the limit in the previous relation, we conclude that
		\begin{equation}
			\label{eq:A17}
			\Phi(-\infty,R,v) \geq \mathcal{A}(r_0) \,\Phi\left(-r_0,\frac{5}{2}R,v\right) \!,
		\end{equation}
		where, since we will take~$r_0 \leq 1$, we have
		\begin{equation*}
			\mathcal{A}(r_0) \coloneqq C_4^{-\frac{2}{r_0}} R^{\frac{2}{r_0}\Sigma}
		\end{equation*}
		and~$C_4 \geq 1$, possibly by choosing a larger~$C_2 \geq 1$ depending on~$\chi$. In addition, since~\eqref{eq:A18} implies that
		\begin{equation*}
			\begin{split}
				\Phi(s,2R,v) &\leq \abs{B_1}^{\frac{1}{s}-\frac{1}{r_0}} \left(\frac{5}{2}R\right)^{\!\! n\left(\frac{1}{s}-\frac{1}{r_0}\right)} \Phi\left(r_0,\frac{5}{2}R,v\right) \\
				&\leq c_{n,s} \, R^{n\left(\frac{1}{s}-\frac{1}{r_0}\right)} C_\sharp^\frac{1}{r_0} \Phi\left(-r_0,\frac{5}{2}R,v\right)
			\end{split}
		\end{equation*}
		for~$0<s\leq r_0 \leq 1$ and some~$c_{n,s}>0$, we deduce from~\eqref{eq:A17} that
		\begin{equation}
			\label{eq:finale-1}
			\Phi(s,2R,v) \leq c_{n,s} \, R^{n\left(\frac{1}{s}-\frac{1}{r_0}\right)} C_\sharp^\frac{1}{r_0} \mathcal{A}(r_0)^{-1} \Phi(-\infty,R,v).
		\end{equation}
		This allows to conclude for $0<s\leq r_0$. \newline
		
		We now need to establish~\eqref{eq:A18}. Consider~$\beta = -1$ and observe that, by taking~$\eta=1$ in~$B_{4R}(x_0)$ with~$\abs{\nabla\eta} \leq 1/R$ in~\eqref{eq:eq-diff-final}, we have
		\begin{equation}
			\label{eq:stima-bolla}
			\int_{B_{4R}(x_0)} \rho \,\abs{\nabla w}^2 \, dx \leq 2 \, 5^{n} \abs{B_1} K' \max\left\{1,(2C_1)^{p-2}\right\} R^{n-2}.
		\end{equation}
		Up to possibly replacing~$v$ with~$v/h$, where
		\begin{equation*}
			h \coloneqq e^{\dashint_{B_{4R}(x_0)}\log v \, dx},
		\end{equation*}
		we can assume that~$w$ has zero mean on~$B_{4R}(x_0)$. The constant~$h$ does not affect the following calculations and can be canceled at the end of the proof. For zero-mean functions, we can apply the weighted Sobolev inequality~\eqref{eq:sobolev-pesata} -- note that for a ball the constants~$\widehat{C}$ in the two cases of Theorem~\ref{th:w-sobolev} differ only for a dimensional factor. Thus, we obtain
		\begin{equation}
			\label{eq:A24}
			\norma{w}_{L^\mathfrak{q}(B_{4R}(x_0))}^2 \leq C \, C_\mathcal{S}^2 \int_{B_{4R}(x_0)} \rho \,\abs{\nabla w}^2 \, dx \leq C_5^2 R^{n-2} \leq C_5^2,
		\end{equation}
		by~\eqref{eq:stima-bolla} and recalling that~$R\leq 1$, where~$C_5>0$ depends on an upper bound for~$C_\mathcal{S}$. Moreover, observe that the constant in~\eqref{eq:stima-bolla}, and hence~$C_5$, do not blow up as~$\tau \to 0$.
		
		Take~$\eta \in C^1_c(B_{4R}(x_0))$, with~$\eta\geq 0$ in~$\Omega$, and define
		\begin{equation*}
			\Phi \coloneqq \frac{\eta^2}{v} \left(\abs{w}^\beta+(2\beta)^\beta\right) \geq 0
		\end{equation*}
		for~$\beta\geq 1$, so that
		\begin{equation*}
			\nabla\Phi=2\;\frac{\eta}{v} \left(\abs{w}^\beta+(2\beta)^\beta\right) \,\nabla\eta + \frac{\eta^2}{v^2} \left(\beta\sign(w)\abs{w}^{\beta-1}-\abs{w}^{\beta}-(2\beta)^\beta \right) \,\nabla v.
		\end{equation*}
		By using~$\Phi$ as test function in~\eqref{eq:eqforu}, we get
		\begin{equation*}
			\int_\Omega -\big\langle \abs*{\nabla u_2}^{p-2}\,\nabla u_2-\abs*{\nabla u_1}^{p-2}\,\nabla u_1, \nabla\Phi \big\rangle \, dx \leq \int_\Omega c \left(u_2-u_1\right) \Phi + f\Phi \, dx,
		\end{equation*}
		that is
		\begin{align}
			\label{eq:eq-diff-1'}
			\int_\Omega &\,\frac{\eta^2}{v^2} \left(-\beta\sign(w)\abs{w}^{\beta-1}+\abs{w}^{\beta}+(2\beta)^\beta \right) \left\langle \abs*{\nabla u_2}^{p-2}\,\nabla u_2 - \abs*{\nabla u_1}^{p-2}\,\nabla u_1, \nabla v\right\rangle \notag \\
			&- 2\,\frac{\eta}{v} \left(\abs{w}^\beta+(2\beta)^\beta\right) \left\langle \abs*{\nabla u_2}^{p-2}\,\nabla u_2 - \abs*{\nabla u_1}^{p-2}\,\nabla u_1, \nabla \eta\right\rangle dx \\
			&\leq \int_\Omega c \left(u_2-u_1\right) \Phi + f \Phi \, dx. \notag
		\end{align}
		Note that, for~$\beta \geq 1$, the following elementary inequalities hold
		\begin{gather}
			\label{eq:est-beta1}
			2\beta\,\abs{w}^{\beta-1} \leq \frac{\beta-1}{\beta} \,\abs{w}^{\beta} + \frac{1}{\beta} \,(2\beta)^\beta \leq \abs{w}^{\beta} + (2\beta)^\beta,\\
			\label{eq:est-beta2}
			-\beta\sign(w)\abs{w}^{\beta-1}+\abs{w}^{\beta}+(2\beta)^\beta \geq \beta\,\abs{w}^{\beta-1}.
		\end{gather}
		Moreover, by~\eqref{eq:est-beta1}, a simple computation reveals that
		\begin{equation}
			\label{eq:est-beta3}
			\abs{w}^\beta \leq e \left(\abs{w}^{\beta+1}+(2\beta)^\beta\right) \!.
		\end{equation}
		Exploiting~\eqref{eq:disp-basso}-\eqref{eq:disp-alto}, we can estimate each term on the left-hand side of~\eqref{eq:eq-diff-1'}. Indeed, by~\eqref{eq:est-beta2} and the monotonicity of the~$p$-Laplace operator, we have
		\begin{equation*}
			\begin{split}
				\frac{\eta^2}{v^2}& \left(-\beta\sign(w)\abs{w}^{\beta-1}+\abs{w}^{\beta}+(2\beta)^\beta \right) \left\langle \abs*{\nabla u_2}^{p-2}\,\nabla u_2 - \abs*{\nabla u_1}^{p-2}\,\nabla u_1, \nabla v\right\rangle \\
				&\geq c(p)\beta\rho\eta^2 \,\abs{w}^{\beta-1}\,\abs{\nabla w}^2,
			\end{split}
		\end{equation*}
		for the second summand, taking into account~\eqref{eq:est-beta3} as well, we get
		\begin{equation*}
			\begin{split}
				&\abs*{2\,\frac{\eta}{v} \left(\abs{w}^\beta+(2\beta)^\beta\right) \left\langle \abs*{\nabla u_2}^{p-2}\,\nabla u_2 - \abs*{\nabla u_1}^{p-2}\,\nabla u_1, \nabla \eta\right\rangle} \\
				&\quad \leq 2\left(e+1\right)C(p) \,\frac{\eta}{v} \left(\abs{w}^{\beta+1}+(2\beta)^\beta\right) \rho \,\abs{\nabla v}\abs{\nabla \eta}.
			\end{split}
		\end{equation*}
		For the first term on the right-hand side of~\eqref{eq:eq-diff-1'}, recalling~\eqref{eq:stime-v} and~\eqref{eq:est-beta3}, we easily see that
		\begin{equation*}
			\begin{split}
				\abs*{c \left(u_2-u_1\right) \Phi} &\leq \norma*{c}_{L^\infty(B_{5R}(x_0))} \eta^2 \left(\abs{w}^\beta+(2\beta)^\beta\right) \\
				&\leq \left(e+1\right) \norma*{c}_{L^\infty(B_{5R}(x_0))} \eta^2 \left(\abs{w}^{\beta+1}+(2\beta)^\beta\right).
			\end{split}
		\end{equation*}
		For the last term in~\eqref{eq:eq-diff-1'}, we argue as for~\eqref{eq:rhs-wh} to deduce
		\begin{equation*}
			\begin{split}
				\int_\Omega f \Phi \, dx &= \int_\Omega \,\frac{f}{v} \,\eta^2 \left(\abs{w}^\beta+(2\beta)^\beta\right) dx \\
				&\leq \left(\int_{\mathrm{supp}\,\eta} \,\abs*{\frac{f}{v}}^q \, dx \right)^{\!\!\frac{1}{q}} \left(\int_{\Omega} \left[\eta^2 \left(\abs{w}^\beta+(2\beta)^\beta\right)\right]^{q'} \! dx \right)^{\!\!\frac{1}{q'}} \\
				&\leq (e+1) \left(\int_{\Omega} \left[\eta^2 \left(\abs{w}^{\beta+1}+(2\beta)^\beta\right)\right]^{q'} \! dx \right)^{\!\!\frac{1}{q'}}\!,
			\end{split}
		\end{equation*}
		where we used~\eqref{eq:est-f<1} and~\eqref{eq:est-beta3}. Then, H\"older inequality and the weighted Sobolev inequality~\eqref{eq:sobolev-pesata} imply
		\begin{align*}
			\int_{\Omega} &\left[\eta^2 \left(\abs{w}^{\beta+1}+(2\beta)^\beta\right)\right]^{q'} dx \leq \left(\int_{\Omega} \eta^2 \ell \, dx \right)^{\!\!\theta} \left(\int_{\Omega} \left[\eta \left(\abs{w}^{\frac{\beta+1}{2}}+(2\beta)^\frac{\beta}{2}\right)\right]^{\mathfrak{q}} dx \right)^{\!\!1-\theta} \\
			&\leq \left(\int_{\Omega} \eta^2 \ell \, dx \right)^{\!\!\theta} \left(C_\mathcal{S}^2 \int_{\Omega} 2\rho\,\abs*{\nabla\eta}^2 \ell + \left(\frac{\beta+1}{2}\right)^{\!\!2} \rho\eta^2 \,\abs*{w}^{\beta-1} \abs{\nabla w}^2 \, dx \right)^{\!\!(1-\theta)\frac{\mathfrak{q}}{2}} \\
			&\leq \widehat{C}_q \, C_\mathcal{S}^{(1-\theta)\mathfrak{q}} \left(\int_{\Omega} \eta^2 \ell \, dx \right)^{\!\!\theta} \Bigg\{ \left(\int_{\Omega} \rho\,\abs*{\nabla\eta}^2 \ell \, dx \right)^{\!\!(1-\theta)\frac{\mathfrak{q}}{2}} \\
			&\qquad+ (\beta+1)^{(1-\theta)\mathfrak{q}} \left(\int_{\Omega} \rho\eta^2 \,\abs{w}^{\beta-1} \abs{\nabla w}^2 \, dx \right)^{\!\!(1-\theta)\frac{\mathfrak{q}}{2}} \Bigg\} \\
			&\leq \widehat{C}_q \, C_\mathcal{S}^{(1-\theta)\mathfrak{q}} \left(\int_\Omega \left[\eta^2+\rho\,\abs*{\nabla\eta}^2\right] \ell \, dx\right)^{\!\! q'} \\
			&\quad + \widehat{C}_q \, C_\mathcal{S}^{(1-\theta)\mathfrak{q}} \left(\beta+1\right)^{(1-\theta)\mathfrak{q}} \left(\int_{\Omega} \eta^2 \ell \, dx \right)^{\!\!\theta} \left( \int_{\Omega} \rho\eta^2 \,\abs{w}^{\beta-1} \abs{\nabla w}^2 \, dx \right)^{\!\!(1-\theta)\frac{\mathfrak{q}}{2}}\!,
		\end{align*}
		where~$\widehat{C}_q \coloneqq  \max\left\{2, 2^{(1-\theta)\frac{\mathfrak{q}}{2}-1}\right\}$ and we set~$\ell=\ell(w) \coloneqq \abs{w}^{\beta+1}+(2\beta)^\beta$ to shorten the notation. Therefore, we obtain
		\begin{equation*}
			\begin{split}
				\int_\Omega f \Phi \, dx &\leq \widetilde{C}_q^\frac{1}{q'} \, C_\mathcal{S}^{(1-\theta)\frac{\mathfrak{q}}{q'}} \int_\Omega \left[\eta^2+\rho\,\abs*{\nabla\eta}^2\right] \ell \, dx \\
				&\quad + \widetilde{C}_q^\frac{1}{q'} \, C_\mathcal{S}^{(1-\theta)\frac{\mathfrak{q}}{q'}} \left(\beta+1\right)^{(1-\theta)\frac{\mathfrak{q}}{q'}} \left(\int_{\Omega} \eta^2 \ell \, dx \right)^{\!\!\frac{\theta}{q'}} \left( \int_{\Omega} \rho\eta^2 \,\abs{w}^{\beta-1} \abs{\nabla w}^2 \, dx \right)^{\!\!(1-\theta)\frac{\mathfrak{q}}{2q'}}\!,
			\end{split}
		\end{equation*}
		where~$\widetilde{C}_q= \left(e+1\right)\widehat{C}_q$. Now, since~$\beta \geq 1$,~\eqref{eq:eq-diff-1'} yields
		\begin{align*}
			c(p) &\int_\Omega \rho\eta^2 \abs{w}^{\beta-1} \,\abs{\nabla w}^2 \, dx \leq 2\left(e+1\right)C(p)\int_\Omega \rho \eta \ell \,\abs{\nabla w}\abs{\nabla \eta} \\
			&\quad+  \left(e+1\right) \norma*{c}_{L^\infty(B_{5R}(x_0))} \int_\Omega  \eta^2 \ell \, dx + \widetilde{C}_q^\frac{1}{q'} C_\mathcal{S}^{(1-\theta)\frac{\mathfrak{q}}{q'}} \int_\Omega \left[\eta^2+\rho\,\abs*{\nabla\eta}^2\right] \ell \, dx \\
			&\quad+ \widetilde{C}_q^\frac{1}{q'} C_\mathcal{S}^{(1-\theta)\frac{\mathfrak{q}}{q'}} \frac{\left(\beta+1\right)^{(1-\theta)\frac{\mathfrak{q}}{q'}}}{\beta} \left(\int_{\Omega} \eta^2 \ell \, dx \right)^{\!\!\frac{\theta}{q'}} \left( \int_{\Omega} \rho\eta^2 \,\abs{w}^{\beta-1} \abs{\nabla w}^2 \, dx \right)^{\!\!(1-\theta)\frac{\mathfrak{q}}{2q'}} \\
			&\leq 2\left(e+1\right)C(p)\int_\Omega \rho \eta \ell \,\abs{\nabla w}\abs{\nabla \eta} \\
			&\quad+ \left(e+1\right) \norma*{c}_{L^\infty(B_{5R}(x_0))} \int_\Omega  \eta^2 \ell \, dx + \widetilde{C}_q^\frac{1}{q'} C_\mathcal{S}^{(1-\theta)\frac{\mathfrak{q}}{q'}} \int_\Omega \left[\eta^2+\rho\,\abs*{\nabla\eta}^2\right] \ell \, dx \\
			&\quad+ 2 \widetilde{C}_q^\frac{1}{q'} C_\mathcal{S}^{(1-\theta)\frac{\mathfrak{q}}{q'}} \left(\int_{\Omega} \eta^2 \ell \, dx \right)^{\!\!\frac{\theta}{q'}} \left( \int_{\Omega} \rho\eta^2 \,\abs{w}^{\beta-1} \abs{\nabla w}^2 \, dx \right)^{\!\!(1-\theta)\frac{\mathfrak{q}}{2q'}}\!,
		\end{align*}
		by virtue of~\eqref{eq:tq'<1}. Finally, by Young's inequality and~\eqref{eq:eq-diff-final}, we conclude that
		\begin{equation}
			\label{eq:daiterare'}
			\int_\Omega \rho\eta^2 \,\abs{w}^{\beta-1}\,\abs{\nabla w}^2 \, dx \leq C \int_\Omega \left[\eta^2+\abs*{\nabla\eta}^2\right] \left(\abs{w}^{\beta+1}+(2\beta)^\beta\right) dx,
		\end{equation}
		where~$C \geq 1$ depends only on~$n$,~$p$,~$q$,~$C_1$,~$\norma*{c}_{L^\infty(B_{5R}(x_0))}$, as well as an upper bound for~$C_\mathcal{S}$. Particular care should be dedicated to the estimate of
		\begin{equation*}
			\begin{split}
				(2\beta)^\beta \int_\Omega \rho\eta^2 \,\abs{\nabla w}^2 \, dx &\leq (2\beta)^\beta K' \max\left\{1,(2C_1)^{p-2}\right\} \int_\Omega \left[\eta^2+\abs*{\nabla\eta}^2\right] dx \\
				&\leq C \int_\Omega \left[\eta^2+\abs*{\nabla\eta}^2\right] \ell \, dx,
			\end{split}
		\end{equation*}
		where we used~\eqref{eq:eq-diff-final} and the constant~$C \geq 1$ does not depend on~$R$. Now, assuming that~$R \leq h' < h'' \leq 4R$ and arguing as above, from~\eqref{eq:daiterare'}, we deduce that
		\begin{equation}
			\label{eq:A34}
			\Phi(\chi r,h',w) \leq \left(\frac{C_6 \, r}{h''-h'}\right)^{\!\! \frac{2}{r}} \left[\Phi(\chi r,h'',w)+2r\right]
		\end{equation}
		for some~$C_6 \geq 1$ depending on un upper bound for~$C_\mathcal{S}$. This is precisely (A.34) in~\cite{ds-har} or (5.22) in~\cite{trud-lin}.
		
		We now shall iterate~\eqref{eq:A34} in order to prove that
		\begin{equation}
			\label{eq:A35}
			\Phi(\mathfrak{p},3R,w) \leq C_\flat R^{\gamma(\mathfrak{p})} \left(\Phi(q,4R,w) + \mathfrak{p}\right) \quad \mbox{for every } \mathfrak{p} \geq \mathfrak{q},
		\end{equation}
		for some~$C_\flat \geq 1$, independent of~$\mathfrak{p}$, and some explicit function~$\gamma(\mathfrak{p})$ of~$\mathfrak{p}$. More precisely, we take
		\begin{equation*}
			r_k = \mathfrak{q}\chi^{k} \quad \mbox{ and } \quad h_k = 3R\left(1+\frac{1}{3 \, 2^{k}}\right).
		\end{equation*}
		Since~$\mathfrak{q}>2$, in view of this definition~$\beta_k = r_k -1 >1$. Thus, we have
		\begin{equation*}
			\begin{split}
				\Phi(\mathfrak{q}\chi^{k},h_{k},w) &\leq \left(\frac{\mathfrak{q} C_6 \, 2^k \chi^{k-1} }{2R}\right)^{\!\! \frac{2}{\mathfrak{q}\chi^{k}}} \left[\Phi(q\chi^{k-1},h_{k-1},w)+2\mathfrak{q}\chi^{k-1}\right] \\
				&\leq \prod_{j=0}^{k-1} \left(\frac{\mathfrak{q} C_6 \, 2^j \chi^{j-1} }{2R}\right)^{\!\! \frac{2}{\mathfrak{q}\chi^{j}}} \Phi(\mathfrak{q},4R,w) +2\mathfrak{q}\chi^k \sum_{j=0}^{k-1} \chi^{j-k} \prod_{l=j}^{k-1} \left(\frac{\mathfrak{q} C_6 \, 2^l \chi^{l-1} }{2R}\right)^{\!\! \frac{2}{\mathfrak{q}\chi^{l}}} \!.
			\end{split}
		\end{equation*}
		For convenience of notation, set
		\begin{equation*}
			\Sigma\left(k,j\right) \coloneqq \frac{2}{\mathfrak{q}} \sum_{l=j}^{k-1} \frac{1}{\chi^l} = \frac{2}{\mathfrak{q}} \frac{(1/\chi)^j-(1/\chi)^k}{1-1/\chi}
		\end{equation*}
		and observe that~$\Sigma\left(k,0\right) \geq \Sigma\left(k,j\right)$. Therefore, we have
		\begin{equation*}
			\begin{split}
				\Phi(\mathfrak{q}\chi^{k},h_{k},w) &\leq \frac{C}{R^{\Sigma\left(k,0\right)}} \,\Phi(\mathfrak{q},4R,w) +2q\chi^k C \sum_{j=0}^{k-1} \frac{\chi^{j-k}}{R^{\Sigma\left(k,j\right)}} \\
				&\leq \frac{C}{R^{\Sigma\left(k,0\right)}} \left[\Phi(\mathfrak{q},4R,w) +2\mathfrak{q}\chi^k \sum_{j=0}^{k-1} \chi^{j-k} \right] \leq \frac{C}{R^{\Sigma\left(k,0\right)}} \left[\Phi(\mathfrak{q},4R,w) +2\mathfrak{q}\chi^k \right]\!,
			\end{split}
		\end{equation*}
		where~$C \geq 1$ does not depend on~$k$, but only on~$n$,~$p$,~$q$,~$C_1$, and an upper bound for~$C_\mathcal{S}$. By setting~$k_\mathfrak{p} \coloneqq \min\left\{h \in \N \,\lvert\, \mathfrak{q}\chi^h \geq \mathfrak{p}\right\}$, we get
		\begin{equation*}
			\Phi(\mathfrak{p},3R,w) \leq \abs{B_1}^n \left(3R\right)^{\frac{n}{\mathfrak{p}}-\frac{n}{\mathfrak{q}\chi^{k_\mathfrak{p}}}} \,\Phi(\mathfrak{q}\chi^{k_\mathfrak{p}},h_{k_\mathfrak{p}},w) \leq \frac{C_\flat}{R^{\Sigma\left(k_\mathfrak{p},0\right)}} R^{\frac{n}{\mathfrak{p}}-\frac{n}{\mathfrak{q}\chi^{k_\mathfrak{p}}}} \left[\Phi(\mathfrak{q},4R,w) + \mathfrak{p} \right]\!.
		\end{equation*}
		This proves the validity of~\eqref{eq:A35} with~$\gamma(\mathfrak{p}) \coloneqq \frac{n}{\mathfrak{p}}-\frac{n}{q\chi^{k_\mathfrak{p}}} - \Sigma\left(k_\mathfrak{p},0\right)$.
		
		We now use~\eqref{eq:A35} to prove that~\eqref{eq:A18} holds true. In particular, we look for~$r_0 = \mathcal{C} R^\mu$, with~$\mu>0$ and~$\mathcal{C}>0$ to be shortly determined. Considering the power series expansion of~$e^{r_0\abs{w}}$, we get
		\begin{equation*}
			\int_{B_{5/2R\left(x_{0}\right)}} e^{r_0\abs{w}} \, dx \leq \sum_{\mathfrak{p}=0}^{+\infty} \int_{B_{3R(x_0)}} \frac{\left(r_0\abs{w}\right)^\mathfrak{p}}{\mathfrak{p}!} \, dx \leq \sum_{\mathfrak{p}=0}^{+\infty} \frac{\left(r_0 \Phi(\mathfrak{p},3R,w) \right)^\mathfrak{p}}{\mathfrak{p}!}.
		\end{equation*}
		Let~$a_\mathfrak{p}$ denote the general term of the preceding series. Then, taking advantage of~\eqref{eq:A24} and~\eqref{eq:A35}, we obtain
		\begin{equation}
			\label{eq:tgen}
			a_\mathfrak{p} \leq C_\flat^\mathfrak{p} r_0^\mathfrak{p} R^{\gamma(\mathfrak{p})\mathfrak{p}} \,\frac{\left(C_5+\mathfrak{p}\right)^\mathfrak{p}}{\mathfrak{p}!} \eqqcolon b_\mathfrak{p} \quad \mbox{for every } \mathfrak{p} \geq \mathfrak{q}.
		\end{equation}
		We shall apply the root test to show that, for a sufficiently small~$r_0$,~$b_\mathfrak{p}$ is the general term of a convergent series. Clearly~$k_\mathfrak{p} \to +\infty$ as~$\mathfrak{p} \to +\infty$, thus~$\Sigma\left(k_\mathfrak{p},0\right) \to \frac{2}{\mathfrak{q}-2}$ and consequently~$\gamma(\mathfrak{p}) \to -\frac{2}{\mathfrak{q}-2}$ as~$\mathfrak{p} \to +\infty$. Therefore
		\begin{equation*}
			\limsup_{\mathfrak{p} \to +\infty} b_\mathfrak{p}^{1/\mathfrak{p}} \leq \limsup_{\mathfrak{p} \to +\infty}  C_\flat r_0 R^{\gamma(\mathfrak{p})} \,\frac{C_5+\mathfrak{p}}{\left(\mathfrak{p}!\right)^{1/\mathfrak{p}}} = e C_\flat r_0 R^{-\frac{2}{\mathfrak{q}-2}},
		\end{equation*}
		where we used Stirling's formula. We can simply take
		\begin{equation*}
			\mu=\frac{2}{\mathfrak{q}-2} \quad \mbox{ and } \quad \mathcal{C} = \frac{1}{4C_\flat\max\left\{e,C_5\right\}}
		\end{equation*}
		to ensure the convergence of the series. By substituting this definition into~\eqref{eq:tgen}, it is straightforward to see that~$b_\mathfrak{p}$ can be further bounded above by the term of a convergent series, which is independent of~$R \leq 1$. Moreover, by~\eqref{eq:A24}, observe that
		\begin{align}
		\label{eq:est<q}
			r_0 \Phi(\mathfrak{p},3R,w) &\leq \mathcal{C} \Phi(\mathfrak{p},4R,w) \\
		\notag
			&\leq \mathcal{C} \min\!\left\{1,\abs{B_1}^n\right\} \left(4R\right)^{\frac{n}{\mathfrak{p}}-\frac{n}{\mathfrak{q}}} \Phi(\mathfrak{q},4R,w) \leq C C_5 \quad\mbox{for all } 1 \leq \mathfrak{p} \leq \mathfrak{q}.
		\end{align}
		Hence, there exists a~$C \geq 1$ independent of~$R$ such that
		\begin{equation*}
			\int_{B_{5/2R(x_0)}} e^{r_0\abs{w}} \, dx \leq C
		\end{equation*}
		and consequently,
		\begin{equation*}
			\int_{B_{5/2R(x_0)}} e^{r_0 w} \, dx \int_{B_{5/2R(x_0)}} e^{-r_0 w} \, dx \leq \left(\int_{B_{5/2R(x_0)}} e^{r_0\abs{w}} \, dx\right)^{\!\! 2} \leq C_\sharp,
		\end{equation*}
		which, in turn, implies~\eqref{eq:A18}.
		
		Since we precisely know~$r_0$, we can clarify the dependence on~$R$ in~\eqref{eq:finale-1}. Observe that
		\begin{equation*}
			\frac{1}{R^\frac{n}{s}} C_\sharp^{-\frac{1}{r_0}} R^{\frac{n}{r_0}} \mathcal{A}(r_0) \geq \frac{1}{R^\frac{n}{s}} \left(C_\sharp^{1/2} C_4\right)^{\!-\frac{2}{r_0}} R^{\frac{n+2\Sigma}{r_0}} \geq \frac{1}{ R^\frac{n}{s}} \left(c_\flat R\right)^{C_\natural/R^\mu} \!,
		\end{equation*}
		where~$c_\flat \leq 1$, up to possibly enlarging~$C_\sharp \geq 1$, and~$C_\natural \geq 1$.
		
		Therefore, recalling~\eqref{eq:stime-v}, we infer that there exists a constant~$C \geq 1$, depending on~$n$,~$p$,~$s$,~$q$,~$C_1$, and on an upper bound for~$C_\mathcal{S}$, such that
		\begin{equation}
			\label{eq:disfinale}
			\begin{split}
				\mathcal{M}(R) \norma{u_2-u_1}_{L^s(B_{2R}(x_0))} &\leq \mathcal{M}(R) \norma{v}_{L^s(B_{2R}(x_0))} \\
				&\leq C \inf_{B_{R}(x_0)} \! v = C \left(\inf_{B_{R}(x_0)} (u_2-u_1) +2\widehat{\mathsf{k}}\right) \!,
			\end{split}
		\end{equation}
		where~$R \leq 1$ and~$\mathcal{M}(R)$ is the function of the radius given by
		\begin{equation*}
			\mathcal{M}(R) \coloneqq  R^{-n/s} \left(c_\flat R\right)^{C_\natural/R^\frac{2}{\mathfrak{q}-2}},
		\end{equation*} 
		for some small universal~$c_\flat \in (0,1)$ and some large universal~$C_\natural \geq 1$. \newline

		Finally, we are left with the case~$r_0 < s < \chi$. Let~$N>1$ be the smallest integer such that~$r_1 = s \chi^{1-N} \leq r_0$ and consider
		\begin{equation*}
			r_{N-k} = s\chi^{-k} \quad \mbox{ and } \quad h_k = \frac{5}{4}R\left[1+\left(\frac{3}{5}\right)^{\!\! \frac{k}{N}}\right] \!.
		\end{equation*}
		Note that, correspondingly, we have~$\abs*{\beta_{N-k}} \geq \abs*{\beta_{N-1}} = 1-s/\chi$ for~$k\geq 1$. By virtue of the condition~$s<\chi$, we can iterate~\eqref{eq:dait->0} to get
		\begin{align*}
			\Phi(s,2R,v) &= \Phi(r_N,h_N,v) \leq \left(\frac{C_3 \left(1+r_{N-1}\right)}{\abs*{\beta_{N-1}}\left(h_{N-1}-h_{N}\right)}\right)^{\!\! \frac{2}{r_{N-1}}} \Phi(r_{N-1},h_{N-1},v) \\
			&\leq \prod_{k=1}^{N-1} \left[\frac{12 C_3}{5\abs*{\beta_{N-1}}R}\left(\frac{5}{3}\right)^{\!\! \frac{N+1-k}{N}}\right]^{\!\! \frac{2}{r_{N-k}}} \Phi(r_1,h_1,v) \\
			&\leq C^{\sum_{k=1}^{N-1}\frac{1}{r_{N-k}}} R^{-2\sum_{k=1}^{N-1}\frac{1}{r_{N-k}}} \left(\frac{5}{3}\right)^{\! 2 \sum_{k=1}^{N-1}\frac{N+1-k}{N}\frac{1}{r_{N-k}}} \Phi\left(r_1,\frac{5}{2}R,v\right) \!,
		\end{align*}
		for some~$C \geq 1$ independent of~$R$. A simple computation reveals that
		\begin{equation*}
			\sum_{k=1}^{N-1}\frac{1}{r_{N-k}} = \frac{1}{s} \frac{\chi^N-\chi}{\chi-1} \quad \mbox{ and } \quad  \sum_{k=1}^{N-1}\frac{N+1-k}{N}\frac{1}{r_{N-k}} \leq C \chi^N,
		\end{equation*}
		which imply that
		\begin{equation*}
			\Phi(s,2R,v) \leq C^{\chi^N} R^{-\frac{2\chi^N}{s(\chi-1)}} \,\Phi\left(r_1,\frac{5}{2}R,v\right) \!.
		\end{equation*}
		Since~$r_1 \leq r_0$ we can infer that~\eqref{eq:A18} and~\eqref{eq:A17} hold true with~$r_1$ in place of~$r_0$, so that
		\begin{equation*}
			\Phi(s,2R,v) \leq C_7^{\chi^N} R^{-\frac{2\chi^N}{s(\chi-1)}} \mathcal{A}(r_1)^{-1} \Phi(-\infty,R,v),
		\end{equation*}
		for some~$C_7 \geq 1$.
		
		Now, we make more explicit the dependence on~$R$. Recall that~$r_1 \leq r_0$ and observe that, by definition of~$N$, we have~$s\chi/ r_0 \leq \chi^N \leq s\chi^2/r_0$. Hence, it follows that~$r_0 \leq \chi r_1$ and
		\begin{equation*}
			\frac{1}{R^\frac{n}{s}} C_7^{-\chi^N} R^{\frac{2\chi^N}{s(\chi-1)}+\frac{n}{s}} \mathcal{A}(r_1) \geq \frac{1}{R^\frac{n}{s}} \left(C_7^{s\chi^2} C_4^{2\chi}\right)^{\! -\frac{1}{r_0}} R^{\left(\frac{2\chi^2}{\chi-1}+\Sigma\right) \frac{1}{r_0}} \geq \frac{1}{R^\frac{n}{s}} \left(c_\flat R\right)^{C_\natural/R^\mu}\!,
		\end{equation*}
		where~$c_\flat \leq 1$, up to eventually enlarging~$C_7 \geq 1$, $C_\natural \geq 1$, and assuming~$R \leq 1$. As above, we can now deduce~\eqref{eq:disfinale}.
	\end{proof}
	
	\begin{remark}
		\label{rem:cost-wk}
		By assuming that~$C_\mathcal{S} \geq 1$, a careful analysis of the previous proof reveals that we can make explicit the dependence on~$C_\mathcal{S}$. More precisely, it holds that
		\begin{equation*}
			\mathcal{M}(R) \coloneqq  R^{-n/s} \left(\frac{c_\flat R}{C_\mathcal{S}}\right)^{\! C_\natural C_\mathcal{S} \left(C_\mathcal{S}^2 / R \right)^\frac{2}{\mathfrak{q}-2}}\!.
		\end{equation*}
		Indeed, keeping track of~$C_\mathcal{S}$, we can take
		\begin{equation*}
			\mathcal{C} = \frac{1}{4C_\flat\max\left\{e,C_5\right\}C_\mathcal{S}^{2\mu+1}},
		\end{equation*}
		so that~$b_\mathfrak{p}$ can be bounded above by the general term of a convergent series, which is independent of~$C_\mathcal{S}$. Moreover, this allows to cancel out the dependence on~$C_\mathcal{S}$ in~\eqref{eq:est<q} caused by~\eqref{eq:A24}. In this way,~$C_\sharp$ does not depend on~$C_\mathcal{S}$ and we get the conclusion taking into account the factor that appears in~$\mathcal{A}(r_0)$.
		
		Moreover, we have~$\mu \geq 1$, that is~$\mathfrak{q} \leq 4$, for every~$n \geq 4$. Whereas, for~$n=3$, the same conclusion holds if we require that~$q > 4$ and~$2<\mathfrak{q}<4$.
	\end{remark}
	
	\subsection{Local boundedness comparison inequality.}
	
	We can also exploit Moser's iteration to establish the following result.
	
	\begin{theorem}
		\label{th:localbound}
		Let~$\Omega \subseteq \R^n$ be an open set, and let~$u_1,u_2\in C^1(\Omega)$ be a weak subsolution and a weak supersolution to~\eqref{eq:eqforu} with~$p > 2$, respectively. Moreover, suppose that~\eqref{eq:assunzione-pesi}-\eqref{eq:bound-gradiente} hold true.
		
		Assume that~$\overline{B_{5R}(x_0)} \subseteq \Omega$ for some~$R > 0$. Let~$q > 2$ be any real number satisfying~\eqref{eq:defi-q>}, and fix any~$2<\mathfrak{q}<2_M$, depending on~$n$,~$p$, and~$q$, such that~\eqref{eq:cond-q-0} holds true.
		
		Furthermore, suppose that~$c \in L^\infty(B_{5R}(x_0))$, that~$g_1(\cdot,u_1)-g_2(\cdot,u_2) \in L^q\!\left(B_{5R}(x_0)\right)$, and define~$\mathsf{k}\coloneqq\norma{g_1(\cdot,u_1)-g_2(\cdot,u_2)}_{L^q\left(B_{5R}(x_0) \, \cap \,\mathrm{supp} \left(u_1-u_2\right)_+\right)}$.
		
		Then, for every~$p_\sharp>1$, there exists a constant~$\mathfrak{C}_\sharp \geq 1$ such that
		\begin{equation*}
			\label{eq:locbound}
			\sup_{B_{R}(x_0)} \! (u_1-u_2) \leq \mathfrak{C}_\sharp \left(\frac{C_\mathcal{S}^2}{R}\right)^{\!\! \frac{2}{p_\sharp} \frac{\mathfrak{q}}{\mathfrak{q}-2}} \left( \norma{\left(u_1-u_2\right)_+}_{L^{p_\sharp}(B_{2R}(x_0))} + \mathsf{k} \right) \!,
		\end{equation*}
		where~$C_\mathcal{S} = C_\mathcal{S}(B_{5R}(x_0))$.
		
		The constant~$\mathfrak{C}_\sharp$ depends only on~$n$,~$p$,~$p_\sharp$,~$q$, $C_1$,~$\norma*{c}_{L^\infty(B_{5R}(x_0))}$, as well as on an upper bound on~$R$.
	\end{theorem}
	\begin{proof}
		Define
		\begin{equation*}
			v \coloneqq \left(u_1-u_2\right)_+ + \mathsf{k},
		\end{equation*}
		then take~$\eta \in C^1_c(B_{5R}(x_0))$, with~$\eta\geq 0$ in $\Omega$, and set
		\begin{equation*}
			\phi \coloneqq \eta^2 \left( v^\beta -\mathsf{k}^\beta \right) \geq 0
		\end{equation*}
		for~$\beta>0$, so that
		\begin{equation*}
			\nabla\phi=2\eta \left( v^\beta -\mathsf{k}^\beta \right) \nabla\eta + \beta\eta^2 v^{\beta-1} \,\nabla v.
		\end{equation*}
		Observe that~$\phi=0$ on the set~$\{x \in \Omega \,\lvert\, u_1(x) \leq u_2(x)\}$, so the integrals need to be evaluated only on the support of~$\left(u_1-u_2\right)_+$. Moreover, it holds that~$\phi \leq \eta^2 v^\beta$.
		
		By using~$\phi$ as test function in~\eqref{eq:eqforu} and setting~$f \coloneqq g_1\left(\cdot,u_1\right)-g_2\left(\cdot,u_2\right)$, we get
		\begin{align}
			\label{eq:eq-diff-11}
			\int_\Omega &\beta \eta^2 v^{\beta-1} \left\langle \abs*{\nabla u_1}^{p-2}\,\nabla u_1 - \abs*{\nabla u_2}^{p-2}\,\nabla u_2, \nabla v\right\rangle \\
			\notag
			&+ 2\eta \left( v^\beta -k^\beta \right) \left\langle \abs*{\nabla u_1}^{p-2}\,\nabla u_1 - \abs*{\nabla u_2}^{p-2}\,\nabla u_2, \nabla \eta\right\rangle dx \leq \int_\Omega c \left(u_2-u_1\right) \phi + f \phi \, dx.
		\end{align}
		Exploiting~\eqref{eq:disp-basso}-\eqref{eq:disp-alto}, we can estimate the terms on the left-hand side and the first one on the righ-hand side of~\eqref{eq:eq-diff-11} on~$\mathrm{supp} \left(u_1-u_2\right)_+$ as follows
		\begin{gather*}
			\left\langle \abs*{\nabla u_1}^{p-2}\,\nabla u_1 - \abs*{\nabla u_2}^{p-2}\,\nabla u_2, \nabla v\right\rangle \geq c(p) \rho \,\abs{\nabla v}^2, \\
			\abs*{2\eta \left( v^\beta -k^\beta \right) \left\langle \abs*{\nabla u_1}^{p-2}\,\nabla u_1 - \abs*{\nabla u_2}^{p-2}\,\nabla u_2, \nabla \eta\right\rangle} \leq 2C(p) \eta v^\beta \rho \,\abs{\nabla v}\abs{\nabla \eta}, \\
			\abs*{c \left(u_2-u_1\right) \phi} \leq \norma*{c}_{L^\infty(B_{5R}(x_0))} \eta^2 v^{\beta+1},
		\end{gather*}
		where~$\rho \coloneqq \left(\abs*{\nabla u_1}+\abs*{\nabla u_2} \right)^{p-2}$. To estimate the last integral we argue exactly as for~\eqref{eq:est-f<1}-\eqref{eq:intf-silim}, adopting the notation of the proof of Theorem~\ref{th:wh}, to get
		\begin{align*}
			\int_\Omega f \phi \, dx &\leq \int_\Omega \,\frac{\abs{f}}{v} \, \eta^2 v^{\beta+1} \, dx \\
			&\leq C_q^\frac{1}{q'} C_\mathcal{S}^{(1-\theta)\frac{\mathfrak{q}}{q'}} \int_\Omega \left[\eta^2+\rho\,\abs*{\nabla\eta}^2\right] v^{\beta+1} \, dx \\
			&\quad+ C_q^\frac{1}{q'} C_\mathcal{S}^{(1-\theta)\frac{\mathfrak{q}}{q'}} \left(\beta+1\right)^{(1-\theta)\frac{\mathfrak{q}}{q'}}  \left(\int_{\Omega} \eta^2 v^{\beta+1} dx \right)^{\!\!\frac{\theta}{q'}} \left( \int_{\Omega} \rho\eta^2 v^{\beta-1} \,\abs{\nabla v}^2 \, dx \right)^{\!\!(1-\theta)\frac{\mathfrak{q}}{2q'}}\!.
		\end{align*}
		Since we may assume that~$\beta \geq p_\sharp-1>0$,~\eqref{eq:eq-diff-11} yields
		\begin{equation*}
			\begin{split}
				\int_\Omega &\rho\eta^2 v^{\beta-1} \,\abs{\nabla v}^2 \, dx \leq \frac{2C(p)}{c(p)(p_\sharp-1)} \int_\Omega \rho \eta v^\beta \,\abs{\nabla v}\abs{\nabla \eta} \, dx + \frac{\norma*{c}_{L^\infty(B_{5R}(x_0))}}{c(p)(p_\sharp-1)} \int_\Omega \eta^2 v^{\beta+1} \, dx \\
				&+ \frac{C_q^\frac{1}{q'}}{c(p)(p_\sharp-1)} \, C_\mathcal{S}^{(1-\theta)\frac{\mathfrak{q}}{q'}} \int_\Omega \left[\eta^2+\rho\,\abs*{\nabla\eta}^2\right] v^{\beta+1} \, dx \\
				&+ C_q^\frac{1}{q'} C_\mathcal{S}^{(1-\theta)\frac{\mathfrak{q}}{q'}} \frac{p_\sharp^{(1-\theta)\frac{\mathfrak{q}}{q'}}}{c(p)(p_\sharp-1)} \left(\int_{\Omega} \eta^2 v^{\beta+1} \, dx \right)^{\!\!\frac{\theta}{q'}} \left( \int_{\Omega} \rho\eta^2 v^{\beta-1} \,\abs{\nabla v}^2 \, dx \right)^{\!\!(1-\theta)\frac{\mathfrak{q}}{2q'}}\!. 
			\end{split}
		\end{equation*}
		By applying Young's inequality, we conclude that
		\begin{equation}
			\label{eq:eq-diff-22}
			\int_\Omega \rho\eta^2 v^{\beta-1} \,\abs{\nabla v}^2 \, dx \leq 4K \max\left\{1, C_\mathcal{S}^2 \right\} \int_\Omega \left[\eta^2+\rho\,\abs*{\nabla\eta}^2\right] v^{\beta+1} \, dx,
		\end{equation}
		where~$K \geq 1$ depends only on~$n$,~$p$,~$p_\sharp$,~$q$, and~$\norma*{c}_{L^\infty(B_{5R}(x_0))}$. Define~$r \coloneqq \beta+1$ and
		\begin{equation*}
			w \coloneqq v^\frac{\beta+1}{2},
		\end{equation*}
		so that we can rewrite~\eqref{eq:eq-diff-22} as
		\begin{equation}
			\label{eq:eq-diff-final-2}
			\int_\Omega \rho\eta^2 \,\abs{\nabla w}^2 \, dx \leq
			K \max\left\{1, C_\mathcal{S}^{2} \right\} r^2 \int_\Omega \left[\eta^2+\rho\,\abs*{\nabla\eta}^2\right] w^2 \, dx.
		\end{equation}
		Since~$2<\mathfrak{q}<2_M$, by~\eqref{eq:sobolev-pesata} and~\eqref{eq:eq-diff-final-2}, we obtain
		\begin{equation*}
			\norma{\eta w}^2_{L^\mathfrak{q}(\Omega)} \leq C_2^2 \max\left\{1, C_\mathcal{S}^{2} \right\}^{\!2} r^2 \int_\Omega \left[\eta^2+\abs*{\nabla\eta}^2\right] w^2 \, dx,
		\end{equation*}
		for some~$C_2 \geq 1$ depending on~$n$,~$p$,~$p_\sharp$,~$q$,~$C_1$, and~$\norma*{c}_{L^\infty(B_{5R}(x_0))}$. Furthermore, we may suppose that~$C_\mathcal{S} \geq 1$, set~$\chi \coloneqq \mathfrak{q}/2>1$, and assume that~$R \leq h' < h'' \leq 5R$. By applying the iterative technique, we deduce that
		\begin{equation*}
			\Phi(\chi r,h',v) \leq \left(\frac{C_3 \, C_\mathcal{S}^{2} \, r}{h''-h'}\right)^{\!\! \frac{2}{r}} \Phi(r,h'',v),
		\end{equation*}
		where~$C_3 \geq 1$ depends on~$n$,~$p$,~$p_\sharp$,~$q$,~$C_1$,~$\norma*{c}_{L^\infty(B_{5R}(x_0))}$, as well as on an upper bound on~$R$. We now take
		\begin{equation*}
			r_k = p_\sharp \chi^k \quad \mbox{ and } \quad h_k = R\left(1+\frac{1}{2^{k}}\right),
		\end{equation*}
		so that~$\beta_k \geq r_k-1 \geq r_0-1 = p_\sharp-1$. Iterating, we get
		\begin{equation*}
			\begin{split}
				\Phi(r_{k+1},h_{k+1},v) &\leq \left(\frac{2^{k+1} C_3 \, C_\mathcal{S}^{2} \, r_k}{R}\right)^{\!\! \frac{2}{r_k}} \Phi(r_k,h_k,v) \\
				&\leq \left(2\chi\right)^{\sum_{j=0}^{k}\frac{j}{\chi^j}} \left(\frac{2 C_3 \, C_\mathcal{S}^2  \, p_\sharp}{R}\right)^{\!\! \frac{2}{p_\sharp} \sum_{j=0}^{k}\frac{1}{\chi^j}} \Phi(p_\sharp,2R,v)
			\end{split}
		\end{equation*}
		and, taking the limit for~$k \to +\infty$, we have
		\begin{equation*}
			\Phi(+\infty,R,v) \leq C \left(\frac{C_\mathcal{S}^2}{R}\right)^{\!\! \frac{2}{p_\sharp} \frac{\chi}{\chi-1}} \Phi(p_\sharp,2R,v),
		\end{equation*}
		for some~$C \geq 1$ depending on~$n$,~$p$,~$p_\sharp$,~$q$,~$C_1$,~$\norma*{c}_{L^\infty(B_{5R}(x_0))}$, as well as on an upper bound on~$R$. This concludes the proof.
	\end{proof}
	
	A straightforward adaptation of the previous proof, as pointed out in Theorem~8.25 of~\cite{gt}, allows us to obtain the local boundedness estimate at the boundary, namely, the following result.
	
	\begin{theorem}
		\label{th:localbound-bordo}
		Let~$\Omega \subseteq \R^n$ be an open set, and let~$u_1,u_2\in C^1(\overline{\Omega})$ be a weak subsolution and a weak supersolution to~\eqref{eq:eqforu} with~$p>2$, respectively. Moreover, suppose that~\eqref{eq:assunzione-pesi}-\eqref{eq:bound-gradiente} hold true.
		
		Assume that~$x_0 \in \R^n$, let~$q > 2$ be any real number satisfying~\eqref{eq:defi-q>}, and fix any~$2<\mathfrak{q}<2_M$, depending on~$n$,~$p$, and~$q$, such that~\eqref{eq:cond-q-0} holds true. Furthermore, suppose that~$c \in L^\infty\!\left(B_{5R}(x_0) \cap \Omega \right)$, that~$g_1(\cdot,u_1)-g_2(\cdot,u_2) \in L^q\!\left(B_{5R}(x_0) \cap \Omega \right)$, and set
		\begin{align*}
			M &\coloneqq \sup_{\partial\Omega \cap B_{2R}(x_0)} \! \left(u_1-u_2\right)_+, \\
			\mathsf{k} &\coloneqq\norma{g_1(\cdot,u_1)-g_2(\cdot,u_2)}_{L^q(B_{5R}(x_0) \cap \left\{x \in \Omega \,\lvert\, u_1-u_2>M\right\})}, \\
			\left(u_1-u_2\right)_M^+ &\coloneqq
			\begin{dcases}
				\sup\left\{u_1-u_2,M\right\}	& \quad \mbox{if } x \in \Omega, \\
				M								& \quad \mbox{if } x \not\in \Omega.
			\end{dcases}
		\end{align*}
		
		Then, for every~$p_\sharp>1$, there exists a constant~$\mathfrak{C}_\sharp \geq 1$ such that
		\begin{equation*}
			\sup_{B_{R}(x_0)} \!\left(u_1-u_2\right)_M^+ \leq \mathfrak{C}_\sharp \left(\frac{C_\mathcal{S}^2}{R}\right)^{\!\! \frac{2}{p_\sharp} \frac{\mathfrak{q}}{\mathfrak{q}-2}} \left( \norma{\left(u_1-u_2\right)_M^+}_{L^{p_\sharp}(B_{2R}(x_0))} + \mathsf{k} \right) \!,
		\end{equation*}
		where~$C_\mathcal{S} = C_\mathcal{S}(B_{5R}(x_0) \cap \Omega)$.
		
		The constant~$\mathfrak{C}_\sharp$ depends only on~$n$,~$p$,~$p_\sharp$,~$q$, $C_1$,~$\norma*{c}_{L^\infty(B_{5R}(x_0) \cap \Omega)}$, as well as on an upper bound on~$R$.
	\end{theorem}
	
	\begin{remark}
	\label{rem:forma-eq}
		A careful examination of their proofs reveals that Proposition~\ref{prop:wcp}, Theorem~\ref{th:wh}, Theorem~\ref{th:localbound}, and Theorem~\ref{th:localbound-bordo} hold whenever one has
		\begin{equation*}
			-\Delta_p u_1 + c \, u_1 -g_1(x,u_1) \leq -\Delta_p u_2 + c \, u_2 -g_2(x,u_2) \quad \mbox{in } \Omega.
		\end{equation*}
	\end{remark}
	
	
	\section{Preliminary results -- The case of the ball}
	\label{sec:prelim-bolla-plap}
	
	In this section and in Section~\ref{sec:proof-ball} below, we will refer to a constant as \textit{universal} if it depends only on~$n$,~$p$, $C_0$,~$\norma{f}_{C^{0,1}([0,C_0])}$,~$f$,~$\norma{\kappa}_{L^\infty(B_1)}$, and~$F$.
	
	\subsection{A priori estimates}
	
	First, by Theorem~1 in~\cite{lieb}, for a bounded weak solution of~\eqref{eq:mainprob} there exist an~$\alpha_\star \in (0,1)$, depending only on~$n$,~$p$,~$\norma{f}_{L^\infty([0,C_0])}$, and~$\norma{\kappa}_{L^\infty(B_1)}$, and a constant~$C_1 \geq 1$, depending only on~$n$,~$p$,~$C_0$,~$\norma{f}_{L^\infty([0,C_0])}$, and~$\norma{\kappa}_{L^\infty(B_1)}$, such that~$u \in C^{1,\alpha_\star}(\overline{B_1})$ with
	\begin{equation}
		\label{eq:bound-univ-gradiente}
		\norma{\nabla u}_{L^\infty(\overline{B_1})} + \left[\nabla u\right]_{C^{0,\alpha_\star}(\overline{B_1})} \leq C_1.
	\end{equation}
	
	Following~\cite{cicopepo}, we show that there exists a~$C_2 \geq 1$, depending only on~$n$,~$p$,~$C_0$, $\norma{f}_{L^\infty([0,C_0])}$,~$\norma{\kappa}_{L^\infty(B_1)}$, and~$F$, such that
	\begin{equation}
		\label{eq:ulinearbound}
		u(x) \geq \frac{1}{C_2} (1-\abs{x}) \quad \mbox{for all } x \in B_1.
	\end{equation}
	
	We start by claiming that the maximum of~$u$ is attained in a ball~$B_{1-\delta_1}$ for some~$\delta_1>0$, which depends only on~$C_0$ and~$C_1$. Indeed, take~$x \in \overline{B_1} \setminus B_{1-\delta_1}$ and denote with~$x_\star \in \partial B_1$ the point belonging to the same radius from the origin. Thus, by~\eqref{eq:bound-univ-gradiente}, we have
	\begin{equation*}
		u(x) = \abs*{u(x)-u(x_\star)} \leq C_1 \abs{x-x_\star} \leq C_1 \delta_1.
	\end{equation*}
	Choosing~$\delta_1 < \frac{1}{C_0C_1}$, we deduce
	\begin{equation*}
		u(x) < \frac{1}{C_0} \quad \mbox{for all } x \in \overline{B_1} \setminus B_{1-\delta_1}.
	\end{equation*}
	
	Now, we seek for a lower bound on~$u$ in~$\overline{B_{1-\delta_1}}$. Here, the technical assumption~\eqref{eq:ipotesif} plays a central role. By Theorems~5,~7, and~9 in~\cite{serr}, there exists a~$C_3 \geq 1$, depending on~$n$,~$p$,~$C_0$,~$\norma{f}_{L^\infty([0,C_0])}$,~$\norma{\kappa}_{L^\infty(B_1)}$, and~$F$, such that
	\begin{equation}
		\label{eq:serrin-harnack}
		\frac{1}{C_0} \leq \sup_{B_{1-\delta_1}} u \leq C_3 \inf_{B_{1-\delta_1}} u
	\end{equation}
	which, in turn, implies
	\begin{equation}
		\label{eq:bound-basso-u-B1d}
		u \geq \frac{1}{C_0 C_3} \eqqcolon c_\flat \quad \mbox{in } \overline{B_{1-\delta_1}}.
	\end{equation}
	
	Finally, we construct an appropriate barrier in~$B_1 \setminus \overline{B_{1-\delta_1}}$. Define the radial function
	\begin{equation*}
		\Psi(r) \coloneqq B\left(e^{\beta(1-r)}-1\right),
	\end{equation*}
	for some~$B,\beta>0$ and~$1-\delta_1 \leq r \leq 1$, and note that~$\Psi(1)=0$. Since~$\Psi$ is radial, we can compute
	\begin{equation*}
		\Delta_p \Psi = B \beta \,\abs*{\Psi'}^{p-2} \left( (p-1)\beta-\frac{n-1}{r} \right) e^{\beta(1-r)} \geq 0,
	\end{equation*}
	for every~$1-\delta_1 \leq r \leq 1$, provided~$\beta \geq \frac{n-1}{(p-1)(1-\delta_1)}$. Moreover, let~$B>0$ be such that~$\Psi(1-\delta_1)=c_\flat$. Therefore, we have
	\begin{equation*}
		\Delta_p u \leq 0 \leq \Delta_p \Psi \quad \text{in} \; B_1 \setminus \overline{B_{1-\delta_1}}
	\end{equation*}
	with~$u=\Psi=0$ on~$\partial B_1$, and~$\Psi \leq u$ on~$\partial B_{1-\delta_1}$. Hence, the weak comparison principle of Theorem~2.4.1 in~\cite{ps} implies
	\begin{equation}
		\label{eq:Psi<u}
		\Psi \leq u \quad \mbox{in } B_1 \setminus \overline{B_{1-\delta_1}}.
	\end{equation}
	Exploiting this, we deduce that
	\begin{equation*}
		u(x) \geq c_\star \left(1-\abs{x}\right) \quad \mbox{for all } x \in \overline{B_1} \setminus B_{1-\delta_1},
	\end{equation*}
	for some~$c_\star>0$ which depends on~$n$,~$p$,~$C_0$, ~$\norma{f}_{L^\infty([0,C_0])}$,~$\norma{\kappa}_{L^\infty(B_1)}$, and~$F$. The last estimate, together with~\eqref{eq:bound-basso-u-B1d}, leads to the validity of~\eqref{eq:ulinearbound}.
	
	\begin{remark}
		Assumption~\eqref{eq:ipotesif} is merely technical and enables us to exploit the Harnack inequality in~\cite{serr} and deduce~\eqref{eq:serrin-harnack}. Thus, once one has an estimate of the type~\eqref{eq:serrin-harnack}, the above procedure can be carried out analogously to get~\eqref{eq:ulinearbound}.
	\end{remark}
	
	In addition, this procedure allows us to estimate the interior normal derivative of~$u$ at a boundary point in terms of universal constants. Indeed, let~$x \in \partial B_1$ and~$\nu$ be the inner normal vector at~$x$. Since~$u(x)=\Psi(x)=0$, by~\eqref{eq:Psi<u},
	\begin{equation*}
		\label{eq:stimadernormale}
		\partial_\nu u(x) \geq \partial_\nu \Psi(x) = - \partial_r \Psi(x) = \beta B>0.
	\end{equation*}
	
	This, together with~\eqref{eq:bound-univ-gradiente}, allows us to quantify the distance between~$\partial B_1$ and the set of critical points~$\mathcal{Z}_u$. Assume~$\delta_2 \in (0,1)$, take~$x \in \overline{B_1} \setminus B_{1-\delta_2}$, and determine~$x_\star \in \partial B_1$ as before, then
	\begin{equation}
		\label{eq:grad>0}
		\beta B-\abs*{\nabla u(x)} \leq \abs*{\nabla u(x_\star)}-\abs*{\nabla u(x)} \leq \left[\nabla u\right]_{C^{0,\alpha_\star}(\overline{B_1})}  \abs{x-x_\star}^{\alpha_\star} \leq C_1 \delta_2^{\alpha_\star}.
	\end{equation}
	Thus, by taking
	\begin{equation*}
		\delta_2 < \left(\frac{\beta B}{C_1}\right)^{\!\!\frac{1}{\alpha_\star}} \eqqcolon \delta_0,
	\end{equation*}
	we get~$\abs*{\nabla u(x)}>0$. Consequently, with this choice,~$\mathcal{Z}_u \Subset B_{1-\delta_2}$.
	
	\subsection{Summability property of the gradient.}
	
	In order to apply the results of Section~\ref{sec:tectool}, we need to prove that if~$u \in C^1(\overline{B_1})$ is a weak solution to~\eqref{eq:mainprob} for~$p>2$, then~$\rho=\abs*{\nabla u}^{p-2}$ fulfills~\eqref{eq:condizione-peso}, at least for~$\gamma=0$.
	
	This is essentially Theorem~2.3 in~\cite{ds-poinc}. Although the poof is already available, we need a precise quantification of the constant appearing therein. Since this will be enough for our purposes, we state the result for~$\gamma=0$.
	
	Finally, we are going to assume that
	\begin{equation}
		\label{eq:k-below}
		\kappa \geq \underline{\kappa} >0 \quad \mbox{a.e.\ in } B_1.
	\end{equation}
	
	\begin{theorem}
	\label{th:int-grad}
		Let~$u\in C^1(\overline{B_1})$ be a weak solution to~\eqref{eq:mainprob}. Assume that~\eqref{eq:mainhyp-f}-\eqref{eq:ipotesif} and~\eqref{eq:k-below} are in force. Then, for every~$r \in (0,1)$, we have
		\begin{equation*}
			\label{eq:integrabilita-bolla}
			\int_{B_1}  \frac{1}{\abs*{\nabla u}^{(p-1)r}} \, dx \leq \mathscr{C}
		\end{equation*}
		for some~$\mathscr{C}>0$ depending only on~$n$,~$p$,~$r$,~$C_0$,~$\norma{f}_{C^{0,1}([0,C_0])}$,~$f$,~$\norma{\kappa}_{L^\infty(B_1)}$,~$\underline{\kappa}$, and~$F$.
	\end{theorem}
	
	For the sake of completeness, we provide more details about the proof of this result in Appendix~\ref{ap:proofsum}. \newline
	
	When~$p>2$, we can choose~$p-2<(p-1)r<p-1$ and apply Theorem~\ref{th:int-grad} to get that~\eqref{eq:condizione-peso} is fulfilled with~$\Omega = B_1$,~$\gamma=0$, and for some constant~$C^\ast>0$ depending only on~$n$,~$p$,~$C_0$,~$\norma{f}_{C^{0,1}([0,C_0])}$,~$f$,~$\norma{\kappa}_{L^\infty(B_1)}$,~$\underline{\kappa}$, and~$F$. In addition, we will shortly see that the dependence on~$\underline{\kappa}$ can actually be omitted in our case.
	
	From this discussion, it follows that when~$u_i=u$ for~$i=1,2$, then~\eqref{eq:assunzione-pesi} holds and we can apply Proposition~\ref{prop:wcp}, Theorem~\ref{th:wh}, Theorem~\ref{th:localbound}, and Theorem~\ref{th:localbound-bordo}.
	
	
	\section{Proof of Theorem~\ref{th:mainth}}
	\label{sec:proof-ball}
	
	First, note that we may suppose, without loss of generality, that
	\begin{equation*}
		\label{eq:defi-small}
		\defi(\kappa) \leq \gamma
	\end{equation*}
	for some small~$\gamma \in (0,1)$ to be determined later in dependence of universal constants. Indeed, if~$\defi(\kappa) > \gamma$, thanks to~\eqref{eq:boundu} one has
	\begin{equation*}
		\begin{split}
			\abs*{u(x)-u(y)} &\leq \abs*{u(x)} + \abs*{u(y)} \leq 2C_0 \\
			&\leq \frac{2C_0}{\abs*{\log\left(C \gamma\right)}^{-\alpha}} \,\abs*{\log\left(C \defi(\kappa)\right)}^{-\alpha} \quad \mbox{for all } x,y \in B_1.
		\end{split}
	\end{equation*}
	Hence,~\eqref{eq:qsym-toprove} follows trivially.
	
	Moreover, by~\eqref{eq:mainhyp-k} and the definition of the deficit~\eqref{eq:deficit}, we have
	\begin{equation*}
		\underline{\kappa} \geq \norma{\kappa}_{L^\infty(B_1)} - \defi(\kappa) \geq \frac{1}{2} \norma{\kappa}_{L^\infty(B_1)}
	\end{equation*}
	as long as $\gamma \leq \gamma_0 \coloneqq 1/2 \,\norma{\kappa}_{L^\infty(B_1)}$. Thus,~\eqref{eq:k-below} holds true.
	
	Finally, observe that, up to a rotation, it suffices to show that
	\begin{equation}
		\label{eq:mainclaim-bolla}
		u(x',x_n)-u(x',-x_n) \leq C \, \abs*{\log\left(C \defi(\kappa)\right)}^{-\alpha} \quad \mbox{for all } x \in B_1 \mbox{ such that } x_n>0,
	\end{equation}
	for some universal constants~$C \geq 1$ and~$\alpha \in (0,1)$.
	
	\renewcommand\thesubsection{\bfseries Step \arabic{subsection}}
	
	\subsection{Starting the moving planes procedure.}
	\label{sub:step1}
	
	Let~$\lambda \in (0,1)$. With usual notation, we write
	\begin{equation*}
		\Sigma_\lambda \coloneqq \left\{x \in B_1 \,\lvert\, x_n > \lambda\right\} \quad \mbox{ and } \quad x^\lambda \coloneqq \left(x',2\lambda-x_n\right),
	\end{equation*}
	for every~$x = (x', x_n) \in B_1$ and
	\begin{equation*}
		u_\lambda(x) \coloneqq u(x^\lambda) \quad \mbox{for } x \in \Sigma_\lambda.
	\end{equation*}
	Note that~$u_\lambda$ is a weak solution to
	\begin{equation*}-\Delta_p u_\lambda = \kappa_\lambda f(u_\lambda) \quad \mbox{in } \Sigma_\lambda,
	\end{equation*}
	where~$\kappa_\lambda(x) \coloneqq \kappa(x^\lambda)$. By defining
	\begin{equation*}
		c_{\lambda} (x) \coloneqq
		\begin{dcases}
			\frac{f(u(x))-f(u(x^{\lambda}))}{u(x)-u(x^{\lambda})}	& \quad \mbox{if } u(x) \neq u(x^{\lambda}), \\
			0													& \quad \mbox{if } u(x) = u(x^{\lambda}),
		\end{dcases}
	\end{equation*}
	and
	\begin{equation*}
		g_{1,\lambda} \coloneqq \kappa f(u) \quad\mbox{and}\quad g_{2,\lambda} \coloneqq \kappa_\lambda f(u),
	\end{equation*}
	we have that
	\begin{equation}
		\label{eq:eq-for-u-ul-ball}
		\begin{split}
			-\Delta_p u - \kappa_\lambda \, c_\lambda u - g_{1,\lambda} = -\Delta_p u_\lambda - \kappa_\lambda  \, c_\lambda u_\lambda - g_{2,\lambda} \quad \mbox{in } \Sigma_{\lambda}.
		\end{split}
	\end{equation}
	Additionally,~$c_\lambda \in L^\infty(\Sigma_{\lambda})$ with~$\abs*{c_\lambda}  \leq \norma{f}_{C^{0,1}([0,C_0])}$ and~$g_{i,\lambda} \in L^\infty(\Sigma_\lambda)$ for~$i=1,2$. Hence, in view of Proposition~\ref{prop:wcp} and Remark~\ref{rem:forma-eq}, there exist a small universal~$\delta_\star \in (0,1)$ and a universal constant~$\mathscr{K} \geq 1$ such that
	\begin{equation*}
		\norma{(u-u_\lambda)_{+}}_{L^\infty(\Sigma_\lambda)} \leq \mathscr{K} \,\norma{g_{1,\lambda}-g_{2,\lambda}}_{L^\infty(\Sigma_\lambda)} \leq \mathscr{K} \,\norma{f}_{C^{0,1}([0,C_0])} \defi(\kappa),
	\end{equation*}
	provided~$\abs*{\Sigma_\lambda} \leq \delta_\star$. Moreover, note that the last condition holds, for instance, if~$\lambda \geq 1-\delta_\star \abs*{B_1'}^{-1}$.
	
	Define now~$C_3 \coloneqq \max\left\{\mathscr{K},\sqrt{\pi^3}/2\right\} \norma{f}_{C^{0,1}([0,C_0])}$ and the set
	\begin{equation*}
		\Lambda \coloneqq \left\{\lambda \in (0,1) \,\lvert\, \norma{(u-u_\mu)_{+}}_{L^\infty(\Sigma_\mu)} \leq C_3 \defi(\kappa) \mbox{ in } \Sigma_\mu \mbox{ for all } \mu \in [\lambda,1) \right\}.
	\end{equation*}
	We have just showed that~$[\lambda_0,1) \subseteq \Lambda$, for~$\lambda_0 \coloneqq \max\left\{1/2,1-\delta_\star \abs*{B_1'}^{-1}\right\}$. Thus,
	\begin{equation*}
		\lambda_\star \coloneqq \inf \Lambda
	\end{equation*}
	is a well-defined real number and~$\lambda_\star \in [0,\lambda_0]$.
	
	\subsection{Reaching the critical position.}
	
	We claim that
	\begin{equation}
		\label{eq:est-last}
		\lambda_\star \leq \lambda_{1} \coloneqq \lambda_{1}(\defi(\kappa)),
	\end{equation}
	for some~$\lambda_{1}(\defi(\kappa)) \to 0$ as~$\defi(\kappa) \to 0$ to be determined shortly.
	
	We first prove that
	\begin{equation*}
		\lambda_\star \leq \frac{1}{4}.
	\end{equation*}
	To this aim, we argue by contradiction and suppose instead that~$\lambda_\star \in \left( \frac{1}{4} , \lambda_0 \right]$.
	
	We shall exploit the lower bound~\eqref{eq:ulinearbound} to find a small region where~$u_{\lambda_\star}-u$ is positive. Next, by taking advantage of the fundamental weak Harnack inequality of Theorem~\ref{th:wh} and arguing as in~\cite{cicopepo}, we build a suitable Harnack chain to propagate this information to a larger region. Ultimately, we will derive from this a contradiction.
	
	First, observe that, by continuity, it holds~$u-u_{\lambda_\star} \leq C_3 \defi(\kappa)$ in~$\Sigma_{\lambda_\star}$. Given~$\delta \in (0,1)$, define
	\begin{equation*}
		\Sigma_{\lambda_\star,\delta} \coloneqq \left\{ x \in \Sigma_\lambda \,\lvert\, \mathrm{dist}(x,\partial\Sigma_{\lambda_\star}) > \delta \right\}.
	\end{equation*}
	Let~$d\leq 1/2$ be the inradius of~$\Sigma_{\lambda_\star}$,~$p_0 \in \Sigma_{\lambda_\star}$ be a point such that~$\mathrm{dist}(p_0,\partial\Sigma_{\lambda_\star})=d$, and~$p \in \Sigma_{\lambda_\star,\delta}$. We may assume~$\delta < d$, so that~$p_0 \in \Sigma_{\lambda_\star,\delta}$.
	
	In order to exploit Theorem~\ref{th:wh}, we first arbitrarily fix~$q>2$, depending only on~$n$ and~$p$, such that~\eqref{eq:defi-q>} holds and then choose~$2<\mathfrak{q}<2_M$, depending only on~$n$ and~$p$, such that~\eqref{eq:cond-q-0} holds. Next, recalling also the definition of~$C_3$, we observe that
	\begin{equation*}
		\norma{g_{1,\lambda}-g_{2,\lambda}}_{L^q(\Sigma_\lambda)} \leq \abs{B_1}^\frac{1}{q} \norma{f}_{C^{0,1}([0,C_0])} \defi(\kappa) \leq C_3 \defi(\kappa),
	\end{equation*}
	so that~\eqref{eq:positivita} is verified. Observe also that, in light of~\eqref{eq:eq-for-u-ul-ball}, the application of Theorem~\ref{th:wh} is justified by  Remark~\ref{rem:forma-eq}.
	
	We now fix an arbitrary~$s \in (0,1)$ and proceed as in Lemma~2.2 of~\cite{cicopepo} considering, for shortness, only the most involved case. To this end, we consider the chain of balls~$\{B_k'\}_{k=0}^N$ with radii~$r_k'$ -- constructed as in~\cite{cicopepo} -- and compare the left-hand side of~\eqref{eq:weak-harnak-fond} on two consecutive balls. For convenience, we factorize~$\mathcal{M}(R) = R^{-n/s} \, m(R)$, observe that~$m(R) \leq 1$ for~$R \leq 1$, and compute
	\begin{equation*}
		\begin{split}
			\mathcal{M}(r'_{k-1}) &\left(\int_{B'_{k-1}} \,\abs*{u_{\lambda_\star}-u}^s \, dx\right)^{\!\! \frac{1}{s}} \leq \mathfrak{C} \left(\inf_{B'_{k-1}} (u_{\lambda_\star}-u) +2C_3 \defi(\kappa) \right) \\
			&\leq \mathfrak{C} \left\{\frac{1}{\abs{B'_{k-1} \cap B'_{k}}^{1/s}} \left(\int_{B'_{k}} \,\abs*{u_{\lambda_\star}-u}^s \, dx\right)^{\!\! \frac{1}{s}} +2C_3 \defi(\kappa) \right\} \\
			&\leq \left(\frac{2^n \,\mathfrak{C}^s}{\min\left\{1,\abs{B_1}\right\}}\right)^{\!\!\frac{1}{s}} \left\{\frac{\mathcal{M}(r'_{k})}{m(r'_{k})} \left(\int_{B'_{k}} \,\abs*{u_{\lambda_\star}-u}^s \, dx\right)^{\!\! \frac{1}{s}} +2C_3 \defi(\kappa) \right\}.
		\end{split}
	\end{equation*}
	Iterating along the chain, we deduce
	\begin{align}
		\label{eq:it-wh-1}
		\mathcal{M}\left(\frac{d}{5}\right) &\left(\int_{B_\frac{d}{5}(p_0)} \,\abs*{u_{\lambda_\star}-u}^s \, dx\right)^{\!\! \frac{1}{s}} \leq \left(\frac{2^{n+1} \,\mathfrak{C}^s}{\min\left\{1,\abs{B_1}\right\}}\right)^{\!\!\frac{N}{s}} \frac{1}{m(r'_1)\dots m(r'_N)} \cdot \\
		\notag
		&\cdot \left\{\mathcal{M}(r'_{N}) \left(\int_{B'_{N}} \,\abs*{u_{\lambda_\star}-u}^s \, dx\right)^{\!\!\frac{1}{s}} +2C_3 \defi(\kappa) \right\}.
	\end{align}
	Using a similar argument, we also obtain
	\begin{equation}
		\label{eq:it-wh-2}
		\mathcal{M}(r'_{N})  \left(\int_{B'_{N}} \,\abs*{u_{\lambda_\star}-u}^s \, dx\right)^{\!\!\frac{1}{s}} \leq \frac{C}{m(\delta/5)} \left(\inf_{B_{\frac{\delta}{5}}(p)} (u_{\lambda_\star}-u) +2C_3 \defi(\kappa)\right) \!.
	\end{equation}
	for some~$C \geq 1$ depending only on~$n$,~$s$, and~$\mathfrak{C}$. As a consequence, chaining~\eqref{eq:it-wh-1} and~\eqref{eq:it-wh-2}, we conclude that
	\begin{equation*}
		\begin{split}
			\mathcal{M}\left(\frac{d}{5}\right) &\left(\int_{B_\frac{d}{5}(p_0)} \,\abs*{u_{\lambda_\star}-u}^s \, dx\right)^{\!\! \frac{1}{s}} \\
			&\leq C \left(\frac{2^{n+1} \,\mathfrak{C}^s}{\min\left\{1,\abs{B_1}\right\}}\right)^{\!\!\frac{N}{s}} \frac{1}{m(r'_1)\dots m(r'_N)m(\delta/5)} \left(\inf_{B_{\frac{\delta}{5}}(p)} (u_{\lambda_\star}-u) +2C_3 \defi(\kappa)\right) \!.
		\end{split}
	\end{equation*}
	for some~$C \geq 1$ depending only on~$n$,~$s$, and~$\mathfrak{C}$.
	
	Since we can derive an analogous estimate starting from~$B_{\delta/5}(p)$ and reaching~$B_{d/5}(p_0)$, we ultimately deduce that
	\begin{equation}
		\label{eq:wh-chainato}
		\begin{split}
			\sup_{p \in \Sigma_{\lambda_\star,\delta}} &\left(\dashint_{B_\frac{\delta}{5}(p)} \,\abs*{u_{\lambda_\star}-u}^s \, dx\right)^{\!\! \frac{1}{s}} \\
			&\leq \frac{C}{m(\delta/5)^{2N+3}} \left(\frac{2^{n+1}\,\mathfrak{C}^s}{\min\left\{1,\abs{B_1}\right\}}\right)^{\!\!\frac{2N}{s}} \left(\inf_{\Sigma_{\lambda_\star,\delta}} (u_{\lambda_\star}-u) + 2 C_3 \defi(\kappa)\right) \!,
		\end{split}
	\end{equation}
	for some universal constant~$C \geq 1$. Given that~$d \geq c_\sharp \,\mathrm{diam}(\Sigma_{\lambda_\star})$ for some~$c_\sharp \in (0,1)$, by Lemma~2.2 in~\cite{cicopepo}, we also have
	\begin{equation*}
		N \leq C \log\left(\frac{d}{\delta}\right) \leq -C \log\delta,
	\end{equation*}
	for some~$C>0$ depending only on~$c_\sharp$. Furthermore, we may assume that~$\delta \leq e^{-3}$ and deduce, from~\eqref{eq:wh-chainato}, that
	\begin{equation}
	\label{eq:wh-final}
		\sup_{p \in \Sigma_{\lambda_\star,\delta}} \left(\dashint_{B_\frac{\delta}{5}(p)} \,\abs*{u_{\lambda_\star}-u}^s \, dx\right)^{\!\!\frac{1}{s}} 
		\leq \frac{C_\flat}{\delta^{\beta}} \, m(\delta/5)^{\beta \log\delta} \left(\inf_{\Sigma_{\lambda_\star,\delta}} (u_{\lambda_\star}-u) + 2 C_3 \defi(\kappa)\right) \!,
	\end{equation}
	where~$C_\flat \geq 1$,~$\beta>0$ depend only on~$n$,~$p$,~$C_0$,~$\norma{f}_{C^{0,1}([0,C_0])}$,~$f$,~$\norma{\kappa}_{L^\infty(B_1)}$, and~$F$. Indeed, note that~$c_\sharp$ can be chosen to depend on~$\lambda_0$ and, consequently, only on universal constants.
	
	Recalling the estimate~\eqref{eq:ulinearbound}, the gradient bound~\eqref{eq:bound-univ-gradiente}, and the Dirichlet boundary datum, we observe that
	\begin{align*}
		\sup_{p \in \Sigma_{\lambda_\star,\delta}} &\left(\dashint_{B_\frac{\delta}{5}(p)} \,\abs*{u_{\lambda_\star}-u}^s \, dx\right)^{\!\!\frac{1}{s}} \geq \left(\dashint_{B_\frac{\delta}{5}\left((1-2\delta)e_n\right)} \,\abs*{u_{\lambda_\star}-u}^s \, dx\right)^{\!\!\frac{1}{s}} \geq \inf_{B_\frac{\delta}{5}\left((1-2\delta)e_n\right)} \left(u_{\lambda_\star}-u\right) \\
		&\geq \frac{1}{C_2} \min\left\{1-\abs*{2\lambda_\star-1+\frac{11}{5}\delta}, 1-\abs*{2\lambda_\star-1+\frac{9}{5}\delta}\right\} - 4 \norma{\nabla u}_{L^\infty(B_1)} \delta \\
		&\geq \frac{1-\lambda_0}{2C_2} - 4 C_1 \delta
	\end{align*}
	for every~$\delta \in \left(0,\min\left\{e^{-3},\frac{1-\lambda_0}{2}\right\}\right]$. As a result, by~\eqref{eq:wh-final}, we deduce that
	\begin{equation}
		\label{eq:inf-analog}
		\inf_{\Sigma_{\lambda_\star,\delta}} (u_{\lambda_\star}-u) \geq \frac{\delta^{\beta}}{C_\flat} \, m(\delta/5)^{-\beta \log\delta} \left( \frac{1-\lambda_0}{2C_2} - 4 C_1 \delta \right) - 2 C_3 \defi(\kappa).
	\end{equation}
	We take~$\delta \coloneqq \min\left\{e^{-3}, \frac{1-\lambda_0}{16C_1C_2}, \frac{\delta_\star}{6(n+2)}\right\}$, where~$\delta_\star>0$ is given by Proposition~\ref{prop:wcp}, and require
	\begin{equation*}
		\gamma \leq \gamma_1 \coloneqq \frac{(1-\lambda_0)\delta^{\beta}}{16C_2C_3C_\flat} \, m(\delta/5)^{-\beta \log\delta},
	\end{equation*}
	so that
	\begin{equation*}
		u_{\lambda_\star}-u \geq \frac{(1-\lambda_0)\delta^{\beta}}{8C_2C_\flat} \, m(\delta/5)^{-\beta \log\delta} \quad \mbox{in } \Sigma_{\lambda_\star,\delta}.
	\end{equation*}
	Therefore, we conclude that
	\begin{equation}
		\label{eq:posit-Slambdel}
		\begin{split}
			\left(u_{\lambda_\star-\epsilon}-u\right)(x) &\geq \left(u_{\lambda_\star}-u\right)(x) -2\norma{\nabla u}_{L^\infty(B_1)}  \epsilon \\
			&\geq  \frac{(1-\lambda_0)\delta^{\beta}}{8C_2C_\flat} \, m(\delta/5)^{-\beta \log\delta} -2C_1 \epsilon \geq 0 \quad \mbox{for all } x \in \Sigma_{\lambda_\star,\delta},
		\end{split}
	\end{equation}
	provided that~$\epsilon>0$ is sufficiently small.
	
	Define now the open set~$\Sigma' \coloneqq \Sigma_{\lambda_\star-\epsilon} \setminus \overline{\Sigma_{\lambda_\star,\delta}}$ and note that~$u \leq u_{\lambda_\star-\epsilon}$ on~$\partial\Sigma'$. Precisely as in~\ref{sub:step1}, we deduce that
	\begin{equation}
		\label{eq:est-Sigma'}
		\norma{(u-u_{\lambda_\star-\epsilon})_{+}}_{L^\infty(\Sigma')} \leq C_3 \defi(\kappa),
	\end{equation}
	provided~$\abs*{\Sigma'} \leq \delta_\star$. By noticing that
	\begin{equation*}
		\Sigma' \subseteq \left(\Sigma_{\lambda_\star-\epsilon} \setminus \Sigma_{\lambda_\star+\delta}\right) \cup \left((B_1 \setminus B_{1-\delta}) \cap \Sigma_{\lambda_\star+\delta} \right) \!,
	\end{equation*}
	recalling the definition of~$\delta$, and provided~$\epsilon \leq \delta$, we have
	\begin{equation*}
		\abs*{\Sigma'} \leq \abs*{B_1'}(\epsilon+\delta) + \abs*{B_1 \setminus B_{1-\delta}} \leq 6(n+2)\delta \leq \delta_\star.
	\end{equation*}
	Hence,~\eqref{eq:posit-Slambdel}-\eqref{eq:est-Sigma'} and the continuity of~$u-u_{\lambda_\star-\epsilon}$ lead to a contradiction to the definition of~$\lambda_\star$. \newline
	
	We can now push~$\lambda_\star$ closer to the origin and prove the validity of~\eqref{eq:est-last}. By contradiction, assume that
	\begin{equation*}
		\lambda_\star \in \left(\lambda_{1}(\defi(\kappa)), \frac{1}{4} \right] \!.
	\end{equation*}
	Observe that this is possible for~$\gamma$ sufficiently small, as we will specify shortly. Arguing as in~\eqref{eq:inf-analog}, we obtain
	\begin{equation*}
		\begin{split}
			\inf_{\Sigma_{\lambda_\star,\delta}} (u_{\lambda_\star}-u) &\geq \frac{\delta^{\beta}}{C_\flat} \, m(\delta/5)^{-\beta \log\delta} \left( \frac{\lambda_\star}{2C_2} - 4 C_1 \delta \right) - 2 C_3 \defi(\kappa) \\
			&\geq \frac{\delta^{\beta}}{4C_1C_2C_\flat} \, m(\delta/5)^{-\beta \log\delta} \lambda_\star - 2 C_3 \defi(\kappa).
		\end{split}
	\end{equation*}
	Note that we require~$\delta \leq \min\left\{e^{-3},  \frac{\lambda_\star}{16C_1C_2}, \frac{\delta_\star}{6(n+2)}\right\}$. Then, a straightforward computation reveals that
	\begin{equation}
		\label{eq:inf-toest}
		\inf_{\Sigma_{\lambda_\star,\delta}} (u_{\lambda_\star}-u) \geq \frac{\delta^{\beta}}{4C_1C_2C_\flat} e^{-\frac{C_4}{\delta^\tau}} \lambda_\star - 2 C_3 \defi(\kappa)
	\end{equation}
	for some universal~$C_4 \geq 1$ and some~$\tau \geq 2$ depending only on~$n$ and~$p$. We now set~$\delta \coloneqq \frac{\lambda_{1}}{C_5} \in \left(0,\frac{1}{8}\right]$, where~$C_5 \coloneqq \max\left\{4C_1C_2, \frac{3(n+2)}{2\delta_\star}, e^3\right\} \geq 4$, ensuring that all the previous requirements on~$\delta$ are satisfied. Substituting this definition in~\eqref{eq:inf-toest} and noting that there exists a universal~$c_\ast>0$ such that
	\begin{equation*}
		t^{\beta+1} \geq e^{- C_4 C_5^\tau t^{-\tau}} \quad\mbox{for }t \leq c_\ast,
	\end{equation*}
	we obtain
	\begin{equation}
		\label{eq:for-remark}
		\begin{split}
			\inf_{\Sigma_{\lambda_\star,\delta}} (u_{\lambda_\star}-u) &\geq \frac{\lambda_{1}^{\beta+1}}{4C_1 C_2 C_5^\beta C_\flat} e^{- C_4 C_5^\tau \lambda_{1}^{-\tau}} - 2 C_3 \defi(\kappa) \\
			&\geq \frac{1}{4C_1 C_2 C_5^\beta C_\flat} e^{- 2 C_4 C_5^\tau \lambda_{1}^{-\tau}} - 2 C_3 \defi(\kappa) = C_3 \defi(\kappa).
		\end{split}
	\end{equation}
	The last identity follows by choosing
	\begin{equation}
		\label{eq:l1-def}
		\lambda_{1}(\defi(\kappa)) \coloneqq (2C_4)^{1/\tau}C_5 \,  \abs*{\log\left(12 C_1 C_2 C_3 C_5^\beta C_\flat \defi(\kappa)\right)}^{-1/\tau} \!.
	\end{equation}
	Consequently,~$\lambda_{1}(\defi(\kappa)) \to 0$ as~$\defi(\kappa) \to 0$. Moreover, we have~$\lambda_{1}(\defi(\kappa)) < \min\{1/4,c_\ast\}$, provided that we restrict ourselves to
	\begin{equation*}
		\gamma \leq \gamma_2 \coloneqq \frac{1}{24 C_1 C_2 C_3 C_5^\beta C_\flat} \exp\left\{ -\frac{2C_4C_5^\tau}{\min\{4^{-1},c_\ast\}^\tau} \right\}.
	\end{equation*}
	For~$\epsilon>0$ sufficiently small, we deduce that~$u_{\lambda_\star-\epsilon}-u \geq 0$ in~$\Sigma_{\lambda_\star,\delta}$. As before, Proposition~\ref{prop:wcp} ensures that~\eqref{eq:est-Sigma'} is verified thanks to the definitions of~$\delta$ and~$C_5$. Hence, we get a contradiction to the definition of~$\lambda_\star$ and~\eqref{eq:est-last} holds true.
	
	\begin{remark}
		\label{rem:wh-logtype}
		From~\eqref{eq:inf-toest} and~\eqref{eq:for-remark}, which follow from the weak Harnack inequality, it is clear that we expect a logarithmic-type behavior for~$\lambda_{1}$, as any power-like behavior would not suffice to establish positivity of the left-hand side of~\eqref{eq:for-remark}.
	\end{remark}
	
	\subsection{Almost radial symmetry in one direction.}
	
	Up to now, we have proved that
	\begin{equation*}
		u(x',x_n)-u(x',2\lambda-x_n) \leq C_3 \defi(\kappa) \quad \mbox{for all } (x',x_n) \in \Sigma_\lambda \mbox{ and } \lambda \in [\lambda_1,1),
	\end{equation*}
	with $\lambda_1$ given by~\eqref{eq:l1-def}. By taking~$\lambda=\lambda_1$ and exploiting the bound~\eqref{eq:bound-univ-gradiente}, we deduce
	\begin{equation*}
		\begin{split}
			u(x',x_n)-u(x',-x_n) &\leq u(x',x_n)-u(x',2\lambda_1-x_n) + \abs*{u(x',2\lambda_1-x_n)-u(x',-x_n)} \\
			&\leq C_3 \defi(\kappa) + 2 \norma{\nabla u}_{L^\infty(B_1)} \,\lambda_{1} \leq C_6 \lambda_{1} \quad \mbox{for all } (x',x_n) \in \Sigma_{\lambda_1},
		\end{split}
	\end{equation*}
	with~$C_6 \coloneqq 2C_1 + C_3$. In addition, for~$(x',x_n) \in \Sigma_0 \setminus \Sigma_{\lambda_1}$, we have
	\begin{equation*}
		\begin{split}
			u(x',x_n)-u(x',-x_n) &\leq \abs*{u(x',x_n)-u(x',0)} + \abs*{u(x',0)-u(x',-x_n)} \\
			&\leq 2 \norma{\nabla u}_{L^\infty(B_1)} \, x_n \leq 2C_1 \lambda_{1} \leq C_6 \lambda_{1}.
		\end{split}
	\end{equation*}
	The last two inequalities establish the validity of~\eqref{eq:mainclaim-bolla}.
	
	
	\section{Preliminary results -- The case of the space}
	
	In this section and in Section~\ref{sec:proof-space} below, we will refer to a constant as \textit{universal} if it depends only on~~$n$,~$p$,~$C_0$, $c_1$,~$\norma{\kappa}_{L^\infty(\R^n)}$,~$\underline{\kappa}$, and~$R_0$.
	
	\renewcommand\thesubsection{\arabic{section}.\arabic{subsection}}
	
	\subsection{A priori estimates}	
	
	First, by Theorem~1 in~\cite{diben}, for a weak solution of~\eqref{eq:mainprob-crit} there exists a constant~$C_1 \geq 1$, depending only on~$n$,~$p$, and~$C_0$, such that
	\begin{equation}
		\label{eq:C1boundonu}
		\norma*{u}_{C^1(\R^n)} \leq C_1.
	\end{equation}
	Moreover, arguing as in Theorem~1.1 of~\cite{vet}, one can prove, up to possibly enlarging~$C_1$, that
	\begin{equation}
		\label{eq:bound-gradu-plap}
		\abs*{\nabla u(x)} \leq \frac{C_1}{1+\abs*{x}^\frac{n-1}{p-1}} \quad \mbox{for all } x \in \R^n.
	\end{equation}
	
	Secondly, we show that any non-trivial solution to~\eqref{eq:mainprob-crit} has strictly positive mass. The following result is essentially Lemma~2.3 in~\cite{vet}.
	
	\begin{lemma}
		Suppose that~$1<p<n$ and~$u \in \mathcal{D}^{1,p}(\R^n)$ is a non-negative, non-trivial weak solution to~\eqref{eq:mainprob-crit}. Then, there exists a constant~$\mathcal{M}>0$, depending only on~$n$,~$p$, and~$\norma{\kappa}_{L^\infty(\R^n)}$, such that
		\begin{equation}
			\label{eq:boundmassa}
			\norma*{u}_{L^{\past}\!(\R^n)} \geq \mathcal{M}.
		\end{equation}
	\end{lemma}
	\begin{proof}
		Since~$u$ is non-trivial, by Sobolev inequality, we get
		\begin{equation*}
			\begin{split}
				0 &< \int_{\R^n} u^{\past} dx \leq S^{-\past} \left(\int_{\R^n} \,\abs*{\nabla u}^p \, dx\right)^{\!\!\frac{\past}{p}} = S^{-\past} \left(\int_{\R^n} \kappa u^{\past} dx\right)^{\!\!\frac{\past}{p}} \\
				&\leq S^{-\past} \norma{\kappa}_{L^\infty(\R^n)}^{\past/p} \left(\int_{\R^n} u^{\past} dx\right)^{\!\!\frac{\past}{p}}\!.
			\end{split}
		\end{equation*}
		The conclusion immediately follows since~$\past>p$.
	\end{proof}
	
	Finally, we shall prove that any solution to~\eqref{eq:mainprob-crit} enjoys a universal pointwise lower bound. We remark that, under more general assumptions than ours, such a bound for non-negative solutions to~\eqref{eq:mainprob-crit} has been already established by V\'etois in Theorem~1.1 of~\cite{vet}. Here, we provide a proof that fits our assumptions and relies only elementary tools.
	
	\begin{lemma}
		Suppose that~$1<p<n$ and~$u \in \mathcal{D}^{1,p}(\R^n)$ is a non-negative, non-trivial weak solution to~\eqref{eq:mainprob-crit} satisfying~\eqref{eq:decadimento}. Then, there exists a constant~$c_0>0$, depending only on~$n$,~$p$,~$C_0$, and~$\norma{\kappa}_{L^\infty(\R^n)}$, such that
		\begin{equation}
			\label{eq:bb-u}
			u(x) \geq \frac{c_0}{1+\abs{x}^\frac{n-p}{p-1}} \quad \mbox{for every } x \in \R^n.
		\end{equation}
	\end{lemma}
	\begin{proof}
		Let~$r >0$ to be determined shortly and define~$B_r \coloneqq B_r(0)$. Taking advantage of~\eqref{eq:decadimento}, we have
		\begin{equation*}
			\int_{\R^n \setminus B_{r}} u^{\past} dx \leq C_0^{\past} \int_{\R^n \setminus B_{r}} \,\abs*{x}^{-\frac{np}{p-1}} \, dx = \frac{(p-1)C_0^{\past} \Haus^{n- 1}(\sfera^{n - 1})}{n} \, r^{-\frac{n}{p-1}} = \frac{\mathcal{M}^{\past}}{2},
		\end{equation*}
		for 
		\begin{equation*}
			r \coloneqq \left(\frac{2(p-1)C_0^{\past} \Haus^{n- 1}(\sfera^{n - 1})}{n \mathcal{M}^{\past}}\right)^{\!\!\frac{p-1}{n}} \!,
		\end{equation*}
		which depends only on~$n$,~$p$,~$C_0$, and~$\norma{\kappa}_{L^\infty(\R^n)}$. Therefore, using~\eqref{eq:boundmassa}, we deduce that
		\begin{equation}
			\label{eq:bb-Br}
			\int_{B_{r}} u^{\past} dx \geq \frac{\mathcal{M}^{\past}}{2}.
		\end{equation}
		Since~$\Delta_p u \leq 0$ weakly in~$\R^n$, by applying Theorem~7.1.2 in~\cite{ps}, we conclude that for any~$s \in \left(0,(p-1)\past/p\right)$, there exists a constant~$\mathfrak{c}>0$, depending only on~$n$,~$p$, and~$s$, such that
		\begin{equation}
			\label{eq:wh-ubb}
			r^{-n/s} \norma*{u}_{L^s(B_{r})} \leq \mathfrak{c} \inf_{B_{r}} u.
		\end{equation}
		Now, we fix some~$s$ as above and observe that
		\begin{equation*}
			\int_{B_{r}} u^{\past} dx \leq C_0^{\past-s} \int_{B_{r}} u^{s} \, dx. 
		\end{equation*}
		Hence, the latter, together with~\eqref{eq:bb-Br},~\eqref{eq:wh-ubb}, and the definition of~$r$, implies that
		\begin{equation}
			\label{eq:bb-int-Br}
			\inf_{B_{r}} u \geq \mathfrak{c}_\sharp
		\end{equation}
		for some~$\mathfrak{c}_\sharp>0$, depending only on~$n$,~$p$,~$C_0$, and~$\norma{\kappa}_{L^\infty(\R^n)}$.
		
		Define~$\Psi(x) \coloneqq \mathfrak{c}_\sharp \, r^{\frac{n-p}{p-1}} \abs*{x}^{-\frac{n-p}{p-1}}$. Then,~$\Psi$ is a weak solution to~$\Delta_p \Psi = 0$ in~$\R^n \setminus \{0\}$, hence, we have
		\begin{equation*}
			\Delta_p u \leq 0 \leq \Delta_p \Psi \quad \mbox{in } \R^n \setminus \overline{B_{r}}.
		\end{equation*}
		Furthermore, by~\eqref{eq:bb-int-Br} and the continuity of~$u$, it follows that~$\Psi \leq u$ on~$\partial B_r$. Clearly, $\liminf_{\abs{x} \to +\infty} (u-\Psi)(x) =0$, thus the weak comparison principle of Theorem~2.4.1 in~\cite{ps} entails that
		\begin{equation*}
			\Psi \leq u \quad \mbox{in } \R^n \setminus \overline{B_{r}}.
		\end{equation*}
		This, together with~\eqref{eq:bb-int-Br}, leads to the validity of~\eqref{eq:bb-u}.
	\end{proof}
	
	\subsection{Summability property of the gradient.}
	
	In our proof, we will also need the counterpart of Theorem~\ref{th:int-grad}, stated as follows.
	
	\begin{theorem}
		\label{th:int-grad-sp}
		Let~$u\in C^1(\R^n)$ be a weak solution to~\eqref{eq:mainprob-crit}. Assume that~\eqref{eq:k-below-sp},~\eqref{eq:decadimento}, and~\eqref{eq:bb-grad} are in force. Then, for every~$r \in (0,1)$ and~$R\geq R_0$, we have
		\begin{equation*}
			\label{eq:integrabilita-spazio}
			\int_{B_{\! R}}  \frac{1}{\abs*{\nabla u}^{(p-1)r}} \, dx \leq \mathscr{C}_{\! R}
		\end{equation*}
		for some~$\mathscr{C}_R>0$ depending only on~$n$,~$p$,~$r$,~$C_0$,~$c_1$,~$\norma{\kappa}_{L^\infty(\R^n)}$,~$\underline{\kappa}$,~$R_0$, and~$R$.
	\end{theorem}
	
	In Appendix~\ref{ap:proofsum}, we point out how to modify the proof of Theorem~\ref{th:int-grad} in order to prove the last result.
	
	
	\section{Proof of Theorem~\ref{th:plap-spazio}}
	\label{sec:proof-space}
	
	The argument follows the structure of the proof of Theorem~1.1 in~\cite{ccg}. The differences are mainly technical, with some already appearing in the proof of the symmetry result in~\cite{sciu}, while others arise in~\ref{step:int2point-plap} and~\ref{step:prop-posi-plap} below. For these reasons, some details are omitted or abbreviated compared to the corresponding proof in~\cite{ccg}.
	
	
	\renewcommand\thesubsection{\bfseries Step \arabic{subsection}}
	
	\subsection{Preliminary observations.}
	
	We start by noticing that, by~\eqref{eq:k-below-sp}, the assumption~$\kappa \in L^\infty(\R^n)$, the non-triviality, and non-negativity of~$u$ the reference constant~$\kappa_0$ given in~\eqref{eq:kappa0-def} is well-defined and
	\begin{equation*}
		\underline{\kappa} \leq \kappa_0 \leq \norma{\kappa}_{L^\infty(\R^n)}.
	\end{equation*}
	
	Next, we observe that it suffices to prove Theorem~\ref{th:plap-spazio} when~$\defi(u,\kappa)$ is smaller than a universal constant~$\gamma \in \left( 0, \frac{1}{2} \right]$. Indeed, if~$\defi(u,\kappa) > \gamma$, by~\eqref{eq:C1boundonu} and~\eqref{eq:bound-gradu-plap}, we have
	\begin{equation*}
		\abs*{u(x)-u(y)} \leq \abs*{u(x)} + \abs*{u(y)} \leq 2C_1 \leq \frac{2C_1}{\abs{\log \gamma}^{-\vartheta}} \,\abs{\log\defi(u,\kappa)}^{-\vartheta}
	\end{equation*}
	for every~$x,y \in \R^n$
	and
	\begin{equation*}
		\norma*{u-u_\Theta}^p_{\mathcal{D}^{1,p}(\R^n)} \leq 2^p C_1^p \int_{\R^n} \left(1 + \abs*{x}^{\frac{n-1}{p-1}}\right)^{\!\!-p} dx \leq \frac{2^p C_1^p C}{\abs{\log \gamma}^{-\vartheta}} \,\abs{\log\defi(u,\kappa)}^{-\vartheta}.
	\end{equation*}
	Therefore, both inequalities stated in Theorem~\ref{th:plap-spazio} hold trivially. Hence, we assume that
	\begin{equation*}
		\defi(u,\kappa) \leq \gamma,
	\end{equation*}
	for some small~$\gamma \in \left( 0, \frac{1}{2} \right]$ to be determined later in terms of universal quantities.
	
	In the remainder of the proof we establish the validity of the quasi-symmetry statements of Theorem~\ref{th:plap-spazio}. To this end, given a direction~$\omega \in \sfera^{n-1}$ and~$\lambda \in \R$, we define the sets
	\begin{equation*}
		\Sigma_{\omega, \lambda} \coloneqq \Big\{ {x \in \R^n \, \big\lvert \, \left\langle \omega , x \right\rangle > \lambda} \Big\}, \quad
		T_{\omega, \lambda} \coloneqq \partial \Sigma_{\omega, \lambda} = \Big\{ {x \in \R^n \, \big\lvert \, \left\langle \omega , x \right\rangle = \lambda} \Big\},
	\end{equation*}
	and denote the reflection of a point~$x \in \R^n$ across the hyperplane~$T_{\omega, \lambda}$ by
	\begin{equation*}
		x^{\omega, \lambda} \coloneqq x + 2 \left(\lambda - \left\langle \omega , x \right\rangle\right) \omega.
	\end{equation*}
	Finally, we define
	\begin{equation*}
		u_{\omega, \lambda}(x) \coloneqq u(x^{\omega, \lambda}) \quad \mbox{for } x \in \R^n.
	\end{equation*}
	Our goal in Steps~2-6 is to show that for every~$\omega \in \sfera^{n - 1}$, there is a number~$\lambda_\star = \lambda_\star(\omega) \in \R$ for which
	\begin{equation} \label{eq:mainMPclaim-plap}
		\begin{gathered}
			u(x) - u_{\omega, \lambda}(x) \le C_\sharp \abs*{\log \defi(u,\kappa)}^{-\alpha_\sharp} \quad \mbox{for every } x \in \Sigma_{\omega, \lambda} \mbox{ and } \lambda \ge \lambda_\star, \\
			\norma*{u - u_{\omega, \lambda_\star}}_{L^\infty(\R^n)} + \norma*{\nabla (u - u_{\omega, \lambda_\star})}_{L^p(\R^n)} \le C_\sharp \abs*{\log \defi(u,\kappa)}^{- \alpha_\star},
		\end{gathered}
	\end{equation}
	for some universal constants~$C_\sharp \ge 1$,~$ 0 < \alpha_\star \leq \alpha_\sharp$, and provided that~$\gamma$ is smaller than a universal constant~$\gamma_{\sharp} \in \left( 0, \frac{1}{2} \right]$.
	
	
	\subsection{Starting the moving planes procedure.}
	\label{sub:Mpstarts-plap}
	
	After a rotation, it suffices to verify~\eqref{eq:mainMPclaim-plap} for~$\omega = e_n$. Under this assumption, we drop any reference to the dependence on~$\omega$ and simply write
	\begin{gather*}
		\Sigma_\lambda \coloneqq \Big\{ {x \in \R^n \, \big\lvert \, x_n > \lambda} \Big\}, \quad
		T_\lambda \coloneqq \Big\{ {x \in \R^n \, \big\lvert \, x_n = \lambda} \Big\}, \\
		x^\lambda \coloneqq (x', 2 \lambda - x_n), \quad \mbox{and} \quad u_\lambda(x) \coloneqq u(x^\lambda),
	\end{gather*}
	for every~$x = (x', x_n) \in \R^n$.
	
	We claim here that
	\begin{equation} \label{eq:MPstarts}
		\norma*{(u-u_\lambda)_+} _{L^{\past}\!(\Sigma_\lambda)} \leq \mathscr{K} \defi(u,\kappa)^{\frac{1}{p-1}} \quad \mbox{for all } \lambda \ge \lambda_1,
	\end{equation}
	for some universal constants~$\lambda_1 > 0$, depending only on~$n$,~$p$,~$C_0$,~$c_1$,~$\norma{\kappa}_{L^\infty(\R^n)}$, and~$R_0$, and~$\mathscr{K} \geq 1$, depending only on~$n$ and~$p$.
	
	In order to verify~\eqref{eq:MPstarts}, note that, without loss of generality, we may assume
	\begin{equation} \label{eq:L2gradnormnot0}
		\norma*{\nabla (u-u_\lambda)_+}_{L^p(\Sigma_\lambda)} > 0.
	\end{equation}
	Otherwise, we would have~$(u - u_\lambda)+ = 0$ in~$\Sigma_\lambda$, in which case~\eqref{eq:MPstarts} holds trivially.
	
	Assuming~\eqref{eq:L2gradnormnot0}, by a density argument, we can test equation~\eqref{eq:mainprob-crit} against the function~$\left(u-u_\lambda\right)_+ \chi_{\Sigma_\lambda}$. By doing so, we obtain
	\begin{equation*}
		\int_{\Sigma_\lambda} \abs*{\nabla u}^{p-2} \left\langle \nabla u, \nabla (u-u_\lambda)_+ \right\rangle \, dx = \int_{\Sigma_\lambda} \kappa u^{\past-1} (u-u_\lambda)_+ \, dx.
	\end{equation*}
	Since~$u_\lambda$ satisfies the same equation as~$u$ with~$\kappa$ replaced by~$\kappa_\lambda$, we also get
	\begin{equation*}
		\int_{\Sigma_\lambda} \abs*{\nabla u_\lambda}^{p-2} \left\langle \nabla u_\lambda, \nabla (u-u_\lambda)_+ \right\rangle \, dx = \int_{\Sigma_\lambda} \kappa_\lambda u^{\past-1}_\lambda (u-u_\lambda)_+ \, dx.
	\end{equation*}
	Subtracting these two identities, we deduce
	\begin{align}
		\notag
		\int_{\Sigma_\lambda} &\left\langle \abs*{\nabla u}^{p-2} \,\nabla u - \abs*{\nabla u_\lambda}^{p-2} \,\nabla u_\lambda, \nabla (u-u_\lambda)_+ \right\rangle \, dx \\
		\label{eq:test-finale}
		&= \int_{\Sigma_\lambda} \left(\kappa u^{\past-1} - \kappa_\lambda u^{\past-1}_\lambda\right) (u-u_\lambda)_+ \, dx \\
		\notag
		&= \int_{\Sigma_\lambda} \left(\kappa-\kappa_0 \right) u^{\past-1} (u-u_\lambda)_+  \, dx + \int_{\Sigma_\lambda} \left(\kappa_0 -\kappa_\lambda\right) u_\lambda^{\past-1} (u-u_\lambda)_+  \, dx \\
		\notag
		&\quad+ \kappa_0 \int_{\Sigma_\lambda} \left(u^{\past-1}-u^{\past-1}_\lambda\right) (u-u_\lambda)_+ \, dx.
	\end{align}
	We now estimate each term on the right-hand side of~\eqref{eq:test-finale}.
	
	For the first summand we use H\"older and Sobolev inequalities in order to get
	\begin{equation}
		\label{eq:est-rs1}
		\begin{split}
			\int_{\Sigma_\lambda} \left(\kappa-\kappa_0\right) u^{\past-1} (u-u_\lambda)_+ \, dx &\leq \norma*{\left(\kappa-\kappa_0\right) u^{\past-1}}_{L^{(\past)'}\!(\Sigma_\lambda)} \,\norma*{ (u-u_\lambda)_+}_{L^{\past}\!(\Sigma_\lambda)} \\
			&\leq S^{-1} \defi(u,\kappa) \,\norma*{\nabla (u-u_\lambda)_+}_{L^{p}(\Sigma_\lambda)}.
		\end{split}
	\end{equation}
	The second term can be handled analogously. As for the last one, we may assume that~$\lambda \geq 1$, apply the mean value theorem, and exploit~\eqref{eq:decadimento} to deduce
	\begin{equation*}
		\begin{split}
			\kappa_0 \int_{\Sigma_\lambda} \left(u^{\past-1}-u^{\past-1}_\lambda\right) &(u-u_\lambda)_+ \, dx \leq \left(\past-1\right) \kappa_0 \int_{\Sigma_\lambda} u^{\past-2} (u-u_\lambda)_+^2 \, dx \\
			&\leq \left(\past-1\right) \kappa_0 \, C_0^{\past-2} \int_{\Sigma_\lambda} \frac{1}{\abs*{x}^{(\past-2)\frac{n-p}{p-1}}} (u-u_\lambda)_+^2 \, dx.
		\end{split}
	\end{equation*}
	We now proceed in the estimate as in Step~1 of the proof of Theorem~3.1 in~\cite{sciu}. To this end, define
	\begin{equation*}
		s_\ast \coloneqq - \frac{n-1}{p-1} (p-2) > 2-n \quad \mbox{and} \quad \beta^\ast \coloneqq \left(\past-2\right) \frac{n-p}{p-1}+s_\ast-2 >0,
	\end{equation*}
	so that we can apply the weighted Hardy inequality of Lemma~2.3 in~\cite{dam-ram}. This only requires that the function vanishes at infinity. Indeed, by observing that~$\abs*{x} \geq \lambda$ in~$\Sigma_\lambda$ and assuming also~$\lambda \geq R_0$, we have
	\begin{align}
		\notag
		\int_{\Sigma_\lambda} \frac{1}{\abs*{x}^{(\past-2)\frac{n-p}{p-1}}} (u-u_\lambda)_+^2 \, dx &\leq \frac{1}{\lambda^{\beta^\ast}} \int_{\Sigma_\lambda} \,\abs*{x}^{s_\ast-2} (u-u_\lambda)_+^2 \, dx \\
		\label{eq:est-rs2}
		&\leq \frac{1}{\lambda^{\beta^\ast}} \left(\frac{2}{n+s_\ast-2}\right)^{\!2} \int_{\Sigma_\lambda} \,\abs*{x}^{s_\ast} \abs*{\nabla (u-u_\lambda)_+}^2 \, dx \\
		\notag
		&\leq \frac{1}{c_1^{p-2}\lambda^{\beta^\ast}} \left(\frac{2}{n+s_\ast-2}\right)^{\!2} \int_{\Sigma_\lambda} \,\abs*{\nabla u}^{p-2} \,\abs*{\nabla (u-u_\lambda)_+}^2 \, dx,
	\end{align}
	where we used~\eqref{eq:bb-grad} in the last line.
	
	Going back to~\eqref{eq:test-finale}, we exploit~\eqref{eq:disp-basso} to deduce that
	\begin{equation*}
		\begin{split}
			c(p)& \int_{\Sigma_\lambda} \left(\abs*{\nabla u}+\abs*{\nabla u_\lambda}\right)^{p-2} \abs*{\nabla (u-u_\lambda)_+}^2 \, dx \\
			&\leq \int_{\Sigma_\lambda} \left\langle \abs*{\nabla u}^{p-2} \nabla u - \abs*{\nabla u_\lambda}^{p-2} \nabla u_\lambda, \nabla (u-u_\lambda)_+ \right\rangle dx.
		\end{split}
	\end{equation*}
	From this, together with~\eqref{eq:est-rs1},~\eqref{eq:est-rs2}, and the fact that~$p>2$, we infer
	\begin{align*}
		\int_{\Sigma_\lambda} &\left(\abs*{\nabla u}+\abs*{\nabla u_\lambda}\right)^{p-2} \abs*{\nabla (u-u_\lambda)_+}^2 \, dx \\
		&\leq \frac{\left(\past-1\right) \norma{\kappa}_{L^\infty(\R^n)} \, C_0^{\past-2}}{c(p) c_1^{p-2}\lambda^{\beta^\ast}} \left(\frac{2}{n+s_\ast-2}\right)^{\!2} \int_{\Sigma_\lambda} \left(\abs*{\nabla u}+\abs*{\nabla u_\lambda}\right)^{p-2} \,\abs*{\nabla (u-u_\lambda)_+}^2 \, dx \\
		&\quad +2S^{-1} c(p)^{-1} \defi(u,\kappa) \,\norma*{\nabla (u-u_\lambda)_+}_{L^{p}(\Sigma_\lambda)}.
	\end{align*}
	Hence,
	\begin{multline}
		\label{eq:est-interm}
		\int_{\Sigma_\lambda} \left(\abs*{\nabla u}+\abs*{\nabla u_\lambda}\right)^{p-2} \abs*{\nabla (u-u_\lambda)_+}^2 \, dx \\
		\leq 4S^{-1} c(p)^{-1} \defi(u,\kappa) \norma*{\nabla (u-u_\lambda)_+}_{L^{p}(\Sigma_\lambda)},
	\end{multline}
	provided~$\lambda \geq \lambda_1$ with
	\begin{equation*}
		\label{eq:lambda1def}
		\lambda_1 \coloneqq \max\left\{R_0, \left(\frac{2\left(\past-1\right) \norma{\kappa}_{L^\infty(\R^n)} \, C_0^{\past-2}}{c(p) \, c_1^{p-2}} \left(\frac{2}{n+s_\ast-2}\right)^{\!2}\right)^{\! \frac{1}{\beta^\ast}}\right\}.
	\end{equation*}
	Since~$p>2$, from~\eqref{eq:est-interm}, we also obtain
	\begin{equation*}
		\norma*{\nabla (u-u_\lambda)_+}_{L^{p}(\Sigma_\lambda)}^p \leq 4S^{-1} c(p)^{-1} \defi(u,\kappa) \,\norma*{\nabla (u-u_\lambda)_+}_{L^{p}(\Sigma_\lambda)}.
	\end{equation*}
	Finally, recalling that~\eqref{eq:L2gradnormnot0} is in force, we conclude that
	\begin{equation*}
		\norma*{\nabla (u-u_\lambda)_+}_{L^{p}(\Sigma_\lambda)}^{p-1} \leq C \defi(u,\kappa),
	\end{equation*}
	for some~$C\geq 1$ depending only on~$n$ and~$p$. A further application of Sobolev inequality leads to the validity of~\eqref{eq:MPstarts}.
	
	Estimate~\eqref{eq:MPstarts} implies that the set
	\begin{equation*}
		\Lambda \coloneqq \Big\{ \lambda \in \R \, \big\lvert \, \norma*{ (u-u_\mu)_+}_{L^{\past}\!(\Sigma_\mu)} \leq \mathscr{K} \defi(u,\kappa)^{\frac{1}{p-1}} \,\, \mbox{ for every } \mu \geq \lambda \Big\}
	\end{equation*}
	is non-empty. Thus,~$\lambda_\star \coloneqq \inf\Lambda$ is well-defined and~$\lambda_\star \in [-\infty,\lambda_1]$.
	
	Arguing as at the end of Step~2 of the proof of Theorem~1 in~\cite{ccg}, by taking
	\begin{equation*}
		\gamma \le \gamma_1 \coloneqq \left(\frac{\mathcal{M}}{2^{1/\past}\mathscr{K}}\right)^{\! p-1}\!,
	\end{equation*}
	we conclude that~$\lambda_\star \geq - \lambda_2$, where
	\begin{equation*}
		\lambda_2 \coloneqq \max \left\{ \left( 2\left(p-1\right) \frac{4^{\past} C_0^{\past}\Haus^{n-1}(\sfera^{n - 1})}{\mathcal{M}^{\past}} \right)^{\!\! \frac{p-1}{n}} \!, 1 \right\}.
	\end{equation*}
	Combining this with~\eqref{eq:MPstarts}, we obtain, in particular, that
	\begin{equation}
		\label{eq:limit-lstar-plap}
		\lambda_\star \in [-\lambda_0, \lambda_0],
	\end{equation}
	with~$\lambda_0 \coloneqq \max \{ \lambda_1, \lambda_2 \}$, depending only on~$n$,~$p$,~$C_0$,~$c_1$,~$\norma{\kappa}_{L^\infty(\R^n)}$, and~$R_0$.
	
	
	\subsection{From integral to pointwise estimates.}
	\label{step:int2point-plap}
	
	Now, we show that, as long as~\eqref{eq:MPstarts} holds true, we can obtain a pointwise upper bound for~$\left(u-u_{\lambda}\right)_+$.
	
	We know that~\eqref{eq:MPstarts} holds for any~$\lambda > \lambda_\star$, by definition of~$\Lambda$. Fatou's lemma ensures that it also holds for~$\lambda = \lambda_\star$, thus
	\begin{equation}
		\label{eq:stimaperlambda}
		\norma*{(u - u_{\lambda})_+}_{L^{\past}\!(\Sigma_{\lambda})} \leq \mathscr{K} \defi(u,\kappa)^{\frac{1}{p-1}} \quad \mbox{for all } \lambda \geq \lambda_\star.
	\end{equation}
	
	Let~$\lambda \geq \lambda_\star$ be fixed. First, we shall prove the pointwise bound near infinity. Indeed, by~\eqref{eq:decadimento}, we have
	\begin{equation}
		\label{eq:bpunt-lont}
		\left(u-u_{\lambda}\right)_+(x) \leq u(x) \leq C_0 \abs{x}^\frac{p-n}{p-1} \leq C_0 \defi(u,\kappa)^{\alpha_\flat}
	\end{equation}
	for some~${\alpha_\flat}>0$ to be determined soon, provided that
	\begin{equation*}
		\abs*{x} \geq \defi(u,\kappa)^{-{\alpha_\flat} \frac{p-1}{n-p}} \eqqcolon R_1.
	\end{equation*}
	
	Secondly, exploiting Theorem~\ref{th:localbound} and Theorem~\ref{th:localbound-bordo}
	on balls with radii~$1 \leq R \leq 6$, we take care of the points where~$\abs*{x} \leq R_1$. More precisely, we choose~$R=6$ in a neighborhood of~$T_\lambda$ and~$R=1$ away from this hyperplane.
	
	Note also that~\eqref{eq:assunzione-pesi}-\eqref{eq:bound-gradiente} hold true. In addition, exploiting~\eqref{eq:bb-grad} and fixing~$r$ in terms of~$p$, we obtain
	\begin{equation*}
		\int_{B_{\! R_1+30}\setminus B_{\! R_0}} \frac{1}{\abs*{\nabla u}^{(p-1)r}} \, dx \leq C_3 \left(R_1+30\right)^{(n-1)r+n},
	\end{equation*}
	for some constant~$C_3\geq 1$, depending on~$n$,~$p$, and~$c_1$. Thus, by Theorem~\ref{th:int-grad-sp} with~$R=R_0$, we may assume
	\begin{equation}
		\label{eq:Cast-bound}
		C^\ast \leq C_3 \, R_1^{(n-1)r+n},
	\end{equation}
	up to possibly enlarging~$C_3 \geq 1$ and restricting ourselves to
	\begin{equation*}
		\gamma \leq \gamma_2 \coloneqq \min\left\{30^\frac{n-p}{{\alpha_\flat}(p-1)}, \left(\frac{C_3}{\mathscr{C}_{\! R_0}}\right)^{\! \frac{n-p}{{\alpha_\flat}(p-1)[(n-1)r+n]}}\right\}.
	\end{equation*}
	This provides an estimate for the constant appearing in~\eqref{eq:condizione-peso}.
	
	Furthermore, to apply Theorem~\ref{th:localbound} and Theorem~\ref{th:localbound-bordo}, we fix any arbitrarily large~$q$, in dependence of~$n$ and~$p$, such that~\eqref{eq:defi-q>} holds and choose any~$2<\mathfrak{q}<2_M$, depending only on~$n$ and~$p$, so that~\eqref{eq:cond-q-0} is satisfied. When~$n=3$, we also take into account the restrictions given in Remark~\ref{rem:cost-wk}. Moreover, we have
	\begin{equation*}
		\begin{split}
			-\Delta_p u + \Delta_p u_\lambda &= \kappa u^{\past-1} - \kappa_\lambda u^{\past-1}_\lambda \\
			&= \left(\kappa-\kappa_0\right) u^{\past-1} + \left(\kappa_0-\kappa_\lambda\right) u^{\past-1}_\lambda + \kappa_0 \left(u^{\past-1}-u^{\past-1}_\lambda\right) \\
			&= \left(\kappa-\kappa_0\right) u^{\past-1} + \left(\kappa_0-\kappa_\lambda\right) u^{\past-1}_\lambda + \kappa_0 \, c_\lambda \left(u-u_\lambda\right),
		\end{split}
	\end{equation*}
	where
	\begin{equation*}
	\label{eq:clambdadef}
		c_{\lambda} (x) \coloneqq
		\begin{dcases}
			\frac{u^{\past-1}(x)-u^{\past-1}(x^{\lambda})}{u(x)-u(x^{\lambda})}	& \quad \mbox{if } u(x) \neq u(x^{\lambda}), \\
			0													& \quad \mbox{if } u(x) = u(x^{\lambda}),
		\end{dcases}
	\end{equation*}
	with~$c_\lambda \in L^\infty(\Sigma_{\lambda})$ and~$0 \leq c_\lambda \leq \left(\past-1\right) C_0^{\past-2}$. Setting
	\begin{equation*}
		g_{\lambda,1} \coloneqq \left(\kappa_0-\kappa\right) u^{\past-1} \quad\mbox{and}\quad g_{\lambda,2} \coloneqq \left(\kappa_0-\kappa_\lambda\right) u^{\past-1}_\lambda,
	\end{equation*}
	this yields
	\begin{equation}
		\label{eq:eq-for-u-ul}
		\begin{split}
			-\Delta_p u - \kappa_0 \, c_\lambda u + g_{\lambda,1} = -\Delta_p u_\lambda - \kappa_0 \, c_\lambda u_\lambda + g_{\lambda,2},
		\end{split}
	\end{equation}
	where~$g_{\lambda,i} \in L^\infty(\Sigma_{\lambda}) \cap L^{(\past)'}\!(\Sigma_{\lambda})$ for~$i=1,2$. Since~$(\past)' < 2 < q <+\infty$, we also have $g_{\lambda,i} \in L^q(\Sigma_{\lambda})$ for~$i=1,2$.
	
	Consequently, considering Remark~\ref{rem:forma-eq}, using the Dirichlet datum on~$T_\lambda$, and choosing~$p_\sharp = \past$\!, we deduce that for every~$x \in \Sigma_\lambda \cap B_{\! R_1}$
	\begin{equation}
		\label{eq:bpunt-1}
		(u-u_\lambda)_+(x) \leq C \, C_\mathcal{S}^{\frac{4}{\past} \frac{\mathfrak{q}}{\mathfrak{q}-2}} \left( \norma{\left(u-u_\lambda\right)_+}_{L^{\past}\!(\Sigma_\lambda)} + \norma*{g_{\lambda,1}-g_{\lambda,2}}_{L^q\left(\Sigma_\lambda\right)} \right),
	\end{equation}
	for some~$C>0$, depending on~$n$,~$p$,~$C_0$, and~$\norma{\kappa}_{L^\infty(\R^n)}$. Recalling definition~\eqref{eq:def_cfm}, by interpolation, we compute
	\begin{equation}
		\label{eq:rhs-est}
		\norma*{g_{\lambda,i}}_{L^q\left(\Sigma_\lambda \right)} \leq \norma*{g_{\lambda,i}}_{L^{(\past)'}\left(\Sigma_\lambda\right)}^{(\past)'/q}\norma*{g_{\lambda,i}}_{L^\infty\left(\Sigma_\lambda\right)}^{1-(\past)'/q} \leq \mathsf{C} \defi(u,\kappa)^{(\past)'/q} \quad \mbox{for } i=1,2
	\end{equation}
	for some~$\mathsf{C}>0$ depending only on~$n$,~$p$, and~$\norma{\kappa}_{L^\infty(\R^n)}$. By Theorem~\ref{th:w-sobolev} and~\eqref{eq:Cast-bound}, it follows that
	\begin{equation*}
		C_\mathcal{S} \leq C R_1^{\frac{p-2}{2(p-1)r}\left[(n-1)r+n\right]}
	\end{equation*}
	for some~$C>0$ depending only on~$n$,~$p$, and~$c_1$. Note that~$C_\mathcal{S}$ refers to the Sobolev constant of the ball on which~\eqref{eq:bpunt-1} is computed, so that the constant~$C_M$ appearing in Theorem~\ref{th:w-sobolev} is also universally bounded. Thus, considering the definition of~$R_1$,~\eqref{eq:stimaperlambda}, and~\eqref{eq:bpunt-1}, we shall fix~$\alpha_\flat>0$ in such a way that
	\begin{gather*}
		h_1(\alpha_\flat) \coloneqq \frac{1}{p-1} - 2{\alpha_\flat} \frac{\mathfrak{q}}{\mathfrak{q}-2} \frac{p-2}{npr}\left[(n-1)r+n\right] > 0, \\
		h_2(\alpha_\flat) \coloneqq \frac{(\past)'}{q} - 2{\alpha_\flat} \frac{\mathfrak{q}}{\mathfrak{q}-2} \frac{p-2}{npr}\left[(n-1)r+n\right] > 0.
	\end{gather*}
	Since~$h_1(0), h_2(0) \in (0,1)$ for~$p>2$, by continuity, there exist some small values~$\alpha_\flat^i \in (0,1)$, depending only on~$n$ and~$p$, such that~$h_i(\alpha_\flat^i) \in (0,1)$ for~$i=1,2$. By~\eqref{eq:bpunt-lont} and~\eqref{eq:bpunt-1}, we conclude that
	\begin{equation}
		\label{eq:bpunt}
		\norma*{(u-u_\lambda)_+}_{L^\infty(\Sigma_\lambda)} \leq C_4 \defi(u,\kappa)^{\alpha_\sharp} \quad \mbox{for every } \lambda \geq \lambda_\star,
	\end{equation}
	where~${\alpha_\sharp} \coloneqq \min\left\{\alpha_\flat^i, h_i(\alpha_\flat^i) \,\lvert\, i=1,2 \right\} \in (0,1)$ depends only on~$n$ and~$p$, and~$C_4 \geq 2\mathsf{C}$ depends only on~$n$,~$p$,~$C_0$,~$c_1$, and~$\norma{\kappa}_{L^\infty(\R^n)}$.
	
	
	\subsection{Finding mass in a small region.}
	
	We shall prove that
	\begin{equation}
		\label{eq:u-ulambdastar-small}
		\norma*{(u-u_{\lambda_\star})_-}_{L^{\past}\!(\Sigma_{\lambda_\star})} \le C \, \abs*{\log\mathrm{def}(u,\kappa)}^{-\alpha},
	\end{equation}
	for some universal constants~$C \ge 1$ and~$\alpha > 0$. We argue by contradiction and assume that
	\begin{equation}
		\label{eq:stima-assurda}
		\norma*{(u-u_{\lambda_\star})_-}_{L^{\past}\!(\Sigma_{\lambda_\star})} > \mathcal{B} \,\abs*{\log\mathrm{def}(u,\kappa)}^{-\alpha}
	\end{equation}
	for some~$\alpha > 0$ and a sufficiently large large~$\mathcal{B} \ge 1$, to be determined shortly in terms of universal constants. As a consequence of~\eqref{eq:stima-assurda}, we will show that the function~$(u - u_{\lambda_\star})_- = (u_{\lambda_\star}-u)_+$ has positive mass in a small cube. In~\ref{step:prop-posi-plap}, we will enlarge the positivity region and, in~\ref{step:contr-plap}, w will derive a contradiction from this.
	
	As in the corresponding step of the proof of Theorem~1.1 in~\cite{ccg}, taking advantage of~\eqref{eq:decadimento}, one can easily show that, by choosing
	\begin{equation}
		\label{eq:R2def}
		\begin{split}
			R_2 &\coloneqq \left( \frac{2(p-1) C_0^{\past} \Haus^{n-1}(\sfera^{n-1})}{n \mathcal{B}^{\past}} \right)^{\!\! \frac{p-1}{n}} \abs*{\log\defi(u,\kappa)}^{\!\frac{p(p-1) \alpha}{n-p}} \\
			&= C_5 \,\mathcal{B}^{-\frac{\past(p-1)}{n}} \,\abs*{\log\defi(u,\kappa)}^{\frac{p(p-1) \alpha}{n-p}},
		\end{split}
	\end{equation}
	we obtain
	\begin{equation}
		\label{eq:massa-bolla}
		\int_{B_{2R_2}(x_\star) \cap \Sigma_{\lambda_\star}} (u_{\lambda_\star}-u)_+^{\past} \, dx \geq \frac{\mathcal{B}^{\past}}{2} \, \abs{\log\defi(u,\kappa)}^{-\past \alpha},
	\end{equation}
	where~$x_\star=\left(0',\lambda_\star\right) \in T_{\lambda_\star}$. Moreover, we have~$R_2 \ge \lambda_0 + 2$ provided that
	\begin{equation*}
		\gamma \leq \gamma_3 \coloneqq \exp\left\{-\mathcal{B}^\frac{1}{\alpha} \left(C_5^{-1}(\lambda_0 + 2)\right)^{\!\frac{n-p}{p(p-1)\alpha}} \right\}.
	\end{equation*}
	
	For~$\delta \in (0, 1)$ to be determined later, consider the slab~$\mathcal{S}_{\lambda_\star,\delta} \coloneqq \Sigma_{\lambda_\star} \setminus \Sigma_{\lambda_\star+3\sqrt{n}\delta}$. Reasoning once more as in~\cite{ccg}, and provided~$\delta$ is sufficiently small, we can exploit~\eqref{eq:C1boundonu} to obtain an estimate analogous to~\eqref{eq:massa-bolla} over~$B_{2R_2}(x_\star) \cap \Sigma_{\lambda_\star+3\sqrt{n}\delta}$. More precisely, we get that
	\begin{equation}
		\label{eq:massa-bolla-1}
		\int_{B_{2R_2}(x_\star) \cap \Sigma_{\lambda_\star+3\sqrt{n}\delta}} (u_{\lambda_\star}-u)_+^{\past} \, dx \geq \frac{\mathcal{B}^{\past}}{4} \,\abs*{\log\defi(u,\kappa)}^{-\past \alpha},
	\end{equation}
	provided
	\begin{equation}
		\label{eq:cond-delta-1}
		\delta \leq \left( \frac{ \mathcal{B}^{\past} \abs*{\log\defi(u,\kappa)}^{-\past \alpha}}{4^{n} \left( 6\sqrt{n} \, C_1 \right)^{\past} \! R_2^{n-1}} \, \right)^{\!\! \frac{1}{\past + 1}} \!.
	\end{equation}
	
	We now enclose~$B_{2R_2}(x_\star) \cap \Sigma_{\lambda_\star + 3\sqrt{n} \delta}$ within the half-cube
	\begin{equation*}
		\big( {-N_0 \delta, N_0 \delta} \big)^{n - 1} \times \big( {\lambda_\star + 3\sqrt{n} \delta, \lambda_\star + (3\sqrt{n} + N_0) \delta} \big),
	\end{equation*}
	where~$N_0 \coloneqq \left\lceil \frac{2 R_2}{\delta} \right\rceil$. Next, we further subdivide this region into the family~$\mathcal{Q}$ of contiguous, essentially disjoint closed cubes of side length~$\delta$. By using that~$\mbox{card}(\mathcal{Q}) = 2^{n - 1} N_0^n$, it follows from~\eqref{eq:R2def} and~\eqref{eq:massa-bolla-1}, that there exists a cube~$Q_\delta \in \mathcal{Q}$ satisfying
	\begin{equation}
		\label{eq:massa-cubo}
		\dashint_{Q_\delta} (u_{\lambda_\star}-u)_+^{\past} \, dx \ge \frac{\mathcal{B}^{p \cdot \past}}{2 \left(8C_5\right)^n} \,\abs{\log\defi(u,\kappa)}^{-p \cdot \past \alpha}.
	\end{equation}
	
	
	\subsection{Propagating positivity to a large region.}
	\label{step:prop-posi-plap}
	
	Now, we take advantage of  the previous step to establish the positivity of~$\left(u_{\lambda_\star}-u\right)_+$ in a universally large region.
	
	Let~$\mathsf{r} \in [\lambda_0 + 2, R_2]$ to be determined later, depending only on universal constants. By~\eqref{eq:limit-lstar-plap}, we have~$B_{\mathsf{r}}(0) \cap \Sigma_{\lambda_\star} \neq \varnothing$. Let~$\delta \in (0, 1)$ be as above and define
	\begin{equation}
	\label{eq:def-Kdelta-plap}
		K_\delta \coloneqq \overline{B_{\mathsf{r}}(0) \cap \Sigma_{\lambda_\star+3\sqrt{n}\delta}}.
	\end{equation}
	Then, the set~$K_\delta$ is a compact subset of~$\Sigma_{\lambda_\star}$, disjoint from the hyperplane~$T_{\lambda_\star}$ where the function~$u-u_{\lambda_\star}$ vanishes. Our goal is to show that~$\left(u_{\lambda_\star}-u\right)_+$ is positive in~$K_\delta$.
	
	Define now~$v \coloneqq u_{\lambda_\star}-u$. Let~$B^{(\delta)}$ be the ball of radius~$\sqrt{n}\delta/2$ such that~$Q_\delta \subseteq B^{(\delta)}$ and let~$B^\star$ a ball of the same radius chosen so that~$\min_{K_\delta} v = \inf_{B^\star} v$ and~$5B^\star \subseteq \Sigma_{\lambda_\star}$ -- here,~$5B^\star$ denotes the ball concentric with~$B^\star$ and with five times the radius. Then, let~$B^{(1)}$ a unit-radius ball whose center is at distance at most~$5$ from that of~$B^{(\delta)}$ and such that~$5B^{(1)} \subseteq \Sigma_{\lambda_\star}$. Given~$N \in \N$, define~$B^{(N)}$ analogously in relation to~$B^\star$. We construct a suitable Harnack chain to connect these balls. This step is technically more involved than the corresponding one in~\cite{ccg}, due to the explicit upper bound for the constant in the weighted Sobolev inequality of Theorem~\ref{th:w-sobolev}, as given in Theorem~\ref{th:wh}.
	
	We begin by noting that~$u$ and~$u_{\lambda_\star}$ satisfy~\eqref{eq:eq-for-u-ul}. Moreover, from~\eqref{eq:bpunt}, we have~$v+C_4 \defi(u,\kappa)^{\alpha_\sharp} \geq 0$ in~$\Sigma_{\lambda_\star}$ and, by~\eqref{eq:rhs-est}, also
	\begin{equation*}
		\norma*{g_{\lambda,1}-g_{\lambda,2}}_{L^q\left(\Sigma_{\lambda_\star}\right)}\leq 2 \mathsf{C} \defi(u,\kappa)^{(\past)'/q},
	\end{equation*}
	where~$q$ and~$\mathfrak{q}$ have been fixed in~\ref{step:int2point-plap}. Finally, since~$\alpha_\sharp \leq (\past)'/q$, we can express~$C_4 \defi(u,\kappa)^{\alpha_\sharp} = 2\mathsf{C} \mathfrak{c} \defi(u,\kappa)^{(\past)'/q}$ for some~$\mathfrak{c} \geq 1$, allowing us to apply Theorem~\ref{th:wh}.
	
	First, we connect~$B^\star$ and~$B^{(1)}$ through a chain constructed as in Lemma~2.2 of~\cite{cicopepo}. By Theorem~\ref{th:int-grad-sp} with~$R=\mathsf{r}+10$, we can control~$C_\mathcal{S}$ from above along this chain in terms of universal quantities. Therefore, by assuming
	\begin{equation}
		\label{eq:cond-m}
		m\left(\frac{\sqrt{n}\delta}{2}\right) \leq 1,
	\end{equation}
	we get
	\begin{equation*}
		\mathcal{M}(1) \,\norma*{v}_{L^s(B^{(N)})} \leq
		\frac{C^{N_1/s}}{m(\sqrt{n}\delta/2)^{N_1+1}} \left( \min_{K_\delta} v + 2 C_4 \defi(u,\kappa)^{\alpha_\sharp} \right)
	\end{equation*}
	for some~$C \geq 1$ depending on universal constants, some small~$s \in (0,1)$, and with
	\begin{equation*}
		N_1 \leq C \log\left(\frac{5}{\delta}\right),
	\end{equation*}
	with~$C \geq 1$ depending on universal constants. As a result, we conclude that
	\begin{equation}
		\label{eq:wh1}
		\norma*{v}_{L^s(B^{(N)})} \leq
		\frac{C_\flat}{\delta^\beta \, m(\sqrt{n}\delta/2)^{\beta \log\left(5/\delta\right)}} \left( \min_{K_\delta} v + 2 C_4 \defi(u,\kappa)^{\alpha_\sharp} \right) \!,
	\end{equation}
	for some universal~$C_\flat \geq 1$ and~$\beta>0$.
	
	Next, we connect~$B^{(1)}$ and~$B^{(N)}$ through a Harnack chain $\{B^{(k)}\}_{k=1}^N$ consisting of balls of unit radius, with centers lying along the segment joining those of~$B^{(1)}$ and~$B^{(N)}$. Additionally, we require that
	\begin{equation*}
		\label{eq:misurainter}
		\frac{\abs{B^{(k)}}}{\abs*{B^{(k-1)}\cap B^{(k)}}} \le 4^n \quad\mbox{for } k=2,\dots,N,
	\end{equation*}
	and observe that we can choose these balls in such a way that~$N \leq 16 \sqrt{n} R_2$. This chain possibly extends beyond the region where~$C_\mathcal{S}$ is universally controlled, therefore we shall exploit Theorem~\ref{th:wh}, together with Remark~\ref{rem:cost-wk}, and consider
	\begin{equation*}
		\mathcal{M}'(R) = R^{-n/s} m'(R) = R^{-n/s} \left(\frac{c_\flat R}{C_\mathcal{S}}\right)^{\!\! C_\natural C_\mathcal{S} \left(C_\mathcal{S}^2 / R \right)^\mu} \quad \mbox{with } \mu = \frac{2}{\mathfrak{q}-2} \geq 1.
	\end{equation*}
	Furthermore, we may assume
	\begin{equation}
		\label{eq:ass-Cs}
		\mathcal{M}'(1) \leq 1.
	\end{equation}
	Thus, a standard chaining argument gives
	\begin{equation}
		\label{eq:wh2}
		\norma*{v}_{L^s(B^{(1)})} \leq \left(\frac{2 \, 4^{n/s} \mathfrak{C} }{m'(1) \min\left\{1,\abs{B_1}\right\}^{1/s}}\right)^{\!\! N} \left( \norma*{v}_{L^s(B^{(N)})} + 2 C_4 \defi(u,\kappa)^{\alpha_\sharp} \right) \!.
	\end{equation}
	Finally, following Lemma~2.2 of~\cite{cicopepo} and assuming
	\begin{equation}
		\label{eq:cond-m'}
		m'\left(\frac{\sqrt{n}\delta}{2}\right) \leq 1,
	\end{equation}
	we obtain
	\begin{equation}
	\label{eq:wh3}
		\mathcal{M}'\left(\frac{\sqrt{n}\delta}{2}\right) \norma*{v}_{L^s(B^{(\delta)})}
		\leq \frac{C_\flat}{\delta^\beta \, m'(\sqrt{n}\delta/2)^{\beta \log\left(5/\delta\right)}} \left(\norma*{v}_{L^s(B^{(1)})} + 2 C_4 \defi(u,\kappa)^{\alpha_\sharp} \right) \!,
	\end{equation}
	up to possibly enlarging~$C_\flat \geq 1$ and~$\beta>0$.
	
	Chaining in this order~\eqref{eq:wh3},~\eqref{eq:wh2}, and~\eqref{eq:wh1} and using that~$N \leq 16\sqrt{n} R_2$, we finally conclude that
	\begin{multline}
		\label{eq:wh-final-plap}
		\mathcal{M}'\left(\frac{\sqrt{n}\delta}{2}\right) \norma*{v}_{L^s(B^{(\delta)})} \\
		\leq \frac{C_\flat^2}{\delta^{2\beta}}  \frac{m'(\sqrt{n}\delta/2)^{-\beta \log\left(5/\delta\right)}}{m(\sqrt{n}\delta/2)^{\beta \log\left(5/\delta\right)}} \left(\frac{C_\star}{m'(1)}\right)^{\!\! 16 \sqrt{n} R_2} \left(\min_{K_\delta} v + 2 C_4 \defi(u,\kappa)^{\alpha_\sharp}\right) \!,
	\end{multline}
	for some~$C_\star \geq 1$ depending on universal constants, up to possibly enlarging~$C_\flat \geq 1$ once more.
	
	Next, we shall take care of the constant~$C_\mathcal{S}$. In particular, we shall estimate this constant on each ball used to obtain~\eqref{eq:wh2} and~\eqref{eq:wh3}. Exploiting~\eqref{eq:limit-lstar-plap}, the definition of~$N_0$, and the lower bound for~$R_2$, it is easy to see that these balls are contained in~$B_{c_n R_2}(0)$ for some dimensional~$c_n \geq 1$. Then, fixing~$r$ as above, we have
	\begin{equation*}
		\int_{B_{c_n R_2+10}\setminus B_{\! R_0}} \frac{1}{\abs*{\nabla u}^{(p-1)r}} \, dx \leq C_3 \left(c_n R_2+10\right)^{(n-1)r+n} \!,
	\end{equation*}
	Thus, by Theorem~\ref{th:int-grad-sp} with~$R=R_0$, we may assume
	\begin{equation}
		\label{eq:Cast-bound-2}
		C^\ast \leq C_3 \, R_2^{(n-1)r+n},
	\end{equation}
	possibly after enlarging~$C_3 \geq 1$, if we restrict ourselves to
	\begin{equation*}
		\gamma \leq \gamma_4 \coloneqq \min\left\{\exp\left\{- \mathcal{B}^\frac{1}{\alpha} \left(\frac{10}{c_n C_5}\right)^{\!\!\frac{n-p}{p(p-1)\alpha}}\right\}, \exp\left\{-C_5^{-\frac{n-p}{p(p-1)\alpha}}\mathcal{B}^\frac{1}{\alpha}\left(\frac{\mathscr{C}_{\! R_0}}{C_3 }\right)^{\!\!\frac{n-p}{p(p-1)\alpha[(n-1)r+n]}} \right\} \right\}.
	\end{equation*}
	By Theorem~\ref{th:w-sobolev} and~\eqref{eq:Cast-bound-2} it follows that
	\begin{equation}
		\label{eq:defomega}
		C_\mathcal{S} \leq C_6 \, R_2^{\varpi} \quad \mbox{with } \varpi = \max\left\{\frac{np}{p-1},\frac{p-2}{2(p-1)r}\left[(n-1)r+n\right]\right\}
	\end{equation}
	for some~$C_6\geq 16\sqrt{n}$ depending only on~$n$,~$p$, and~$c_1$. Note that~$\varpi > 2$, since we have~$np>p^2>2(p-1)$ for every~$2<p<n$, and it depends only on~$n$ and~$p$.
	
	Now, we consider~\eqref{eq:ass-Cs}. Since~$c_\flat \in (0,1)$, this condition is satisfied whenever~$C_\mathcal{S}>1$. In light of~\eqref{eq:defomega}, we obviously have
	\begin{equation}
		\label{eq:dicot-CS}
		\mbox{either} \quad C_\mathcal{S} \leq 1 \quad \mbox{or} \quad  1 < C_\mathcal{S} \leq C_6 \, R_2^{\varpi}.
	\end{equation}
	In particular,~$R_2^\varpi \geq 1$ and the second possibility in~\eqref{eq:dicot-CS} is consistent if we suppose
	\begin{equation*}
		\gamma \leq \gamma_5 \coloneqq \exp\left\{- C_5^{-\frac{n-p}{p(p-1)\alpha\varpi}} \mathcal{B}^\frac{1}{\alpha} \right\}.
	\end{equation*}
	If the first alternative in~\eqref{eq:dicot-CS} holds true,~$C_\mathcal{S}$ is actually universally bounded from above, and~\eqref{eq:wh-final-plap} simplifies to
	\begin{equation*}
		\begin{split}
			\min_{K_\delta} v \geq \left(\frac{2}{\sqrt{n}}\right)^{\!\!\frac{n}{s}}
			&\,\frac{\delta^{2\beta}}{C_\flat^2} \left(\frac{c_\flat \sqrt{n}\delta}{2}\right)^{\!\! C_\natural\left(2/\sqrt{n}\delta\right) ^{\mu}+2\beta C_\natural\left(2/\sqrt{n}\delta\right)^{\mu}\log(5/\delta)}
			\widetilde{C}^{-16\sqrt{n}R_2} \cdot \\ &\cdot\delta^{-n/s}\norma*{v}_{L^s(B^{(\delta)})} - 2 C_4 \defi(u,\kappa)^{\alpha_\sharp}
		\end{split}
	\end{equation*}
	for some universal~$\widetilde{C}\geq 1$. Moreover, we assume that~$\delta \in \left(0,1/3\sqrt{n}\right)$ satisfies
	\begin{equation}
		\label{eq:req-dR2}
		\delta \geq \frac{2C_6}{c_\flat \sqrt{n}} R_2^{-\varpi}.
	\end{equation}
	Since there exists some~$\delta_\flat \in (0,1)$, depending only on~$n$ and~$p$, such that
	\begin{equation*}
		\delta^{\mu} \log(5/\delta) \leq 1,
	\end{equation*}
	for~$\delta \leq \delta_\flat$, by requiring~$\delta \leq \delta_\sharp \coloneqq \min\left\{\delta_\flat, 1/6\sqrt{n}\right\}$, we get
	\begin{equation}
		\label{eq:minv-case1}
		\min_{K_\delta} v \geq  C_7 \,\delta^{2\beta} e^{-C_8 R_2^{2\varpi\mu+1}} \,\delta^{-n/s}\norma*{v}_{L^s(B^{(\delta)})} - 2 C_4 \defi(u,\kappa)^{\alpha_\sharp},
	\end{equation}
	where~$C_7 \coloneqq C_\flat^{-2} \left(2/\sqrt{n}\right)^{n/s}$ and~$C_8 \geq 1$ is a universal constant. We now observe that~$s<\past$, therefore
	\begin{equation*}
		\dashint_{Q_\delta} \left(u_{\lambda_\star}-u\right)_+^{\past} dx \leq C_0^{\past-s} \,\dashint_{Q_\delta} \left(u_{\lambda_\star}-u\right)_+^{s} \,dx.
	\end{equation*}
	From the latter and~\eqref{eq:massa-cubo}, we deduce that
	\begin{equation}
		\label{eq:massa-media}
		\delta^{-n/s}\norma*{v}_{L^s(B^{(\delta)})} \geq \frac{\abs{B_1}^{n/s} \mathcal{B}^{p}}{2^{1/s} \left(8C_5\right)^{n/s} C_0^{\past/s-1}} \,\abs*{\log\mathrm{def}(u,\kappa)}^{-\frac{p\cdot\past\alpha}{s}}.
	\end{equation}
	Finally, using~\eqref{eq:R2def} and~\eqref{eq:massa-media}, we estimate the right-hand side of~\eqref{eq:minv-case1}, obtaining
	\begin{equation*}
		\begin{split}
			\min_{K_\delta} v &\geq \frac{\widehat{C}_7 \,\delta^{2\beta} \mathcal{B}^{p}}{2^{1/s} \left(8C_5\right)^{n/s} C_0^{\past/s-1}} \,\abs{\log\defi(u,\kappa)}^{-\frac{p \cdot \past\alpha}{s}} \cdot \\ &\quad \cdot e^{-C_8 \, C_5^{2\varpi\mu+1} \mathcal{B}^{-\frac{\past(p-1)(2\varpi\mu+1)}{n}} \abs*{\log\defi(u,\kappa)}^{\frac{p(p-1)(2\varpi\mu+1) \alpha}{n-p}}} - 2 C_4 \defi(u,\kappa)^{\alpha_\sharp},
		\end{split}
	\end{equation*}
	where~$\widehat{C}_7 \coloneqq \abs{B_1}^{n/s} C_7$. First, we fix
	\begin{equation*}
		\mathcal{B} \coloneqq \max\left\{ \left(4 C_4 \widehat{C}_7^{-1} \, 2^{1/s} \left(8C_5\right)^{n/s}C_0^{\past/s-1}\right)^{\!\frac{1}{p}}, \left(2 {\alpha_\sharp}^{\!-1} C_8 \, C_5^{2\varpi\mu+1}\right)^{\!\!\frac{n}{\past(2\varpi\mu+1)(p-1)}}\right\} \!,
	\end{equation*}
	therefore
	\begin{equation*}
		\min_{K_\delta} v \geq 4 C_4 \,\delta^{2\beta} \abs*{\log\defi(u,\kappa)}^{-\frac{p\cdot\past\alpha}{s}} \, e^{-\frac{{\alpha_\sharp}}{2} \abs*{\log\defi(u,\kappa)}^{\frac{p(p-1)(2\varpi\mu+1) \alpha}{n-p}}} - 2 C_4 \defi(u,\kappa)^{\alpha_\sharp}.
	\end{equation*}
	Then, we take
	\begin{equation}
		\label{eq:alphaanddeltadef-plap}
		\alpha \coloneqq \frac{n-p}{p(p-1)(2\varpi\mu+1)} \quad \text{and} \quad \delta \geq \left( \defi(u,\kappa)^\frac{{\alpha_\sharp}}{2} \abs*{\log\defi(u,\kappa)}^{\frac{p\cdot\past\alpha}{s}} \right)^{\!\! \frac{1}{2\beta}}\!,
	\end{equation}
	so that
	\begin{equation}
		\label{eq:positività-plap}
		\min_{K_\delta} v \geq 2C_4 \defi(u,\kappa)^{\alpha_\sharp}.
	\end{equation}
	Taking into account~\eqref{eq:req-dR2} and~\eqref{eq:alphaanddeltadef-plap}, we set
	\begin{equation*}
		\delta \coloneqq \max\left\{\left( \defi(u,\kappa)^\frac{{\alpha_\sharp}}{2} \abs*{\log\defi(u,\kappa)}^{\frac{p\cdot\past\alpha}{s}} \right)^{\!\! \frac{1}{2\beta}}, \frac{2C_6 \, C_5^{-\varpi}}{c_\flat \sqrt{n}} \,\mathcal{B}^{\frac{\varpi\past(p-1)}{n}} \abs*{\log\defi(u,\kappa)}^{-\frac{\varpi}{2\varpi\mu+1}} \right\}.
	\end{equation*}
	Note that~$\delta \to 0$ as~$\defi(u,\kappa) \to 0$. Therefore,~$\delta \leq \delta_\sharp$ provided~$\gamma \leq \gamma_6$ for some small universal~$\gamma_6>0$. This condition also implies the validity of~\eqref{eq:cond-m}. Moreover, one can observe that there exists some universal~$\gamma_{7}>0$ such that
	\begin{equation*}
		\defi(\kappa)^\frac{{\alpha_\sharp}}{4\beta} \abs*{\log\defi(u,\kappa)}^{\frac{1}{2\varpi\mu+1} \left(\frac{np}{2\beta s(p-1)}+\varpi\right)} \leq \frac{2C_6}{c_\flat \sqrt{n}} C_5^{-\varpi} \,\mathcal{B}^{\frac{\varpi\past(p-1)}{n}},
	\end{equation*}
	provided that~$\gamma \leq \gamma_{7}$. This, in turn, implies
	\begin{equation*}
		\delta = \frac{2C_6}{c_\flat \sqrt{n}} R_2^{-\varpi} = \frac{2C_6}{c_\flat \sqrt{n}} C_5^{-\varpi} \,\mathcal{B}^{\frac{\varpi\past(p-1)}{n}} \abs*{\log\defi(u,\kappa)}^{-\frac{\varpi}{2\varpi\mu+1}}.
	\end{equation*}
	We shall finally check that this choice of~$\delta$ is consistent with condition~\eqref{eq:cond-delta-1}, at least if~$\gamma$ is small enough. Since~$\past+1>1$, it suffices to verify that
	\begin{equation*}
		\frac{2C_6}{c_\flat \sqrt{n}} R_2^{-\varpi} \leq  \frac{ \mathcal{B}^{\past} \,\abs*{\log\defi(u,\kappa)}^{-\past \alpha}}{4^{n} \left( 6\sqrt{n} \, C_1 \right)^{\past} \! R_2^{n-1}}.
	\end{equation*}
	Recalling the definition of~$R_2$ given in~\eqref{eq:R2def}, the latter holds if
	\begin{equation*}
		C_{9} \leq \abs*{\log\defi(u,\kappa)}^{\frac{\alpha\past}{n} \left[(p-1)(\varpi-n+1)-n\right]}
	\end{equation*}
	for some universal~$C_{9}>0$. Thanks to~\eqref{eq:defomega}, it follows that~$(p-1)(\varpi-n+1)-n>0$, thus condition~\eqref{eq:cond-delta-1} is verified if we restrict ourselves to
	\begin{equation*}
		\gamma \leq \gamma_{8} \coloneqq \exp\left\{-C_{9}^\frac{n}{\alpha\past \left[(p-1)(\varpi-n+1)-n\right]}\right\}.
	\end{equation*}
	
	At this point, we are left with the second alternative in~\eqref{eq:dicot-CS}. We still assume that~\eqref{eq:req-dR2} holds and, from~\eqref{eq:wh-final-plap}, we deduce that
	\begin{equation*}
		\label{eq:posit-tosimplyf}
		\begin{split}
			\min_{K_\delta} v &\geq \left(\frac{2}{\sqrt{n}}\right)^{\!\!\frac{n}{s}}
			\frac{\delta^{2\beta}}{C_\flat^2} R_2^{-2\varpi \left(1+\beta\varpi\right) C_\natural C_6^{\mu+1} R_2^{3\varpi\mu+2}} R_2^{-\beta\varpi^2 C_\natural R_2^{\varpi\mu+1}} \left(\frac{c_\flat R_2^{-\varpi}}{C_6}\right)^{\!\! C_\natural C_6^{2\mu+2} R_2^{(2\mu+1)\varpi+1}} \!\cdot\\
			&\qquad \cdot C_\star^{-16\sqrt{n}R_2} \,\delta^{-n/s}\norma*{v}_{L^s(B^{(\delta)})} - 2 C_4 \defi(u,\kappa)^{\alpha_\sharp}.
		\end{split}
	\end{equation*}
	After performing some computations and possibly enlarging~$C_\star$, we can simplify this expression to
	\begin{equation*}
		\min_{K_\delta} v \geq C_7 \,\delta^{2\beta} e^{-C_{10} R_2^{3\varpi\mu+3}} \delta^{-n/s}\norma*{v}_{L^s(B^{(\delta)})} - 2 C_4 \defi(u,\kappa)^{\alpha_\sharp},
	\end{equation*}
	for some universal constant~$C_{10} \geq 1$. From this,~\eqref{eq:R2def}, and~\eqref{eq:massa-media}, we conclude that
	\begin{align*}
		\min_{K_\delta} v &\geq \frac{\widehat{C}_7 \,\delta^{2\beta} \mathcal{B}^{p}}{2^{1/s} \left(8C_5\right)^{n/s} C_0^{\past/s-1}} \,\abs{\log\defi(u,\kappa)}^{-\frac{p \cdot \past\alpha}{s}} \cdot \\ &\quad \cdot e^{-C_{10} \, C_5^{3\varpi\mu+3} \mathcal{B}^{-\frac{\past(p-1)(3\varpi\mu+3)}{n}} \abs*{\log\defi(u,\kappa)}^{\frac{p(p-1)(3\varpi\mu+3) \alpha}{n-p}}} - 2 C_4 \defi(u,\kappa)^{\alpha_\sharp}.
	\end{align*}
	First, we fix
	\begin{equation*}
		\mathcal{B} \coloneqq \max\left\{ \left(4C_4 \widehat{C}_7^{-1} \,2^{1/s} \left(8C_5\right)^{n/s}C_0^{\past/s-1}\right)^{\!\frac{1}{p}}, \left(2 {\alpha_\sharp}^{\!-1} C_{10} \, C_5^{3\varpi\mu+3}\right)^{\!\!\frac{n}{\past(3\varpi\mu+3)(p-1)}}\right\} \!,
	\end{equation*}
	therefore
	\begin{equation*}
		\min_{K_\delta} v \geq 4 C_4 \,\delta^{2\beta} \abs*{\log\defi(u,\kappa)}^{-\frac{p\cdot\past\alpha}{s}} \, e^{-\frac{{\alpha_\sharp}}{2} \abs*{\log\defi(u,\kappa)}^{\frac{p(p-1)(3\varpi\mu+3) \alpha}{n-p}}} - 2 C_4 \defi(u,\kappa)^{\alpha_\sharp}.
	\end{equation*}
	Then, we take
	\begin{equation}
		\label{eq:alphadef-plap-2}
		\alpha \coloneqq \frac{n-p}{p(p-1)(3\varpi\mu+3)} \quad \text{and} \quad \delta \geq \left( \defi(u,\kappa)^\frac{{\alpha_\sharp}}{2} \abs*{\log\defi(u,\kappa)}^{\frac{p\cdot\past\alpha}{s}} \right)^{\!\! \frac{1}{2\beta}}\!,
	\end{equation}
	so that~\eqref{eq:positività-plap} holds also in this case. Considering~\eqref{eq:req-dR2} and~\eqref{eq:alphadef-plap-2}, we fix
	\begin{equation*}
		\delta \coloneqq \max\left\{\left( \defi(u,\kappa)^\frac{{\alpha_\sharp}}{2} \abs*{\log\defi(u,\kappa)}^{\frac{p\cdot\past\alpha}{s}} \right)^{\!\! \frac{1}{2\beta}}, \frac{2C_6 \, C_5^{-\varpi}}{c_\flat \sqrt{n}} \,\mathcal{B}^{\frac{\varpi\past(p-1)}{n}} \abs*{\log\defi(u,\kappa)}^{-\frac{\varpi}{3\varpi\mu+3}} \right\}.
	\end{equation*}
	Since~$\delta \to 0$ as $\defi(u,\kappa) \to 0$, our assumption~$\delta < 1/3\sqrt{n}$ is verified provided~$\gamma$ is smaller than a universal $\gamma_9>0$. This condition also implies the validity of~\eqref{eq:cond-m} and~\eqref{eq:cond-m'}. As above, we can observe that there exists some universal~$\gamma_{10}>0$ such that
	\begin{equation*}
		\defi(\kappa)^\frac{{\alpha_\sharp}}{4\beta} \abs*{\log\defi(u,\kappa)}^{\frac{1}{3\varpi\mu+3} \left(\frac{np}{2\beta s(p-1)}+\varpi\right)} \leq \frac{2C_6}{c_\flat \sqrt{n}} C_5^{-\varpi} \,\mathcal{B}^{\frac{\varpi\past(p-1)}{n}},
	\end{equation*}
	provided that~$\gamma \leq \gamma_{10}$. This implies
	\begin{equation*}
		\delta = \frac{2C_6}{c_\flat \sqrt{n}} R_2^{-\varpi} = \frac{2C_6}{c_\flat \sqrt{n}} C_5^{-\varpi} \,\mathcal{B}^{\frac{\varpi\past(p-1)}{n}} \abs*{\log\defi(u,\kappa)}^{-\frac{\varpi}{3\varpi\mu+3}}.
	\end{equation*}
	Finally, one can check that condition~\eqref{eq:cond-delta-1} holds provided that~$\gamma \leq \gamma_{11}$, for some universal~$\gamma_{11}>0$.
	
	
	\subsection{Deriving a contradiction.} \label{step:contr-plap}
	
	Here, we aim to prove that
	\begin{equation}
		\label{eq:contrest-plap}
		\norma*{(u - u_{\lambda_\star - \varepsilon})_+}_{L^{\past} \! (\Sigma_{\lambda_\star - \varepsilon})} \le \mathscr{K} \defi(u,\kappa)^{\frac{1}{p-1}},
	\end{equation}
	for every~$\varepsilon > 0$ small enough and provided that~$\gamma$ is sufficiently small. Clearly, this would lead to a contradiction with the definition of~$\lambda_\star$ and thus to the validity of~\eqref{eq:u-ulambdastar-small}.
	
	For~$\epsilon > 0$ to be determined shortly, define
	\begin{equation*}
		E_{\mathsf{r}}^\epsilon \coloneqq \big( {\R^n \setminus B_{\mathsf{r}}(0)} \big) \cap \Sigma_{\lambda_\star-\epsilon} \quad \text{and} \quad S_\delta^\epsilon \coloneqq B_{\mathsf{r}}(0) \cap \left(\Sigma_{\lambda_\star-\epsilon} \setminus \overline{\Sigma_{\lambda_\star+3\sqrt{n}\delta}} \right).
	\end{equation*}
	In this way,~$E_{\mathsf{r}}^\epsilon$,~$K_\delta$, and~$S_\delta^\epsilon$ form a decomposition of~$\Sigma_{\lambda_\star-\epsilon}$ up to negligible sets -- recall that~$K_\delta$ is given in~\eqref{eq:def-Kdelta-plap}. By exploiting the gradient bound~\eqref{eq:C1boundonu} together with~\eqref{eq:positività-plap}, we get
	\begin{equation*}
		u_{\lambda_\star-\epsilon}-u \geq u_{\lambda_\star}-u - 2 C_1 \,\epsilon \geq  14C_4 \defi(u,\kappa)^{\alpha_\sharp} - 2 C_1 \,\epsilon \quad\text{in} \; K_\delta.
	\end{equation*}
	By taking
	\begin{equation*}
		\epsilon \in \left( 0, \frac{7C_4}{2C_1} \defi(u,\kappa)^{\alpha_\sharp} \right] \!,
	\end{equation*}
	we obtain that~$u_{\lambda_\star-\epsilon}-u  \geq 0$ in~$K_\delta$. In addition, note that~$\epsilon \to 0$ as~$\defi(u,\kappa) \to 0$.
	
	As for~\eqref{eq:test-finale}, we have
	\begin{align*}
		\int_{\Sigma_{\lambda_\star-\epsilon}} &\left\langle \abs*{\nabla u}^{p-2} \nabla u - \abs*{\nabla u_{\lambda_\star-\epsilon}}^{p-2} \nabla u_{\lambda_\star-\epsilon}, \nabla (u-u_{\lambda_\star-\epsilon})_+ \right\rangle \, dx \\
		&=\int_{\Sigma_{\lambda_\star-\epsilon}} \left(\kappa-\kappa_0 \right) u^{\past-1} \left(u-u_{\lambda_\star-\epsilon}\right)_+  \, dx \\
		&\quad+ \int_{\Sigma_{\lambda_\star-\epsilon}} \left(\kappa_0 -\kappa_{\lambda_\star-\epsilon}\right) u_{\lambda_\star-\epsilon}^{\past-1} \left(u-u_{\lambda_\star-\epsilon}\right)_+  \, dx \\
		&\quad+ \kappa_0 \int_{\Sigma_{\lambda_\star-\epsilon}} \left(u^{\past-1}-u^{\past-1}_{\lambda_\star-\epsilon}\right) \left(u-u_{\lambda_\star-\epsilon}\right)_+ \, dx.
	\end{align*}
	We estimate the first two terms on the right-hand side as in~\ref{sub:Mpstarts-plap} deducing
	\begin{equation*}
		\begin{split}
			\int_{\Sigma_{\lambda_\star-\epsilon}} &\left(\kappa-\kappa_{\lambda_\star-\epsilon}\right) u^{\past-1} (u-u_{\lambda_\star-\epsilon})_+ \, dx + \int_{\Sigma_{\lambda_\star-\epsilon}} \left(\kappa_0 -\kappa_{\lambda_\star-\epsilon}\right) u_{\lambda_\star-\epsilon}^{\past-1} \left(u-u_{\lambda_\star-\epsilon}\right)_+  \, dx \\
			&\leq 2S^{-1} \defi(u,\kappa) \,\norma*{\nabla (u-u_{\lambda_\star-\epsilon})_+}_{L^{p}(\Sigma_{\lambda_\star-\epsilon})}.
		\end{split}
	\end{equation*}
	To handle the third term, we take advantage of the decomposition of~$\Sigma_{\lambda_\star-\epsilon}$. Using that~$(u - u_{\lambda_\star - \varepsilon})_+ = 0$ in~$K_\delta$, we get
	\begin{equation*}
		\begin{split}
			\kappa_0 \int_{\Sigma_{\lambda_\star-\epsilon}} &\left(u^{\past-1}-u^{\past-1}_{\lambda_\star-\epsilon}\right) (u-u_{\lambda_\star-\epsilon})_+ \, dx \\
			&\leq \left({\past}-1\right) \kappa_0 \int_{\Sigma_{\lambda_\star-\epsilon}} u^{\past-2} (u-u_{\lambda_\star-\epsilon})_+^2 \, dx \\
			&= \left({\past}-1\right) \kappa_0 \left\{ \int_{E^\epsilon_r} u^{\past-2} (u-u_{\lambda_\star-\epsilon})_+^2 \, dx + \int_{S^\epsilon_\delta} u^{\past-2} (u-u_{\lambda_\star-\epsilon})_+^2 \, dx \right\}.
		\end{split}
	\end{equation*}
	We take care of this two terms separately. For the integral over~$E^\epsilon_{\mathsf{r}}$ we follow the argument from~\ref{sub:Mpstarts-plap} to obtain
	\begin{equation*}
		\int_{E^\epsilon_{\mathsf{r}}} u^{\past-2} (u-u_{\lambda_\star-\epsilon})_+^2 \, dx \leq \frac{C_0^{\past-2}}{c_1^{p-2}\mathsf{r}^{\beta^\ast}} \left(\frac{2}{n+s_\ast-2}\right)^{\!\!2} \int_{\Sigma_{\lambda_\star-\epsilon}} \,\abs*{\nabla u}^{p-2} \,\abs*{\nabla (u-u_{\lambda_\star-\epsilon})_+}^2 \, dx.
	\end{equation*}
	Since~$(u-u_{\lambda_\star-\epsilon})_+$ vanishes on~$K_\delta$, the weighted Poincaré inequality of Corollary~\ref{cor:w-poinc} yields
	\begin{equation*}
		\begin{split}
			\int_{S^\epsilon_\delta} u^{\past-2} (u-u_{\lambda_\star-\epsilon})_+^2 \, dx &= \int_{S^\epsilon_\delta \cup K_\delta} u^{\past-2} (u-u_{\lambda_\star-\epsilon})_+^2 \, dx \\
			&\leq \mathcal{C}_P(S^\epsilon_\delta)\, C_\mathcal{S}^2(B_{\mathsf{r}}(0) \setminus K_\delta) \int_{\Sigma_{\lambda_\star-\epsilon}} \,\abs*{\nabla u}^{p-2} \,\abs*{\nabla (u-u_{\lambda_\star-\epsilon})_+}^2 \, dx.
		\end{split}
	\end{equation*}
	
	By combining the above estimates and arguing as in~\ref{sub:Mpstarts-plap}, we conclude that
	\begin{equation*}
		\begin{split}
			\mathcal{A}(\mathsf{r},\epsilon) \int_{\Sigma_{\lambda_\star-\epsilon}} &\left(\abs*{\nabla u}+\abs*{\nabla u_{\lambda_\star-\epsilon}}\right)^{p-2} \abs*{\nabla (u-u_{\lambda_\star-\epsilon})_+}^2 \, dx \\
			&\leq 2S^{-1} c(p)^{-1} \defi(u,\kappa) \,\norma*{\nabla (u-u_{\lambda_\star-\epsilon})_+}_{L^{p}(\Sigma_{\lambda_\star-\epsilon})},
		\end{split}
	\end{equation*}
	where
	\begin{equation*}
		\mathcal{A}(\mathsf{r},\epsilon) \coloneqq 1 - \frac{\left({\past}-1\right) \kappa_0 \, C_0^{\past-2}}{c(p) c_1^{p-2}\mathsf{r}^{\beta^\ast}} \left(\frac{2}{n+s_\ast-2}\right)^{\!\!2} - \left({\past}-1\right) \kappa_0 \,c(p)^{-1} \mathcal{C}_P(S^\epsilon_\delta)\, C_\mathcal{S}^2(B_{\mathsf{r}}(0) \setminus K_\delta).
	\end{equation*}
	To deduce~\eqref{eq:contrest-plap} from this, it suffices to show that the parameters~$\mathsf{r}$ and~$\epsilon$ can be selected in such a way that
	\begin{equation}
		\label{eq:contrad-A-plap}
		\mathcal{A}(\mathsf{r},\epsilon) \geq \frac{1}{2},
	\end{equation}
	that is
	\begin{equation*}
		\frac{\left({\past}-1\right) \kappa_0 \, C_0^{\past-2}}{c(p) c_1^{p-2}\mathsf{r}^{\beta^\ast}} \left(\frac{2}{n+s_\ast-2}\right)^{\!\!2} + \left({\past}-1\right) \kappa_0 \,c(p)^{-1} \mathcal{C}_P(S^\epsilon_\delta)\,  C_\mathcal{S}^2(B_{\mathsf{r}}(0) \setminus K_\delta) \leq \frac{1}{2}.
	\end{equation*}
	To proceed, we first fix 
	\begin{equation*}
		\mathsf{r} \coloneqq \max\left\{R_0, \lambda_0+2, \left(\frac{4 \left({\past}-1\right) \norma{\kappa}_{L^\infty(\R^n)} \, C_0^{\past-2}}{c(p) c_1^{p-2}}\right)^{\!\!\frac{1}{\beta^\ast}} \left(\frac{2}{n+s_\ast-2}\right)^{\!\!\frac{2}{\beta^\ast}} \right\}
	\end{equation*}
	and observe that we have~$\mathsf{r}\leq R_2$ by requiring
	\begin{equation*}
		\gamma \leq \gamma_{12} \coloneqq \exp\left\{-\mathcal{B}^\frac{1}{\alpha}\left(\mathsf{r} C_5^{-1}\right)^{\!\frac{n-p}{p(p-1)\alpha}}\right\}.
	\end{equation*}
	
	By Theorem~\ref{th:int-grad-sp}, we can assume that~$C^\ast \leq \mathscr{C}_\mathsf{r}$, for some~$\mathscr{C}_\mathsf{r}>0$ depending on~$n$,~$p$, $C_0$,~$c_1$,~$\norma{\kappa}_{L^\infty(\R^n)}$,~$\underline{\kappa}$, and~$R_0$. Moreover, by the definition of~$\mathsf{r}$ and~\eqref{eq:limit-lstar-plap}, it follows that~$B_{1/2}(x) \subseteq K_\delta$ for some~$x \in K_\delta$. As a consequence, we can upper bound the constant~$\widehat{C}$ in Theorem~\ref{th:w-sobolev} and deduce that
	\begin{equation*}
		C_\mathcal{S}(B_{\mathsf{r}}(0) \setminus K_\delta) \leq C_{11},
	\end{equation*}
	for some universal~$C_{11} \geq 1$. In light of this, of Corollary~\ref{cor:w-poinc} with~$\theta=1/2$, and the fact that there exists a small universal~$\gamma_\flat>0$ such that
	\begin{equation*}
		\gamma^{\alpha_\sharp} \,\abs*{\log\gamma}^{\frac{\varpi}{2\varpi\mu+1}} \leq 1,
	\end{equation*}
	for~$\gamma \leq \gamma_\flat$, we choose~$\epsilon$ small such that
	\begin{align*}
		\left({\past}-1\right)& \,\kappa_0 \, c(p)^{-1} \mathcal{C}_P(S^\epsilon_\delta)\, C_\mathcal{S}^2(B_{\mathsf{r}}(0) \setminus K_\delta) \\
		&\leq \left({\past}-1\right) \norma{\kappa}_{L^\infty(\R^n)} \, c(p)^{-1} C_{11}^2 \abs*{S^\epsilon_\delta}^\frac{1}{(p-1)n} \leq C_{11}^2 \mathsf{r}^\frac{n-1}{(p-1)n} \left(3\delta+\epsilon\right)^\frac{1}{(p-1)n} \\
		&\leq C_{12} \,\abs*{\log\gamma}^{-\frac{\varpi}{\left(2\varpi\mu+1\right)\left(p-1\right)n}} \leq \frac{1}{4}.
	\end{align*}
	The last inequality holds provided
	\begin{equation*}
		\gamma \leq \gamma_{13} \coloneqq \min\left\{\gamma_\flat, \exp\left\{-\left(4C_{12}\right)^{\!\frac{(2\varpi\mu+1)(p-1)n}{\varpi}}\right\} \right\}.
	\end{equation*}
	Since~\eqref{eq:contrad-A-plap} holds and~$p>2$, we obtain
	\begin{equation*}
		\norma*{\nabla (u-u_{\lambda_\star-\epsilon})_+}_{L^{p}(\Sigma_{\lambda_\star-\epsilon})}^p \leq 4S^{-1} c(p)^{-1} \defi(u,\kappa) \,\norma*{\nabla (u-u_{\lambda_\star-\epsilon})_+}_{L^{p}(\Sigma_{\lambda_\star-\epsilon})}.
	\end{equation*}
	As in~\ref{sub:Mpstarts-plap}, this leads to the validity of~\eqref{eq:contrest-plap} and, ultimately, to a contradiction.
	
	So far, we have established~\eqref{eq:u-ulambdastar-small}, with some~$\alpha$ depending only on~$n$ and~$p$. This, together with~\eqref{eq:stimaperlambda} for~$\lambda=\lambda_\star$, gives
	\begin{equation*}
		\label{eq:p*-Rn}
		\norma*{u-u_{\lambda_\star}}_{L^{\past}\!(\R^n)} = 2 \norma*{u-u_{\lambda_\star}}_{L^{\past}\!(\Sigma_{\lambda_\star})} \leq C \, \abs*{\log\mathrm{def}(u,\kappa)}^{-\alpha},
	\end{equation*}
	for some universal constant~$C\geq 1$ and some~$\alpha>0$ depending only on~$n$ and~$p$.
	
	We conclude this step by proving the validity of the claims made in~\eqref{eq:mainMPclaim-plap}. In~\eqref{eq:bpunt} we have already obtained the pointwise estimate appearing on the first line.
	The~$L^\infty$-bound on the second line can be easily established by applying the argument of~\ref{step:int2point-plap}, taking into account~\eqref{eq:limit-lstar-plap}, to the functions~$u-u_{\lambda_\star}$ and $u_{\lambda_\star}-u$ and provided $\gamma$ is smaller than a universal $\gamma_{14}>0$. Finally, for the $L^p$-estimate for the gradient we test the equation for~$u$ and~$u_{\lambda_\star}$ with~$u-u_{\lambda_\star}$, and exploit~\eqref{eq:disp-basso-2}, hypothesis~\eqref{eq:decadimento}, and H\"older inequality to conclude
	\begin{equation*}
		\begin{split}
			\hat{c}(p) &\int_{\R^n} \,\abs*{\nabla (u-u_{\lambda_\star})}^p \, dx \leq \int_{\R^n} \left\langle \abs*{\nabla u}^{p-2} \,\nabla u - \abs*{\nabla u_{\lambda_\star}}^{p-2} \,\nabla u_{\lambda_\star}, \nabla (u-u_{\lambda_\star}) \right\rangle \, dx \\
			&= \int_{\R^n} \left(\kappa u^{\past-1} - \kappa_{\lambda_\star} u^{\past-1}_{\lambda_\star}\right) (u-u_{\lambda_\star}) \, dx \\
			&= \int_{\R^n} \left(\kappa-\kappa_0\right) u^{\past-1} \left(u-u_{\lambda_\star}\right) dx + \int_{\R^n} \left(\kappa_{\lambda_\star}-\kappa_0\right) u^{\past-1}_{\lambda_\star} \left(u-u_{\lambda_\star}\right) dx \\
			&\quad+ \kappa_0 \int_{\R^n} \left(u^{\past-1}-u^{\past-1}_{\lambda_\star}\right) \left(u-u_{\lambda_\star}\right) dx \\
			&\leq 2\defi(u,\kappa) \,\norma*{u-u_{\lambda_\star}}_{L^{\past}\!(\R^n)} + \left(\past-1\right) \norma{\kappa}_{L^\infty(\R^n)} \,\norma*{u}_{L^{\past}\!(\R^n)}^{\past-2} \norma*{u-u_{\lambda_\star}}_{L^{\past}\!(\R^n)}^2 \\
			&\leq C \,\abs*{\log\mathrm{def}(u,\kappa)}^{-2\alpha},
		\end{split}
	\end{equation*}
	for some universal constant~$C\geq 1$.
	
	\subsection{Almost radial symmetry with respect to an approximate center.} \label{step:almostradsym-plap}
	
	By repeating the procedure along~$n$ linearly independent directions -- say,~$e_1,\ldots,e_n$ -- we identify~$n$ values~$\lambda_k = \lambda_\star(e_k) \in \R$,~$k = 1, \dots, n$, for which~\eqref{eq:mainMPclaim-plap} holds true. Considering the corresponding maximal hyperplanes~$T_{e_k, \lambda_k}$,~$k=1,\dots,n$, we define the approximate center of symmetry~$\mathcal{O}$ as their intersection. We then define the reflection with respect to~$\mathcal{O}$ by~$u_{\mathcal{O}}(x) \coloneqq u\left(2\lambda_1-x_1,\dots,2\lambda_n-x_n\right)$. As a consequence of this definition and~\eqref{eq:mainMPclaim-plap}, applying the triangle inequality yields
	\begin{equation}
		\label{eq:u-uOests-plap}
		\norma*{u - u_{\mathcal{O}}}_{L^\infty(\R^n)} + \norma*{\nabla (u - u_{\mathcal{O}})}_{L^p(\R^n)} \le n C_\sharp \,\abs{\log\defi(u,\kappa)}^{-\alpha_\star}.
	\end{equation}
	Moreover, in light of~\eqref{eq:limit-lstar-plap}, we know that
	\begin{equation}
		\label{eq:Oinlambdacube-plap}
		\mathcal{O} \in {[-\lambda_0, \lambda_0]}^n.
	\end{equation}
	Let~$\Sigma_{\omega, \lambda_\star}$ and~$T_{\omega, \lambda_\star}$, with~$\lambda_\star = \lambda_\star(\omega)$, denote the maximal half-space and maximal hyperplane corresponding to a direction~$\omega \in \sfera^{n - 1}$. We will prove that
	\begin{equation} \label{eq:claimhyperclose}
		\mbox{either} \quad \mathcal{O} \in \overline{\Sigma_{\omega, \lambda_\star}} \quad \mbox{or} \quad \mbox{dist}(\mathcal{O}, T_{\omega, \lambda_\star}) \le C_{13} \, \abs*{\log \defi(u,\kappa)}^{-\frac{\alpha_\star}{p+1}},
	\end{equation}
	for every~$\omega \in \sfera^{n- 1}$ and for some universal constant~$C_{13} > 0$.
	
	To establish~\eqref{eq:claimhyperclose}, let~$\omega \in \sfera^{n-1}$ be fixed and assume that
	\begin{equation} \label{eq:OnotinSigma}
		\mathcal{O} \notin \overline{\Sigma_{\lambda_\star}},
	\end{equation}
	where, for simplicity, we omit the reference to~$\omega$ in the notation.
	
	Using a change of variables and the~$\mathcal{D}^{1,p}$-estimate in~\eqref{eq:mainMPclaim-plap}, we see that
	\begin{align*}
		\abs*{\norma*{\nabla u}_{L^p(\Sigma_{\lambda_\star})} - \norma*{\nabla u}_{L^p(\R^n \setminus \Sigma_{\lambda_\star})}} \leq C_\sharp \abs*{\log \defi(u,\kappa)}^{-\alpha_\star}.
	\end{align*}
	From this and~\eqref{eq:bound-gradu-plap}, we compute
	\begin{align}
		\notag
		& \abs*{\int_{\Sigma_{\lambda_\star}} \,\abs*{\nabla u}^p \, dx - \frac{1}{2} \int_{\R^n} \,\abs*{\nabla u}^p \, dx} = \frac{1}{2} 
		\,\abs*{\norma*{\nabla u}_{L^p(\Sigma_{\lambda_\star})}^p - \norma*{\nabla u}_{L^p(\R^n \setminus \Sigma_{\lambda_\star})}^p} \\
		\label{eq:energ-semispazio}
		& \hspace{45pt} \leq C
		\,\abs*{\norma*{\nabla u}_{L^p(\Sigma_{\lambda_\star})} - \norma*{\nabla u}_{L^p(\R^n \setminus \Sigma_{\lambda_\star})} } \, \norma*{\nabla u}_{L^p(\R^n)}^{p-1} \\
		\notag
		& \hspace{45pt} \le C \,\abs*{\log \defi(u,\kappa)}^{-\alpha_\star} \left( \int_{\R^n} \left(1 + \abs*{x}^{\frac{n-1}{p-1}}\right)^{\!-p} dx \right)^{\!\! \frac{p-1}{p}} \le C \, \abs*{\log \defi(u,\kappa)}^{- \alpha_\star},
	\end{align}
	for some universal constant~$C \ge 1$.
	
	Let~$\Sigma_{\lambda_\star \mathcal{O}}$ and~$u_{\lambda_\star \mathcal{O}}$ denote the reflections of~$\Sigma_{\lambda_\star}$ and~$u_{\lambda_\star}$ with respect to~$\mathcal{O}$. Consider the slab~$\mathcal{S}_1 \coloneqq \R^n \setminus \left( \Sigma_{\lambda_\star} \cup \Sigma_{\lambda_\star\mathcal{O}} \right)$ and observe that~$\mathcal{S}_1 \neq \varnothing$, thanks to~\eqref{eq:OnotinSigma}. The~$\mathcal{D}^{1,p}$-bounds of~\eqref{eq:mainMPclaim-plap} and~\eqref{eq:u-uOests-plap} give that
	\begin{equation*}
		\norma*{ \nabla (u-u_{\lambda_{\star\mathcal{O}}}) }_{L^p(\R^n)} \le C \, \abs{\log\defi(u,\kappa)}^{- \alpha_\star}.
	\end{equation*}
	Through this, we easily obtain estimate~\eqref{eq:energ-semispazio} with~$\Sigma_{\lambda_\star\mathcal{O}}$ in place of~$\Sigma_{\lambda_\star}$. This, together with~\eqref{eq:energ-semispazio}, implies that
	\begin{equation}
		\label{eq:step1-plap}
		\sigma_1 \coloneqq \norma*{\nabla u}_{L^p(\mathcal{S}_1)} \leq C_{14} \,\abs{\log\defi(u,\kappa)}^{-\frac{\alpha_\star}{p}},
	\end{equation}
	for some universal constant~$C_{14}>0$.
	
	Proceeding inductively, we define~$\mathcal{S}_{2k}$ as the reflection of~$\mathcal{S}_{2k-1}$ with respect to the hyperplane~$T_{\lambda_\star}$ while~$\mathcal{S}_{2k+1}$ is defined as the reflection of~$\mathcal{S}_{2k}$ with respect to~$\mathcal{O}$. We refer to Figure~3 in~\cite{ccg} for a pictorial representation of this construction. It can be shown by induction that
	\begin{equation} 
		\label{eq:sigmakest-plap}
		\sigma_k \coloneqq \norma*{\nabla u}_{L^p(\mathcal{S}_k)} \le k (n+1) \left(C_\sharp + C_{14}\right) \abs{\log\defi(u)}^{-\frac{\alpha_\star}{p}} \quad \mbox{for every } k \in \N.
	\end{equation}
	
	We now show that~$u$ has universally positive energy inside a sufficiently large ball centered at~$\mathcal{O}$. To this end, using~\eqref{eq:Oinlambdacube-plap}, we obtain, for~$R_3 \ge \sqrt{n} \, \lambda_0$,
	\begin{align*}
		\int_{\R^n \setminus B_{2R_3}(\mathcal{O})} \,\abs*{\nabla u}^p \, dx \leq \frac{\left(p-1\right)C_1^p \, \Haus^{n - 1}(\sfera^{n - 1})}{\left(n-p\right)} R_3^{\frac{p-n}{p-1}} \le \frac{1}{2} S^p \mathcal{M}^{p},
	\end{align*}
	provided that
	\begin{equation*}
		R_3 \ge \mathcal{R} \coloneqq \left( \frac{2 \left(p-1\right) C_1^p \, \Haus^{n - 1}(\sfera^{n - 1})}{(n-p) S^p \mathcal{M}^{p}} \right)^{\!\! \frac{p-1}{n - p}} \!.
	\end{equation*}
	Thus, setting~$R_3 \coloneqq \max \left\{\sqrt{n}\,\lambda_0, \mathcal{R} \right\}$ and exploiting Sobolev inequality as well as~\eqref{eq:boundmassa}, we deduce
	\begin{equation*}
		S^p \mathcal{M}^{p} \leq S^p \left(\int_{\R^n} u^{\past} dx\right)^{\!\!\frac{p}{\past}} \leq \int_{\R^n} \,\abs*{\nabla u}^p \, dx.
	\end{equation*}
	As a consequence,
	\begin{equation}
		\label{eq:ener-bolla}
		\int_{B_{2R_3}\left(\mathcal{O}\right)} \,\abs*{\nabla u}^p \, dx \geq \frac{1}{2} S^p\mathcal{M}^{p}.
	\end{equation}
	
	We are finally in position to prove that~$\mathcal{O}$ is close to~$T_{\lambda_\star}$. Let~$\mathfrak{s}>0$ be the width of each slab~$\mathcal{S}_k$ constructed above -- i.e.,~$\mathfrak{s} = 2 \, \mbox{dist} \left( \mathcal{O}, T_{\lambda_\star} \right)$. First, observe that~$\mathfrak{s} < 4R_3$. Indeed, if this were not the case, we would have~$B_{2R_3}(\mathcal{O}) \subseteq \mathcal{S}_1$, which, in view of~\eqref{eq:step1-plap} and~\eqref{eq:ener-bolla}, is impossible provided that
	\begin{equation*}
		\gamma \le \gamma_{15} \coloneqq \exp\left\{-\left( \frac{4 C_{14}^p}{S^p \mathcal{M}^{p}} \right)^{\!\! \frac{1}{\alpha_\star}}\right\}.
	\end{equation*}
	It is easy to check that~$B_{2R_3} (\mathcal{O})$ is covered by the first~$k_0 \coloneqq \left\lceil \frac{4 R_3}{\mathfrak{s}} \right\rceil + 1$ slabs~$\mathcal{S}_k$. By using that~$k_0 \le \frac{12 R_3}{\mathfrak{s}}$ and~\eqref{eq:sigmakest-plap}, we deduce
	\begin{equation*}
		\begin{split}
			\int_{B_{2R_3}\left(\mathcal{O}\right)} \,\abs*{\nabla u}^p \, dx &\leq \sum_{k=1}^{k_0} \sigma_k^p \leq (n+1)^p (C_\sharp + C_{14})^p \, \abs{\log\defi(u,\kappa)}^{-\alpha_\star} \sum_{k = 1}^{k_0} k^p \\
			&\leq \frac{12^{p+1} R_3^{p+1}}{\mathfrak{s}^{p+1}} \, (n+1)^p (C_\sharp + C_{14})^p \, \abs{\log\defi(u,\kappa)}^{-\alpha_\star}.
		\end{split}
	\end{equation*}
	Combining this with~\eqref{eq:ener-bolla}, we conclude that
	\begin{equation*}
		\mathrm{dist}\left(\mathcal{O},T_{\lambda_\star}\right)^{p+1} = \left(\frac{\mathfrak{s}}{2} \right)^{\! p+1} \le \frac{2 \, 6^{p+1} R_3^{p+1}}{S^p \mathcal{M}^{p}} \, (n+1)^p (C_\sharp + C_{14})^p \, \abs{\log\defi(u,\kappa)}^{-\alpha_\star}.
	\end{equation*}
	Thus, claim~\eqref{eq:claimhyperclose} is established.
	
	Given a direction~$\omega \in \sfera^{n - 1}$, let~$\mu_\star = \mu_\star(\omega) \in \R$ be the unique real number for which~$\mathcal{O} \in T_{\omega, \mu_\star}$, i.e.,~$\mu_\star = \left\langle \omega, \mathcal{O} \right\rangle$. We now claim that there exists a universal constant~$C_{15} > 0$ such that, for every~$\omega \in \sfera^{n - 1}$,
	\begin{equation} \label{eq:2ndmainMPclaim-plap}
		\begin{gathered}
			u(x) - u_{\omega, \lambda}(x) \le C_{15} \,\abs*{\log \defi(u,\kappa)}^{-\frac{\alpha_\star}{p+1}} \quad \mbox{for every } x \in \Sigma_{\omega, \lambda} \mbox{ and } \lambda \ge \mu_\star, \\
			\norma*{ u - u_{\omega, \mu_\star} }_{L^\infty(\R^n)} + \norma*{ \nabla (u - u_{\omega, \mu_\star}) }_{L^p(\R^n)}^p \le C_{15} \,\abs*{\log \defi(u,\kappa)}^{-\frac{\alpha_\star}{p+1}}.
		\end{gathered}
	\end{equation}
	Note that once we have established~\eqref{eq:2ndmainMPclaim-plap}, the conclusion follows as in~\cite{ccg}.
	
	To this end, by taking advantage of~\eqref{eq:claimhyperclose}, we see that the pointwise inequality on the first line of~\eqref{eq:2ndmainMPclaim-plap} follows directly from that of~\eqref{eq:mainMPclaim-plap} -- trivially in case the first alternative in~\eqref{eq:claimhyperclose} holds, using~\eqref{eq:C1boundonu} if the second one is true. Applying this inequality for both directions~$\omega$ and~$-\omega$, one deduces the validity of the~$L^\infty$-estimate stated on the second line of~\eqref{eq:2ndmainMPclaim-plap} -- here we are taking advantage of the fact that~$\mu_\star(-\omega) = - \mu_\star(\omega)$ and therefore~$u_{-\omega, \mu_\star(-\omega)} = u_{\omega, \mu_\star(\omega)}$ in~$\R^n$. Finally, by testing the equations for~$u$ and~$u_{\omega, \mu_\star}$ against~$u - u_{\omega, \mu_\star}$, performing a change of changing variables, and using hypotheses~\eqref{eq:decadimento}, we obtain
	\begin{equation*}
		\begin{split}
			\hat{c}(p) \int_{\R^n} \abs*{\nabla (u-u_{\omega, \mu_\star})}^p \, dx
			& \leq 2 \int_{\R^n} \kappa u^{\past} (u-u_{\omega, \mu_\star}) \, dx \\
			& \le 2 \, C_0^{\past} \, \norma*{ \kappa }_{L^\infty(\R^n)} \, \norma*{ u - u_{\omega, \mu_\star} }_{L^\infty(\R^n)} \int_{\R^n} \left(1 + \abs*{x}^{\frac{n-p}{p-1}}\right)^{\!\!-\past} \! dx \\
			& \le C \,\abs*{\log \defi(u,\kappa)}^{-\frac{\alpha_\star}{p+1}},	\end{split}
	\end{equation*}
	for some universal constant~$C > 0$. This establishes the~$\mathcal{D}^{1,p}$-bound, thereby concluding the proof.
	
	
	\appendix
	\section{Proof of Theorem~\ref{th:int-grad} and Theorem~\ref{th:int-grad-sp}}
	\label{ap:proofsum}
	
	We provide here some details of the proofs of Theorem~\ref{th:int-grad} and Theorem~\ref{th:int-grad-sp}.
	
	\begin{proof}[Proof of Theorem~\ref{th:int-grad}]
		Recalling the notation introduced in Section~\ref{sec:prelim-bolla-plap}, we define~$E=B_{1-\frac{1}{2}\delta_0}$ and~$E_{\delta_0}=B_{1-\frac{1}{4}\delta_0}$. In this way,
		\begin{equation*}
			\mathcal{Z}_u \Subset E \Subset E_{\delta_0} \Subset B_1,
		\end{equation*}
		moreover, by~\eqref{eq:grad>0}, we have that~$\abs*{\nabla u} \geq \left(1-\frac{1}{2^{\alpha_\star}}\right)\beta B$ in~$B_1 \setminus E$. As a consequence,
		\begin{equation*}
			\int_{B_1\setminus E} \frac{1}{\abs*{\nabla u}^{(p-1)r}} \, dx \leq C,
		\end{equation*}
		for some~$C>0$ depending only on~$n$,~$p$,~$r$,~$\beta$, and~$B$. Therefore, we only need to estimate the integral over~$E$. To this end, observe that by~\eqref{eq:ulinearbound}, we have
		\begin{equation}
			\label{eq:def-u0}
			u(x) \geq \frac{1}{2} \frac{\delta_0}{C_2} \eqqcolon u_0 \quad \mbox{for all } x \in E
		\end{equation}
		and, by~\eqref{eq:mainhyp-f} and~\eqref{eq:k-below}, we also have
		\begin{equation*}
			\kappa f(u) \geq \underline{\kappa} \min_{\left[u_0,C_0\right]} f \eqqcolon f_0>0 \quad \mbox{a.e.\ in } E.
		\end{equation*}
		
		From now on, the proof is an adaptation of those of Theorem~2.2 and Theorem~2.3 in~\cite{ds-poinc}.
	\end{proof}
	
	\begin{remark}
		\label{rem:dep-f}
		As one can already observe in the proof of Theorem~2.3 in~\cite{ds-poinc}, the dependence on~$f$ of the constant~$\mathscr{C}$ in Theorem~\ref{th:int-grad} is clear and through~$\min_{\left[u_0,C_0\right]} f$, where~$u_0$ is given in~\eqref{eq:def-u0} and depends on universal constants, as well. In particular, if also~$f(0)>0$, one can assume that~$\mathscr{C}$ depends on~$\min_{\left[0,C_0\right]} f$.
	\end{remark}
	
	We now address the case of the critical equation.
	
	\begin{proof}[Proof of Theorem~\ref{th:int-grad-sp}]
		Note that by~\eqref{eq:bb-grad}, we have
		\begin{equation*}
			\mathcal{Z}_u \Subset B_{\! R_0} \Subset B_{\! R+1},
		\end{equation*}
		and~$\abs*{\nabla u} \geq c_1 (R+1)^\frac{1-n}{p-1} $ in~$B_{\! R+1} \setminus B_{\! R_0}$. As a result,
		\begin{equation*}
			\int_{B_{\! R+1} \setminus B_{\! R_0}} \frac{1}{\abs*{\nabla u}^{(p-1)r}} \, dx \leq C,
		\end{equation*}
		for some~$C>0$ depending only on~$n$,~$p$,~$r$,~$c_1$,~$R_0$, and~$R$. Therefore, we only need to estimate the integral over~$B_{\! R_0}$. By observing that, thanks to~\eqref{eq:bb-u}, we have
		\begin{equation*}
			u(x) \geq \frac{c_0}{1+R_0^\frac{n-p}{p-1}} \eqqcolon u_0 \quad \mbox{for every } x \in B_{\! R_0},
		\end{equation*}
		it follows that
		\begin{equation*}
			\kappa u^{\past-1} \geq \underline{\kappa} u_0^{\past-1} \quad \mbox{a.e.\ in } B_{\! R_0}.
		\end{equation*}
		Therefore, we can argue as in the proof of Theorem~\ref{th:int-grad}.
	\end{proof}
	
	
	\section*{Acknowledgments} 
	\noindent The author has been partially supported by the “Gruppo Nazionale per l'Analisi Matematica, la Probabilità e le loro Applicazioni” (GNAMPA) of the “Istituto Nazionale di Alta Matematica” (INdAM, Italy).
	
	He would also like to express his sincere gratitude to Prof.\ Giulio Ciraolo and Prof.\ Matteo Cozzi for their advice during the preparation of this work.
	


\bigskip
	
	\begin{thebibliography}{$99\,$}
		
		\bibitem{berg}
		M. S. Berger,
		\emph{On the existence and structure of stationary states for a nonlinear Klein-Gordan equation},
		J. Funct. Anal. \textbf{9} (1972), no. 3, 249--261.
		
		\bibitem{cgs}
		L. A. Caffarelli, B. Gidas, J. Spruck,
		\emph{Asymptotic symmetry and local behavior of semilinear elliptic equations with critical Sobolev growth},
		Comm. Pure Appl. Math. \textbf{42} (1989), no. 3, 271--297.
		
		\bibitem{tc-schrod}
		T. Cazenave,
		\emph{Semilinear Schr\"odinger equations},
		Courant Lecture Notes, 10, American Mathematical Society, Providence, RI, Courant Institute of Mathematical Sciences, New York, NY, 2003.
		
		\bibitem{sc-stellar}
		S. Chandrasekhar,
		\emph{An introduction to the study of stellar structure},
		Dover, New York, NY, 1957.
		
		\bibitem{cl}
		W. Chen, C. Li,
		\emph{Classification of solutions of some nonlinear elliptic equations},
		Duke Math. J. \textbf{63} (1991), no. 3, 615--622.
		
		\bibitem{ccg}
		G. Ciraolo, M. Cozzi, M. Gatti,
		\emph{A quantitative study of radial symmetry for solutions to semilinear equations in~$\R^n$},
		preprint, arXiv:2501.11595v1, 2025.
		
		\bibitem{cicopepo}
		G. Ciraolo, M. Cozzi, M. Perugini, L. Pollastro,
		\emph{A quantitative version of the Gidas-Ni-Nirenberg theorem},
		J. Funct. Anal. \textbf{287} (2024), no. 9, Paper No. 110585, 29 pp.
		
		\bibitem{cfm}
		G. Ciraolo, A. Figalli, F. Maggi,
		\emph{A quantitative analysis of metrics on~$\R^n$ with almost constant positive scalar curvature, with applications to fast diffusion flows},
		Int. Math. Res. Not. IMRN \textbf{2018} (2018), no. 21, 6780--6797.
		
		\bibitem{cirli}
		G. Ciraolo, X. Li,
		\emph{A quantitative symmetry result for $p$-Laplace equations with discontinuous nonlinearities},
		preprint, arXiv:2410.09482v1, 2024.
		
		\bibitem{dam-comp}
		L. Damascelli,
		\emph{Comparison theorems for some quasilinear degenerate elliptic operators and applications to symmetry and monotonicity results},
		Ann. Inst. H. Poincar\'e C Anal. Non Lin\'eaire \textbf{15} (1998), no. 4, 493--516.
		
		\bibitem{dm}
		L. Damascelli, S. Merch\'an, L. Montoro, B. Sciunzi,
		\emph{Radial symmetry and applications for a problem involving the $-\Delta_{p}(\cdot)$ operator and critical nonlinearity in $\R^{N}$},
		Adv. Math. \textbf{265} (2014), 313--335.
		
		\bibitem{dam-pacella}
		L. Damascelli, F. Pacella,
		\emph{Monotonicity and symmetry of solutions of {$p$}-Laplace equations, $1<p<2$, via the moving plane method},
		Ann. Sc. Norm. Super. Pisa Cl. Sci. (4) \textbf{26} (1998), no. 4, 689--707.
		
		\bibitem{dam-par}
		L. Damascelli, R. Pardo,
		\emph{A priori estimates for some elliptic equations involving the $p$-Laplacian},
		Nonlinear Anal. Real World Appl. \textbf{41} (2018), 475--496.
		
		\bibitem{dam-ram}
		L. Damascelli, M. Ramaswamy,
		\emph{Symmetry of $C^1$ solutions of $p$-Laplace equations in $\R^{N}$},
		Adv. Nonlinear Stud. \textbf{1} (2001), no. 1, 40--64.
		
		\bibitem{ds-poinc}
		L. Damascelli, B. Sciunzi,
		\emph{Regularity, monotonicity and symmetry of positive solutions of $m$-Laplace equations},
		J. Differential Equations \textbf{206} (2004), no. 2, 483--515.
		
		\bibitem{ds-har}
		L. Damascelli, B. Sciunzi,
		\emph{Harnack inequalities, maximum and comparison principles, and regularity of positive solutions of $m$-Laplace equations},
		Calc. Var. Partial Differential Equations \textbf{25} (2006), no. 2, 139--159.
		
		\bibitem{dsw}
		B. Deng, L. Sun, J.-C. Wei,
		\emph{Sharp quantitative estimates of Struwe's decomposition}, preprint,  arXiv:2103.15360v2, 2021, to appear in Duke Math. J.
		
		\bibitem{diben}
		E. DiBenedetto,
		\emph{$C^{1+\alpha}$ local regularity of weak solutions of degenerate elliptic equations},
		Nonlinear Anal. \textbf{7} (1983), no. 8, 827--850.
		
		\bibitem{dpsv}
		S. Dipierro, J. Gon\c{c}alves da Silva, G. Poggesi, E. Valdinoci,
		\emph{A quantitative Gidas-Ni-Nirenberg-type result for the~$p$-Laplacian via integral identities},
		preprint, arXiv:2408.03522v1, 2024.
		
		\bibitem{fms}
		A. Farina, L. Montoro, B. Sciunzi,
		\emph{Monotonicity of solutions of quasilinear degenerate elliptic equation in half-spaces},
		Math. Ann. \textbf{375} (2013), no. 3, 855--893.
		
		\bibitem{fg}
		A. Figalli, F. Glaudo,
		\emph{On the sharp stability of critical points of the Sobolev inequality},
		Arch. Ration. Mech. Anal. \textbf{237} (2020), no. 1, 201--258.
		
		\bibitem{fink-llev}
		R. Finkelstein, R. LeLevier, M. Ruderman,
		\emph{Nonlinear spinor fields},
		Phys. Rev. \textbf{83} (1951), no. 2, 326--332.
		
		\bibitem{gnn-srp}
		B. Gidas, W.-M. Ni, L. Nirenberg
		\emph{Symmetry and related properties via the maximum principle},
		Comm. Math. Phys. \textbf{68} (1979), no. 3, 209--243.
		
		\bibitem{gnn}
		B. Gidas, W.-M. Ni, L. Nirenberg,
		\emph{Symmetry of positive solutions of nonlinear elliptic equations in~$\R^{n}$}, Math. Anal. Appl. Part A, Advances in Mathematics Supplementary Studies, 7A, 369--402, Academic Press, New York, NY, 1981.
		
		\bibitem{gt}
		D. Gilbarg, N. Trudinger,
		\emph{Elliptic partial differential equations of second order},
		Reprint of the 1998 edition, Classics in Mathematics, Springer-Verlag, Berlin, 2001.
		
		\bibitem{kpac}
		S. Kesavan, F. Pacella,
		\emph{Symmetry of positive solutions of a quasilinear elliptic equation via isoperimetric inequalities},
		Appl. Anal. \textbf{54} (1994), no. 1--2, 27--37.
		
		\bibitem{lieb}
		G. M. Lieberman,
		\emph{Boundary regularity for solutions of degenerate elliptic equations},
		Nonlinear Anal. \textbf{12} (1988), no. 11, 1203--1219.
		
		\bibitem{lions}
		P. L. Lions,
		\emph{Two geometrical properties of solutions of semilinear problems},
		Appl. Anal. \textbf{12} (1981), no. 4, 267--272. 
		
		\bibitem{mss}
		C. Mercuri, B. Sciunzi, M. Squassina,
		\emph{On Coron’s problem for the~$p$-Laplacian},
		J. Math. Anal. Appl. \textbf{421} (2015), no. 1, 362--369. 
		
		\bibitem{mos-har}
		J. Moser,
		\emph{On Harnack's theorem for elliptic differential equations},
		Comm. Pure Appl. Math. \textbf{14} (1961), no. 3, 577--591.
		
		\bibitem{stamp-mu}
		M. K. V. Murthy, G. Stampacchia,
		\emph{Boundary value problems for some degenerate-elliptic operators},
		Ann. Mat. Pura Appl. (4) \textbf{80} (1968), no. 1, 1--122.
		
		\bibitem{nehari}
		Z. Nehari,
		\emph{On a nonlinear differential equation arising in nuclear physics},
		Proc. R. Ir. Acad. A \textbf{62} (1963), no. 2, 117--135.
		
		\bibitem{ps}
		P. Pucci, J. Serrin,
		\emph{The maximum principle},
		Progress in Nonlinear Differential Equations and their Applications, 73, Birkh\"auser Verlag, Basel, 2007.
		
		\bibitem{ross}
		E. Rosset,
		\emph{An approximate Gidas–Ni–Nirenberg theorem},
		Math. Methods Appl. Sci. \textbf{17} (1994), no. 13, 1045--1052.
		
		\bibitem{serra}
		J. Serra,
		\emph{Radial symmetry of solutions to diffusion equations with discontinuous nonlinearities},
		J. Differential Equations \textbf{254} (2013), no. 4, 1893--1902.
		
		\bibitem{serr}
		J. Serrin,
		\emph{Local behavior of solutions of quasi-linear equations},
		Acta Math. \textbf{111} (1964), no. 1, 247--302.
		
		\bibitem{sciu}
		B. Sciunzi,
		\emph{Classification of positive~$\mathcal{D}^{1,p}(\R^{N})$-solutions to the critical~$p$-{L}aplace equation in~$\R^{N}$},
		Adv. Math. \textbf{291} (2016), 12--23.
		
		\bibitem{struwe}
		M. Struwe,
		\emph{A global compactness result for elliptic boundary value problems involving limiting nonlinearities},
		Math. Z. \textbf{187} (1984), no. 4, 511--517.
		
		\bibitem{tal}
		G. Talenti,
		\emph{Best constant in Sobolev inequality},
		Ann. Mat. Pura Appl. (4) \textbf{110} (1976), 353--372.
		
		\bibitem{trud-reg}
		N. S. Trudinger,
		\emph{On the regularity of generalized solutions of linear, non-uniformly elliptic equations},
		Arch. Ration. Mech. Anal. \textbf{42} (1971), no. 2, 50--62.
		
		\bibitem{trud-lin}
		N. S. Trudinger,
		\emph{Linear elliptic operators with measurable coefficients},
		Ann. Sc. Norm. Super. Pisa Cl. Sci. (3) \textbf{27} (1973), no. 2, 265--308.
		
		\bibitem{vet}
		J. V\'etois,
		\emph{A priori estimates and application to the symmetry of solutions for critical $p$-Laplace equations},
		J. Differential Equations \textbf{260} (2016), no. 1, 149--161.
		
	\end{thebibliography}
\end{document}